\documentclass[a4paper,11pt]{article}
\pdfoutput=1

\usepackage[a4paper,tmargin=3truecm,bmargin=3truecm,rmargin=2.5truecm,
lmargin=2.5truecm,twoside,verbose=true]{geometry}
\usepackage{cancel,graphicx}

\usepackage{amsmath,amssymb}
\usepackage[amsmath, hyperref, thmmarks]{ntheorem}
\usepackage[all]{xypic}
\usepackage[pdftex]{hyperref}
\usepackage{enumitem}

\numberwithin{equation}{section}

\allowdisplaybreaks[1]

\newcommand\caA{{\mathcal A}}
\newcommand\caB{{\mathcal B}}
\newcommand\caC{{\mathcal C}}
\newcommand\caD{{\mathcal D}}

\newcommand\caF{{\mathcal F}}

\newcommand\caL{{\mathcal L}}

\newcommand\caN{{\mathcal N}}
\newcommand\caS{{\mathcal S}}
\newcommand\caO{{\mathcal O}}
\newcommand\caU{{\mathcal U}}

\newcommand\gone{{ \mathchoice {1\mskip-4mu\mathrm{l} } {1\mskip-4mu\mathrm{l} }{1\mskip-4.5mu\mathrm{l} } {1\mskip-5mu\mathrm{l}} }}
\newcommand\gR{{\mathbb R}}
\newcommand\gA{\mathbb A}

\newcommand\gC{{\mathbb C}}
\newcommand\gM{{\mathbb M}}
\newcommand\gP{{\mathbb P}}
\newcommand\gB{{\mathbb B}}

\newcommand\gN{{\mathbb N}}
\newcommand\gZ{{\mathbb Z}}

\newcommand\algzero{{\mathsf 0}}
\newcommand\superA{{\mathcal A}}

\newcommand\ehH{\mathcal H}

\newcommand\modE{{E}}				
\newcommand\modF{{F}}				

\newcommand\kb{{\mathfrak b}}

\newcommand\kg{{\mathfrak g}}

\newcommand\kP{{\mathfrak P}}

\newcommand\kh{{\mathfrak h}}

\newcommand\kU{{\mathfrak U}}
\newcommand\kV{{\mathfrak V}}
\newcommand\eps{{\varepsilon}}

\newcommand\Ad{{\text{\textup{Ad}}}}

\newcommand\fois{\mathord{\cdot}}
\DeclareMathOperator{\Hom}{\mathsf{Hom}}

\newcommand\Ker{{\text{\textup{Ker}}}}
\newcommand\Imag{{\text{\textup{Im}}}}

\newcommand\dd{{\text{\textup{d}}}}

\newcommand\Lie{{\text{\textup{Lie}}}}
\newcommand\End{{\text{\textup{End}}}}

\newcommand\norm{\mathord{\parallel}}
\newcommand\defin{\bf}

\newcommand\bzeta{\overline\zeta}
\DeclareMathOperator{\convol}{\ast} 
\DeclareMathOperator{\Str}{Str} 

\newcommand\ev{{\text ev}}
\newcommand\od{{\text od}}
\newcommand\qev{q^\ev}
\newcommand\qod{q^\od}
\newcommand\signp{{k}}
\newcommand\signq{{\ell}}
\newcommand\vv{v}
\newcommand\vw{w}
\newcommand\xx{\mathbf x}
\newcommand\yy{\mathbf y}

\newcommand\BCH{\mathrm{BCH}}   
\newcommand{\D}{D}
\newcommand{\DGR}{\D(G_0) } 
\newcommand{\DgR}{\D(\kg_\gR) } 
\newcommand{\DG}{\D(G) }
\newcommand{\Dg}{\D(\kg) }
\newcommand\ehF{\mathcal{F}}   
\newcommand\rH{\mathrm{H}} 	
\newcommand\Hol{Hol}  
\newcommand\cN{\mathcal{N}} 
\newcommand\pio{\pi_0} 
\newcommand\piH{\pi_*^H} 
\newcommand\piG{\pi_*} 
\newcommand\piwod{\widetilde\pio} 
\newcommand\parity{\sigma} 
\newcommand\wod{\text{w.o.d.}} 

\newcommand\Mp{\mathrm{Mp}} 
\newcommand\SO{\mathrm{SO}} 
\newcommand\Sp{\mathrm{Sp}} 
\newcommand\Spin{\mathrm{Spin}} 
\newcommand\SpO{\mathrm{SpO}} 

\newcommand{\dual}[1]{{\vphantom{#1}\,}^{\flat}{\!#1}}  

\def\a{\alpha}
\def\b{\beta}
\def\g{\gamma}

\def\l{\lambda}

\def\m{\mu}
\def\om{\omega}

\theoremsymbol{}
\theorembodyfont{\slshape}
\theoremheaderfont{\normalfont\bfseries}
\theoremseparator{}
\newtheorem{Theorem}{Theorem}[section]
\newtheorem{theorem}[Theorem]{Theorem}

\newtheorem{proposition}[Theorem]{Proposition}

\newtheorem{lemma}[Theorem]{Lemma}

\newtheorem{corollary}[Theorem]{Corollary}

\theorembodyfont{\upshape}
\theoremsymbol{\ensuremath{\blacklozenge}}

\newtheorem{example}[Theorem]{Example}

\newtheorem{remark}[Theorem]{Remark}

\newtheorem{definition}[Theorem]{Definition}

\theoremstyle{nonumberplain}
\theoremheaderfont{\scshape}
\theorembodyfont{\normalfont}
\theoremsymbol{\ensuremath{\blacksquare}}

\newtheorem{proof}{Proof}
\qedsymbol{\ensuremath{_\blacksquare}}
\theoremclass{LaTeX}

\long\def\sideremark#1{\ifvmode\leavevmode\fi\vadjust{\vbox to0pt{\vss
 \hbox to 0pt{\hskip\hsize\hskip1em
 \vbox{\hsize3cm\tiny\raggedright\pretolerance10000
 \noindent #1\hfill}\hss}\vbox to8pt{\vfil}\vss}}}%

\title{Superunitary representations of Heisenberg Supergroups\footnote{Work
supported by the Belgian Interuniversity Attraction Pole (IAP) within the framework ``Dynamics, Geometry and Statistical Physics'' (DYGEST).}}
\author{Axel de Goursac\footnote{Corresponding author, phone: +32 10473164}, Jean-Philippe Michel}

\begin{document}

\maketitle
\vspace*{-1cm}
\begin{center}
\textit{IRMP, Universit\'e Catholique de Louvain,\\ Chemin du Cyclotron, 2, B-1348 Louvain-la-Neuve, Belgium\\
e-mail: \texttt{Axelmg@melix.net, jpmichel82@gmail.com}}\\
\end{center}%

\vskip 2cm

\begin{abstract}
Numerous Lie supergroups do not admit superunitary representations except the trivial one, e.g.,
 Heisenberg and orthosymplectic supergroups in mixed signature.
To avoid this situation, we introduce in this paper a broader definition of superunitary representation,
relying on a new definition of Hilbert superspace. The latter is inspired by the notion of Krein space and 
was developed initially for noncommutative supergeometry. 

For Heisenberg supergroups, this new approach yields a smooth generalization, whatever the signature, 
of the unitary representation theory of the classical Heisenberg group. 
First, we obtain Schr\"odinger-like representations by quantizing
generic coadjoint orbits. They satisfy the new definition of irreducible superunitary representations and serve as ground to the main result of this paper: a generalized Stone-von Neumann theorem. 
Then, we obtain the superunitary dual and build a group Fourier transformation,
satisfying Parseval theorem. We eventually show that metaplectic representations,
which extend Schr\"odinger-like representations to metaplectic supergroups, 
 also fit into this definition of superunitary representations.
 
 \vskip 0.5cm
 MSC classes: 22E27, 58C50, 43A32
\end{abstract}
%
%

\tableofcontents


\section{Introduction}

After the appearance of supersymmetry in field theories \cite{Neveu:1971,Wess:1974},
Lie superalgebras quickly become a classical topic in algebraic representation theory.
They offer a rich playground generalizing Lie algebras, 
with important results 
and many applications.

Superunitary representations are at the core of supersymmetric field theories \cite{Salam:1974}.
They were first investigated from a purely
algebraic perspective, as star-representations \cite{Scheunert:1977}
or $\varepsilon$-unitary representations \cite{Nishiyama:1990} 
of Lie superalgebras
on a graded $\gC$-vector space $V=V_0\oplus V_1$.
Contrary to unitary representations of Lie groups, the preserved structure is not
a positive definite hermitian form (or scalar product), which is not adapted to 
super-commutation relations \cite{Sternberg:1978}, but rather a
non-degenerate sesquilinear form (or inner product) 
$\langle -,- \rangle :V\times V\to V$
which is superhermitian. That is 
$\overline{\langle x,y\rangle} = (-1)^{|x||y|}\langle y,x\rangle$
if $x,y\in V_0\cup V_1$ are of degree $|x|,|y|\in\{0,1\}$.
\medskip


To develop a full theory of superunitary representations, including analytic aspects,
one needs  a topology on the representation space. But none is canonically associated with a superhermitian
inner product. 
A first way to proceed is to start with a graded  Hilbert space and define 
a compatible superhermitian inner product afterwards. 
This yields to the following definition
of Hilbert superspace, which is the most common one 
\cite{Deligne:1999su} among numerous other possibilities \cite{Rudolph:2000}. 
It is coined as standard in this paper.

\begin{definition}
\label{def-oldversion}
A {\defin standard Hilbert superspace} is a $\gZ_2$-graded complex Hilbert space $\ehH=\ehH_0\oplus\ehH_1$ such that $(\ehH_0,\ehH_1)=\{0\}$, where $(-,-)$ denotes the positive definite  hermitian scalar product. A {\it superhermitian} inner product is then uniquely defined on $\ehH$ by
\begin{align*}
&\langle x,y\rangle:= (x,y)\quad\text{ if } x,y\in\ehH_0\\
&\langle x,y\rangle:= i(x,y)\quad\text{ if } x,y\in\ehH_1,\\
&\langle x,y\rangle:= 0\quad\text{ if } x\in\ehH_0, y\in\ehH_1.
\end{align*}
\end{definition}

Up to our knowledge, the above definition first appeared in 
\cite{ElGradechi:1996} 
and served as ground for all the mathematics literature on
 superunitary representations, see e.g.\ 
\cite{ElGradechi:1996,Carmeli:2006,Salmasian:2010}. 
An algebraic version of this definition, where $\ehH$ is a pre-Hilbert space,
was considered before in representation theory \cite{Furutsu:1988},
following physics conventions.

The odd part of Lie supergroups is intrinsically of infinitesimal
nature. This is best encoded in the description of Lie supergroups
 as super Harish-Chandra pairs, that is pairs $(G_0,\kg_\gR)$ where $G_0$ is a Lie group and $\kg_\gR=(\kg_\gR)_0\oplus(\kg_\gR)_1$  is a Lie superalgebra 
carrying an adjoint-like action of $G_0$, with  $(\kg_\gR)_0=Lie(G_0)$.
As a consequence, representations of Lie supergroups involve in general unbounded
operators. More precisely, a definition of superunitary
 representations should come along with a choice of domain
for the operators representing $(\kg_\gR)_1$. In the seminal paper 
\cite{Carmeli:2006}, this domain was chosen as the space of smooth 
vectors of the underlying $G_0$-representation. The definition 
reads as follows.

\begin{definition}[\cite{Carmeli:2006}]\label{Def:standardSUR}
Let $(G_0,\kg_\gR)$ be a super Harish-Chandra pair.
A {\defin standard super\-unitary representation} (standard SUR) of 
$(G_0,\kg_\gR)$ is a triple $(\ehH,\pio,\piG)$ such that
\begin{itemize}
\item $\ehH$ is a standard Hilbert superspace;
\item $\pio:G_0\to \caU(\ehH)$ is a unitary representation of the Lie group $G_0$ on $\ehH$;
\item $\piG:\kg_\gR\to\End(\ehH^\infty)$ is a $\gR$-Lie superalgebra
 morphism such that $\piG=d\pio$ on $(\kg_\gR)_0$ and
\begin{align*}
\forall g\in G_0,\ \forall X\in(\kg_\gR)_1, \qquad \piG(X)^\dag&=-\piG(X),\\ 
\piG(\Ad_g(X))&=\pio(g)\piG(X)\pio(g)^{-1},
\end{align*}
where $\ehH^\infty$ is the space of smooth vectors  of the representation $\pio$, 
$\Ad:G_0\times\kg_\gR\to\kg_\gR$ is the adjoint-like action of $G_0$ on $\kg_\gR$ and 
${}^\dag$ denotes the adjoint operation w.r.t.\ the superhermitian inner product of $\ehH$.
\end{itemize}
\end{definition}

Due to the very definition of standard Hilbert superspace,
the choice of domain in the above definition
proves to be inessential. As a result, the category of standard SUR's is well-behaved
under tensor products and restrictions to sub-pairs \cite{Carmeli:2006}.
In addition, fundamental results, 
like Schur Lemma and Mackey induction 
\cite{Carmeli:2006},  smoothly extend to standard SUR's. 
However, there is a list of Lie supergroups \cite{Neeb:2011} 
for which no non-trivial standard SUR exists:
\begin{align}
\mathrm{SL}(m|n,\gR)\text{ for }m>2&\text{ or }n>2;\quad \mathrm{SU}(p,q|r,s)\text{ for }p,q,r,s>0;\quad \mathrm{OSp}(\signp,\signq|2m)\text{ for }\signp,\signq,m>0;\nonumber\\
&\  \SpO(2m|\signp,\signq)\text{ for }m,\signp,\signq>0;\quad \rH_{2m|\signp,\signq}\text{ for }\signp,\signq>0,\dots\label{eq-list}
\end{align}
We are particularly interested here in the Heisenberg supergroup  $\rH_{2m|\signp,\signq}$.
It can be described as the super Harish-Chandra pair 
$(\rH_{2m},(\kh_{2m|\signp,\signq})_\gR)$ where $\rH_{2m}$ is the 
classical Heisenberg group of dimension $2m+1$ and 
$(\kh_{2m|\signp,\signq})_\gR$ is the $\gR$-Lie superalgebra
with basis $(p_i,q_i,e_\a,f_\b,Z)$ and non-trivial commutation
relations
$$
[p_i,q_j]=\delta_{ij} Z,\qquad [e_\a,e_{\a'}]=\delta_{\a\a'} Z,
\qquad [f_\b,f_{\b'}]=-\delta_{\b\b'} Z,
$$
where $i,j=1,\ldots,m$, $\a,\a'=1,\ldots,\signp$ 
and $\b,\b'=1,\ldots,\signq$.
If 
$\signp\signq=0$, the standard superunitary representation theory
of $\rH_{2m|\signp,\signq}$ is a smooth extension of
the unitary
 representation theory of $\rH_{2m}$
\cite{Salmasian:2010,Alldridge:2013}, e.g.\ 
the standard superunitary dual of $\rH_{2m|\signp,\signq}$ essentially coincides
with the unitary dual of $\rH_{2m}$.
On the contrary, if $\signp\signq\neq 0$, 
the standard superunitary theory of $\rH_{2m|\signp,\signq}$ is empty,
there is no standard SUR except the trivial one \cite{Salmasian:2010}.

\medskip

In this paper, we introduce a new definition for Hilbert superspaces,
generalizing the one of standard Hilbert superspaces.
They are endowed with an algebraic structure, namely a superhermitian inner product, 
and a compatible Hilbert topology. 

Before digging into details, we would like to comment
on how topology might be induced from algebra.
In the usual non-graded operator algebras setting, topology is fully determined by algebra, e.g.,
\begin{itemize}
\item in Hilbert space, the Hilbert topology is canonically determined by the positive definite hermitian scalar product $\norm x\norm =\sqrt{(x,x)}$;
\item in C*-algebra, the C*-norm is determined by the algebraic notion of spectral radius $\rho$: $\norm a\norm =\sqrt{\rho(a^*a)}$;
\item in von Neumann algebra, the strong (or weak) topology is determined (at least on bounded subsets) by bicommutant's properties.
\end{itemize}
This ideal algebraico-topological framework breaks down in some cases, e.g.\ for Lorentzian noncommutative geometry.
There, one needs an extra ingredient to build the topology out of the algebraic structure. This is the case of {\defin Krein spaces}, defined as pseudo-hermitian inner product spaces $\big(\ehH,\langle-,-\rangle\big)$ endowed with 
an extra involutive $\langle-,-\rangle$-isometric automorphism $J$ of $\ehH$ such that 
 $\big(\ehH, \langle -,J -\rangle\big)$ is a Hilbert space. The topology associated to the positive definite 
scalar product  $\langle -,J(-)\rangle$ is actually independent of the choice of 
$J$ and thus canonical \cite{Bognar:1974}.

Noncommutative Supergeometry \cite{Bieliavsky:2010su,deGoursac:2014kv,deGoursac:2008bd} provides a new framework for graded operator algebras, dealing with Hilbert superspaces (as defined below), C*-superalgebras and von Neumann superalgebras. Their natural superalgebraic structures induce topology, in a similar fashion as for Krein spaces. 
Building on \cite{Bieliavsky:2010su}, we introduce the following definition of Hilbert superspace. 
\begin{definition}
\label{def-newversion}
A {\defin Hilbert superspace} of parity $\sigma\in\gZ_2$ is a complex $\gZ_2$-graded vector space $\ehH=\ehH_0\oplus\ehH_1$ endowed with a homogeneous superhermitian inner product $\langle-,-\rangle$ of degree $\sigma$, such that there exists a fundamental symmetry $J$, i.e., an endomorphism $J$ of $\ehH$ of degree $\sigma$ satisfying, $\forall x,y\in\ehH$,
\begin{itemize}
\item $J^2(x)=(-1)^{(\sigma+1)|x|}x$ and $\langle J(x),J(y)\rangle=\langle x,y\rangle$,
\item $(x, y)_J := \langle x, J (y)\rangle$ is a positive definite hermitian scalar product on $\ehH$ for which $\ehH$ is complete.
\end{itemize}
\end{definition}
Then, Hilbert superspaces admit a canonical Hilbert topology, given by any of the scalar product $(-,-)_J$. Particular cases of Hilbert superspaces include Krein spaces, for which $\ehH_1=\algzero$, and standard
Hilbert superspaces, for which $J=\left(\begin{smallmatrix}
\gone &\, 0 \\ 
0 &\, -i\gone
\end{smallmatrix}\right)$ on $\ehH_0\oplus\ehH_1$.
An important example of non-standard Hilbert superspace is the superspace 
of square integrable functions on $\gR^{m|n}$,
whose fundamental symmetry $J$ is a Hodge-like operator (see Example \ref{ex-superhilbert} below).
\medskip

{\defin Superunitary representations} (SUR) of Lie supergroups can be defined on Hilbert superspaces.
Several Lie supergroups in the list \eqref{eq-list} then enjoy SUR's even if they do not admit standard SUR's.
In particular, we build in this paper non-trivial SUR's for the supergroups 
$\SpO(2m|\signp,\signq)$ and $\rH_{2m|\signp,\signq}$ for any integers $m,\signp,\signq$.

More precisely, a SUR of a super Harish-Chandra pair $(G_0,\kg_\gR)$ is a triple $(\ehH,\pio,\piG)$,
where $\ehH$ is a Hilbert superspace, $\pi_0:G_0\to \caB(\ehH)$ is a continuous representation
such that
\begin{equation*}
\forall g\in G_0,\; v,w\in\ehH,\qquad \langle \pi_0(g)v,\pi_0(g)w\rangle=\langle v,w\rangle,
\end{equation*}
 and $\piG$ satisfies the same axioms as in Definition \ref{Def:standardSUR}.
This generalization of standard SUR's comes along with many difficulties. 
First, it includes all the Krein-unitary representations of Lie groups, as defined
in \cite{Naimark:1968b}, which are wild. For instance, it is not known if Schur's Lemma 
holds or not for these representations. Second, the choice of domain for the operators
$\piG(X)$, with $X\in\kg_1$, is now a crucial part of the definition. For instance, the restriction
of a SUR to a sub-pair $(H_0,\kh_\gR)$ might not exist because an operator $\piG(X)$, 
with $X\in\kh_1$, admits no extension to the space of smooth vectors over $H_0$.
We fully solve the first problem and partially solve the second one by introducing a more
restrictive generalization of standard SUR, called {\defin strong SUR}. 
Namely, 
a strong SUR is a SUR such that the  map $\pio$ is unitarizable and the operators $\piG(X)$, with $X\in\kg_1$,
admit a domain extension to the space of smooth vectors of $\DGR$,
the connected Lie subgroup of $G_0$ with Lie algebra $[(\kg_\gR)_1,(\kg_\gR)_1]$.

For Heisenberg supergroups  $\rH_{2m|\signp,\signq}$, with arbitrary integers $\signp,\signq$,
strong SUR's provide the right framework for a smooth generalization of unitary representation theory
of the classical Heisenberg group $\rH_{2m}$ and its harmonic analysis applications.
In particular, Kirillov Orbit Method allows us to construct the graded version of 
the Schr\"odinger representation, which is a strong SUR for any $(\signp,\signq)$, 
and to prove the Stone-von Neumann theorem in this general setting. 
The full description of the superunitary dual follows as well as the group Fourier transformation
on $\rH_{2m|\signp,\signq}$. 
Building on \cite{Furutsu:1988}, we also define a SUR of the metaplectic supergroup, which is defined as a finite covering of the orthosymplectic supergroup $\SpO(2m|\signp,\signq)$,
extending the classical metaplectic representation of $\Sp(2m)$, and whose representation operators
are intertwiners of the graded Schr\"odinger representation.

\medskip

Let us now detail the content of this paper.

In section 2, after providing basics on supergeometry, we recall the framework of Lie supergroups and Harish-Chandra pairs, 
as well as the equivalence between these two notions. The Heisenberg supergroup is presented as an example.

In section 3, we prove the basic analytic properties of Hilbert superspaces, introduced in  Definition \ref{def-newversion}.
In particular, 
we classify them in terms of a generalized notion of signature. Fundamental examples of Hilbert superspaces,
given by functional spaces, are presented.

We introduce in section $4$ the notions of SUR and strong SUR,  for both equivalent settings: 
Harish-Chandra super pairs and Lie supergroups.
We investigate their basic properties and illustrate the general theory with many examples.

In section 5, we derive the expression of the Schr\"odinger representation for the Heisenberg supergroup $\rH_{2m|\signp,\signq}$ in {\it any signature} $(\signp,\signq)$ by performing Kirillov Orbit Method and we prove that it is an irreducible strong SUR. We then show its relation with the classical Schr\"odinger representation of $\rH_{2m}$ and the spinor representation of $\mathrm{Spin}(\signp,\signq)$.

Next, we prove in section 6 the main result of this paper: the Stone-von Neumann theorem. It states that any strong SUR of the Heisenberg supergroup $\rH_{2m|\signp,\signq}$ with character $\hbar\neq 0$ can be decomposed as the tensor product of the Schr\"odinger representation with the trivial representation on some Hilbert superspace. 

Eventually, we give in section 7 the classification of the strong superunitary dual of $\rH_{2m|\signp,\signq}$ and some applications to harmonic analysis. We construct the Fourier transformation between functions on $\rH_{2m|\signp,\signq}$ and operators on the Hilbert superspace of the Schr\"odinger representation and we prove the Parseval-Plancherel identity. Moreover, we build the metaplectic-type representation naturally appearing in this setting and we show that it is a SUR.

\subsection*{Notations}
Throughout the paper,
$\gN=\{0,1,2,\ldots\}$ denotes the set of non-negative integers, 
$\gN^\times=\{1,2,\ldots\}$ the set of positive ones and $i=\sqrt{-1}$. 

\subsection*{Acknowledgments}
We would like to thank many people for useful discussions and comments,
 including Pierre Bieliavsky, Victor Gayral and Gijs Tuynman.

\section{Preliminaries}\label{sec-prelim}

We recall here some basic notions of supergeometry following the concrete approach,
as fully exposed in the book \cite{Tuynman:2005} and summed up in the appendix of 
\cite{Tuynman:2016} (see also \cite{DeWitt:1984,Berezin:1976,Rogers:2007}).
This approach is equivalent to the sheaf-theoretic one \cite{Kostant:1977,Deligne:1999su}, but differs
in the presentation. It makes an explicit use of a fixed supercommutative algebra $\superA$,
which serves as a replacement for the field $\gR$ of real numbers. Such an algebra
appears in the sheaf-theoretic approach as soon as one uses the functor of points.

\medskip
For the rest of the paper, we choose a unital $\gZ_2$-graded real algebra 
$\superA=\superA_0\oplus\superA_1$ which satisfies the three following properties:  
\begin{itemize}
\item supercommutativity, i.e., $ab=(-1)^{|a||b|}ba$
for all $a,b\in\superA_0\cup\superA_1$ of degree $|a|,|b|\in\{0,1\}$;
\item $\superA/\mathcal{N}_{\superA}\simeq \gR$, where $\mathcal{N}_{\superA}$ 
is the set of all nilpotent elements of $\superA$; 
\item for all $n\in\gN^\times$, there exist $a_1,\ldots,a_n\in\superA_1$ such that
$a_1\cdots a_n \neq 0$.
\end{itemize}
For instance, we can choose $\superA=\bigwedge V$, the 
Grassmann algebra of a real infinite-dimensional vector space $V$. 
The quotient map $\gB:\superA\to\gR$, with kernel $\mathcal{N}_{\superA}$, is called the body map. 
The unit of $\superA$ defines a section of $\gB$ via $r\mapsto r\cdot 1$
so that $\superA=\gR\oplus\mathcal{N}_{\superA}$ and
\begin{equation*}
\forall a\in\superA, \quad a=\gB a+\cN(a),
\end{equation*}
where $\gB a:=\gB(a)\in\gR$ and $\cN(a)\in\cN_\superA$. 
The complexified algebra $\superA_\gC=\superA\otimes\gC = \superA \oplus i\,\superA$
is endowed with the complex conjugation: 
\begin{equation*}
\forall a\in\superA,\ \forall\lambda\in\gC,\quad\overline{a\otimes \lambda}= a\otimes\overline{\lambda}.
\end{equation*}
For all $a,b\in\superA_\gC$, it satisfies:  $\overline{a\, b}=
\overline{a}\,\overline{b}=(-1)^{|a||b|}\,\overline{b}\,\overline{a}$. 

In the sequel, graded will always stand for $\gZ_2$-graded, elements of 
degree $0$ and $1$ are called even and odd respectively, and collectively
referred as homogeneous elements.

\subsection{Superspaces and supermanifolds}\label{subsec-linsuper}

A {\defin graded $\superA$-module} is an $\superA$-module $\modE$  
with decomposition $\modE=\modE_0\oplus\modE_1$ and such that 
$$
\forall \a,\b\in\gZ_2,\quad\superA_\a \modE_\b\subseteq \modE_{\a+\b}.
$$
A graded submodule of $\modE$ is an $\superA$-submodule $\modF$
such that $\modF=(\modF\cap\modE_0)\oplus(\modF\cap\modE_1)$.
A $\superA$-linear map $f:E\to F$ between two graded $\superA$-modules
is of degree $|f|\in\gZ_2$ if $f(E_\a)\subseteq F_{\a+|f|}$ for all $\a\in\gZ_2$.
The space of nilpotent vectors of $\modE$ is defined by
$$
\cN_E:=\{x\in E\;|\; \exists a\in\superA,\ a\neq 0 \;\& \; ax=0\}.
$$
It permits to extend the body map as the quotient map $\gB: E\to E/\cN_E =: \gB E$.
If the graded  $\superA$-module $E$ is free, then 
$\gB E=\gB E_0\oplus\gB E_1$ is a graded $\gR$-vector space,
i.e., $\gB E_0$ and $\gB E_1$ are $\gR$-vector spaces.

A {\defin graded $\superA$-vector space} is a free graded $\superA$-module
$E$ together with a decomposition $E=R_E\oplus\cN_E$
such that $\superA\cdot R_E=E$.
Then we have $R_E\simeq \gB E$ as graded $\gR$-vector spaces and 
$E\simeq \superA\otimes \gB E$ as graded $\superA$-modules,
in particular,
\begin{equation}\label{eq-grrealvect}
E_0\simeq \left(\superA_0\otimes\gB E_0\right)\oplus\left(\superA_1\otimes\gB E_1\right) 
\quad \text{and}\quad
E_1\simeq \left(\superA_0\otimes\gB E_1\right)\oplus\left(\superA_1\otimes\gB E_0\right).
\end{equation}
The graded dimension of $\modE$, given by 
$\dim(\modE)=\dim(\gB E_0)|\dim(\gB E_1)$, is a complete invariant. 
Hence, a graded $\superA$-vector space of dimension $m|n$ is isomorphic to 
$\superA^{m|n}=(\superA^{m|n})_0\oplus(\superA^{m|n})_1$, 
with $(\superA^{m|n})_0=(\superA_0)^m\times(\superA_1)^n$ and
$(\superA^{m|n})_1=(\superA_0)^n\times(\superA_1)^m$.
The associated graded $\gR$-vector space is $R_{\superA^{m|n}}=\gR^m\oplus\gR^n$.
Morphisms between graded $\superA$-vector spaces $E$ and $F$ are even
$\superA$-linear maps $\phi:E\to F$ which are {\defin body-compatible}, 
i.e.\ $\phi(R_E)\subseteq R_F$. They 
are completely determined from
the corresponding $\gR$-linear maps between $\gB E\simeq R_E$ and $\gB F\simeq R_F$.
The categories of graded $\superA$-vector spaces
and of graded $\gR$-vector spaces are equivalent via the inverse
isomorphisms $\gB$ and $\superA\otimes-$.

The {\defin superspace}  of dimension $m|n$ is the even part of the graded $\superA$-vector space 
$\superA^{m|n}$,
$$
\gR^{m|n}:=(\superA^{m|n})_0=(\superA_0)^m\times(\superA_1)^n.
$$ 
The body map $\gB:\gR^{m|n}\to\gR^m$
defines the so-called DeWitt topology, as the coarsest topology on $\gR^{m|n}$ 
such that $\gB$ is continuous. 
\medskip 

Each element $x\in\gR^{m|0}=(\superA_0)^m$ 
decomposes as $x=\gB x+\cN(x)$ with $\gB x\in\gR^m$ and $\cN(x)\in(\mathcal{N}_\superA\cap\superA_0)^m$.
Mimicking Taylor expansions, one can extend any smooth function $f\in\caC^\infty(\gR^m,\gC)$
into a function $\tilde{f}:\gR^{m|0}\to\superA_0\otimes\gC$ defined by 
\begin{equation}\label{Function:Rm0}
\forall x\in\gR^{m|0},\quad \tilde{f}(x)=\sum_{\nu\in\gN^m}\frac{\cN(x)^\nu}{\nu!}\,\partial^\nu f(\gB x).
\end{equation}
Here standard index notation is used, i.e., if $x=(x^1,\ldots,x^m)\in\gR^{m|0}$
and $\nu=(\nu_1,\ldots,\nu_m)\in\gN^m$, we have
$\nu!:=\nu_1!\cdots\nu_m!$, $\partial^\nu:=\partial_1^{\nu_1}\cdots\partial_m^{\nu_m}$
and $\cN(x)^\nu:=\cN(x^1)^{\nu_1}\cdots \cN(x^m)^{\nu_m}$.
Moreover, the sum is finite since the elements $\cN(x^1),\ldots,\cN(x^m)$ are nilpotent. 

The space of {\defin smooth functions} $\caC^\infty(\gR^{m|n})$ is the
space of  maps $f:\gR^{m|n}\to\superA_\gC$ such that 
\begin{equation}\label{Function:Rmn}
\forall(x,\xi)\in\gR^{m|n},\quad f(x,\xi)=\sum_{\a\in(\gZ_2)^n} \widetilde{f_\a}(x)\xi^\a,
\end{equation}
where each $\widetilde{f_\a}$ is obtained from a function $f_\a\in\caC^\infty(\gR^m,\gC)$
via Equation \eqref{Function:Rm0}. We use again standard index notation, i.e.,
if $\xi=(\xi_1,\ldots,\xi_n)\in\gR^{0|n}$ and $\a=(\a_1,\ldots,\a_n)\in(\gZ_2)^n$, 
we have $\xi^\a:=\xi_1^{\a_1}\cdots\xi_n^{\a_n}$. 
By convention, $\xi_1^0=\cdots =\xi_n^0=1$. 
The $\gZ_2$-degree of a function is determined by the rules $|\tilde{f}_\a|=0$
and $|\xi^\a|=|\a| \text{ mod }2$, with $|\a|=\a_1+\cdots+\a_n\in\gN$. 
Hence, the splitting of 
a function in even and odd parts corresponds to the splitting of the target space 
$\superA_\gC=\superA_0\otimes\gC\oplus\superA_1\otimes\gC$. 
As a result, the space of smooth functions on $\gR^{m|n}$ is a 
$\gZ_2$-graded $\gC$-algebra such that 
\begin{equation}
\caC^\infty(\gR^{m|n})\simeq \caC^\infty(\gR^m,\gC)\otimes\bigwedge\gR^n.\label{eq-decompRmn}
\end{equation}
The definition of a smooth function on an open subset of $\gR^{m|n}$ is a 
straightforward adaptation. Note that a smooth function on $\gR^{m|n}$ 
is automatically continuous for the DeWitt topology. 
The definition of smooth functions readily extends from $\superA_\gC$-valued
functions to functions with values in arbitrary superspaces. 

\begin{remark}\label{rmk:smoothlinear}
The morphisms between two finite dimensional 
$\superA$-vector spaces $E$ and $F$ are in bijection with the smooth 
$\gR$-linear maps between the superspaces $E_0$ and $F_0$.
\end{remark}

\medskip
On the superspace $\gR^{m|n}$, the {\defin Berezin integration} of a smooth function 
$f\in \caC^\infty(\gR^{m|n})$, decomposed as in Equation \eqref{Function:Rmn},
is defined as the Lebesgue integration 
on $\gR^m$ of its top component in the $\xi$-variables.
Namely, if $\dd\xi=\dd\xi^n\cdots\dd\xi^1$, we have
\begin{equation}
\int\dd x\dd\xi\ f(x,\xi):=\int \dd x\ f_{1_n}(x) ,\label{eq-berezin}
\end{equation}
where $1_n=(1,\dots,1)\in(\gZ_2)^n$ and, abusing notation, $x$ stands for $\gB x$ 
in the right hand side.
Therefore, the Berezin integral of $f$ is well-defined as soon as
the Lebesgue integral of $f_{1_n}\in\caC^\infty(\gR^{m},\gC)$ is. 
Note that Berezin integration over $\xi$-variables corresponds to a multiple derivation, 
$\int\dd\xi^n\cdots\dd\xi^1:=\partial_{\xi^n}\cdots\partial_{\xi^1}$ in this precise order. 

For $p\geq 1$, we define the {\defin Lebesgue superspace} $L^p(\gR^{m|n})$ 
as the completion of the subspace
\begin{equation*}
\{f\in\caC^\infty(\gR^{m|n})\;|\; \forall \alpha\in(\gZ_2)^n,\ f_\alpha\in L^p(\gR^m) \},
\end{equation*}
with respect to the norm $\norm f\norm:=\sum_{\alpha}(\int_{\gR^m}|f_\alpha|^p)^{\frac1p}$. 
This is a $\gZ_2$-graded Banach space. The identification \eqref{eq-decompRmn} 
extends as $L^p(\gR^{m|n})\simeq L^p(\gR^m)\otimes\bigwedge\gR^n$.

\medskip

We introduce briefly the notion of {\defin supermanifold}. First, a proto-supermanifold  
of dimension $m|n$ is a topological space $M$ endowed 
with an atlas of charts valued in open subsets of $\gR^{m|n}$ (for DeWitt topology)
with smooth transition functions. The body map on $\gR^{m|n}$ allows to
 define a canonical projection $M\to\gB M$ on a topological space $\gB M$, 
called the body of $M$. If $\gB M$ is Hausdorff and second-countable, 
then the induced atlas turns $\gB M$  into a smooth manifold of dimension $m$ 
and $M$ is simply called a supermanifold. Of course, a superspace
is a supermanifold. The inclusions 
$\gR^m\subseteq\gR^{m|0}=\superA_0\otimes\gR^m\subseteq\gR^{m|n}$
translate into canonical smooth embeddings $\gB M\subseteq M^{\wod}\subseteq M$,
where $M^{\wod}$ is the sub-supermanifold of $M$
obtained from $\gB M$ by extension of scalars from $\gR$ to $\superA_0$.
Here, $\wod$ stands for "without odd dimensions".
Note that $M^\wod$ is of dimension $m|0$ and $\gB M^\wod=\gB M$.

\subsection{Lie supergroups and super Harish-Chandra pairs}

A {\defin Lie superalgebra} is a graded $\superA$-vector space 
 $\kg=\kg_0\oplus \kg_1$ endowed with a body-compatible $\superA$-bilinear map
$[\cdot,\cdot]:\kg\times\kg\to\kg$ which is even, i.e. 
$[\kg_\a, \kg_\b]\subseteq \kg_{\a+\b}$ for all $\a,\b\in\gZ_2$,
and satisfies the two following properties
\begin{itemize}
\item skew-symmetry: $\forall X,Y\in \kg_0\cup \kg_1$, $[X,Y]=-(-1)^{|X||Y|}[Y,X]$;
\item Jacobi identity: $\forall X,Y,Z\in \kg_0\cup \kg_1$, $[X,[Y,Z]]=[[X,Y],Z]+(-1)^{|X||Y|}[Y,[X,Z]]$.
\end{itemize}
Note that $\superA$-bilinearity means, in particular, that 
$[a X,b Y]=(-1)^{|X||b|}ab[X,Y]$ for all homogeneous elements
$X,Y\in\kg$ and $a,b\in\superA$.
As the bracket $[\cdot,\cdot]$ is body-compatible, it induces a bilinear 
operation on the body space $\gB \kg$. The later defines 
a $\gR$-Lie superalgebra  structure on $\gB \kg$.
The equivalence of categories between $\superA$-vector spaces 
and $\gR$-vector spaces specializes straightforwardly to the Lie setting
with inverse isomorphisms given again by $\gB$ and $\superA\otimes-$.

A {\defin Lie supergroup} is a supermanifold $G$ endowed with a group structure 
for which the multiplication is a smooth map. If $G$ is a Lie supergroup 
then $\gB G$ turns into a Lie group and $G^\wod$ into a Lie supergroup.
Moreover, there is a  canonically associated  Lie superalgebra $\kg=\Lie(G)$, 
and an exponential map $e:U\to G$, with $U\subseteq \kg_0$ an open neighborhood of $0$.
 As a consequence, one can prove that $G$  trivializes over $G^\wod$.

\begin{proposition}[\cite{Tuynman:2005}, Proposition VI.1.7]
\label{prop-diffeom}
Let $G$ be a Lie supergroup, $\gB \kg$ the body of its Lie superalgebra, 
and $\kg^{(1)}:=\superA_1\otimes\gB \kg_1\subseteq\kg_0$. Then, the following map
\begin{equation*}
G^\wod\times\kg^{(1)}\to G,\quad (g,X)\mapsto g \, e^X
\end{equation*}
is a global diffeomorphism satisfying $\gB( g \,  e^X)=\gB g$. 
\end{proposition}

A {\defin super Harish-Chandra pair} $(G_0,\kg_\gR)$ is the data of 
a Lie group $G_0$ and a compatible $\gR$-Lie superalgebra $\kg_\gR$.
Namely, $\kg_\gR$ is a $G_0$-module, its even part $(\kg_\gR)_0$ is the Lie algebra of  $G_0$
and the differential of the action $G_0\times\kg_{\gR}\to\kg_{\gR}$ 
coincides with the adjoint action of $(\kg_\gR)_0$  on $\kg_\gR$.
\medskip

There is an equivalence of categories between Lie supergroups $G$ and
super Harish-Chandra pairs \cite{Rogers:2007,Tuynman:2016} 
(see also \cite{Koszul:1982, Deligne:1999su}). 
In the sequel, we use that a Lie supergroup $G$ induces a super Harish-Chandra 
pair $(G_0,\kg_\gR)$ given by $G_0:=\gB G$, $\kg_\gR:=\gB \kg$ 
(with $\kg=\Lie(G)$), and where the action of $G_0$ on $\kg_{\gR}$ 
is the body of the adjoint action of $G$ on $\kg$.
For the inverse construction of a Lie supergroup out of a super Harish-Chandra
pair, we refer to Remark \ref{Rmk:SupergrouptoSHC} in Appendix.

\subsection{Heisenberg supergroups}

Let $\modE$ be a graded $\superA$-vector space of 
dimension $m'|n'$. A {\defin symplectic form} on $\modE$ 
is an even body-compatible $\superA$-bilinear map 
$\omega:\modE\times\modE\to\superA$ which is
\begin{itemize}
\item skewsymmetric: $\forall X,Y\in\modE$, $\omega(X,Y)=-(-1)^{|X||Y|}\omega(Y,X)$,
\item non-degenerate: ($\forall Y\in\modE,\ \omega(X,Y)=0$)$\Rightarrow$ $X=0$.
\end{itemize}
As a consequence, $m'=2m$ is even, $\omega$ induces 
a symplectic form on $\gB\modE_0$ and 
a quadratic form of given signature $(\signp,\signq)$ on $\gB\modE_1$, with $\signp+\signq=n'$.
The triple $(2m|\signp,\signq)$ is a complete invariant of the symplectic 
space $(\modE,\omega)$. In the following, we will always use the normal 
form of $(\modE,\omega)$ described below.
Let $\varepsilon\in\{-1,+1\}$ and $r,s'\in\gN$ be such that
\begin{equation}\label{def:rs}
\varepsilon s'=\signp-\signq \quad \& \quad 2r+s'=\signp+\signq,
\end{equation}
i.e., $\varepsilon=1$ and $(\signp,\signq)=(r+s',r)$, or,  $\varepsilon=-1$ and $(\signp,\signq)=(r,r+s')$.
Then, setting $W=\superA^{m|r}$ and $V=\superA^{0|s'}$, we have
\begin{equation}\label{decompo:E}
E \simeq(W\oplus W^*)\oplus V \quad\& \quad E_0\simeq\gR^{m|r}\oplus(\gR^{m|r})^*\oplus\gR^{0|s'},
\end{equation}
where $W^*$ stands for the dual of $W$ as $\superA$-module.
Under the above isomorphism, the symplectic form $\om$ restricts as the canonical 
symplectic form on $W\oplus W^*$ and as a definite quadratic form 
on $V$, namely
\begin{equation}\label{decompo:om}
\begin{aligned}
\forall q_1,q_2\in W,\ \forall p_1,p_2\in W^*,\quad &\om((q_1,p_1),(q_2,p_2))=q_1p_2-q_2p_1,\\
\forall \xi_1,\xi_2\in V,\quad &\om(\xi_1,\xi_2)=\eps\xi_1\xi_2,
\end{aligned}
\end{equation}
where $q_1p_2$ and $q_2p_1$ are the scalars obtained by the duality 
pairing between $W$ and $W^*$ and $\xi_1\xi_2$ is the scalar obtained
by the canonical scalar product of two elements in $\superA^{s'}$. 

\medskip

The {\defin Heisenberg Lie superalgebra} $\kh(\modE,\omega)$ is defined 
as the $\superA$-module $\modE\oplus \superA Z$, where $Z$ is an even 
generator, endowed with the Lie bracket
\begin{equation}\label{eq-heisliealg}
[X+aZ,X'+a'Z]=\omega(X,X')Z,
\end{equation}
for any $X,X'\in \modE$ and $a,a'\in\superA$. Its center is $Z(\kh(E,\om))=\superA Z$.
The {\defin Heisenberg Lie supergroup} 
$\rH(\modE_0,\omega)$ is obtained by exponentiation of the even part 
of $\kh(\modE,\omega)$.
Namely, $\rH(\modE_0,\omega)$ is the supermanifold $\modE_0\times \gR^{1|0}$ 
endowed with the group law 
\begin{equation}\label{GroupLaw}
\big(\xx,t\big)\big(\xx',t'\big)=(\xx+\xx',t+t'+\tfrac12\omega(\xx,\xx')),
\end{equation}
for any $\xx,\xx'\in \modE_0$ and $t,t'\in\gR^{1|0}$.
Its neutral element is $(0,0)$ and the inverse is given by $(\xx,t)^{-1}=(-\xx,-t)$.
Its center is $Z(\rH(E,\om))=\gR^{1|0}$.
The Heisenberg Lie superalgebra and supergroup are uniquely determined 
(up to isomorphism) by the dimension $2m|n'$ of $\modE$ and the signature $(\signp,\signq)$ 
of the quadratic form induced by $\omega$ on $ \gB\modE_1$. 
In the following, we use  their canonical versions
\begin{equation}\label{heis:normal}
\kh_{2m|\signp,\signq}:=\kh(\modE,\omega) \quad\& \quad \rH_{2m|\signp,\signq}:=\rH(\modE_0,\omega),
\end{equation}
where $(\modE,\om)$ and $(\modE_0,\om)$ are  in their normal form, as defined by
\eqref{decompo:E} and \eqref{decompo:om}. 
%

Let $\kh_{2m|\signp,\signq}^\ast$ be the dual of 
the Heisenberg Lie superalgebra. The symplectic form $\om$
induces a linear isomorphism, $\dual{}:\kh_{2m|\signp,\signq}\to\kh_{2m|\signp,\signq}^\ast$, 
given by:
\begin{equation*}
\forall X,X'\in\modE,\ \forall a,a'\in\superA,\quad \dual{(X+aZ)}(X'+a'Z)=\omega(X,X')+aa'.
\end{equation*}
Then, one can compute the adjoint and coadjoint actions of the Heisenberg Lie supergroup:
\begin{align}\nonumber
\forall (\xx,t)\in\modE_0\times\gR^{1|0},
\quad
& \Ad_{(\xx,t)}(X+aZ)=X+(a+\omega(\xx,X))Z,\\ \label{coad-action}
& \Ad^\ast_{(\xx,t)}\dual{(X+aZ)}=\dual{(X-a\xx+aZ)}.
\end{align}
The coadjoint orbit $\caO_\mu=\{\Ad^\ast_g\mu,\,g\in \rH_{2m|\signp,\signq}\}$ 
passing through $\mu=\dual{(X+aZ)}$
corresponds to the affine space $a\dual{Z}+\modE_0$,
provided $a\in\gR^\times$ and $X\in E_0$. If $a= 0$, 
then the coadjoint orbit $\caO_\mu$ is a single point.

The super Harish-Chandra pair $(G_0,\kg_\gR)$ corresponding to 
the Heisenberg Lie supergroup $\rH_{2m|\signp,\signq}$ is defined as follows. The associated 
Lie group $G_0:=\gB \rH_{2m|\signp,\signq}$ is the classical Heisenberg group $\rH_{2m}$, i.e., 
the manifold $\gR^{2m}\times\gR$ endowed with the group law
\begin{equation*}
\forall (x,t),(x',t')\in\gR^{2m}\times\gR,\quad
\big(x,t\big)\big(x',t'\big)=(x+x',t+t'+\tfrac12\omega_0(x,x')),
\end{equation*}
with $\omega_0$ the canonical symplectic form on $\gR^{2m}$.
The $\gR$-Lie superalgebra $\kg_\gR:=\gB\kh_{2m|\signp,\signq}$ is the
$\gZ_2$-graded real vector space with components 
$(\kg_\gR)_0=\gR^{2m}\times\gR Z$ and $(\kg_\gR)_1=\gR^{\signp+\signq}$,
endowed with the Lie bracket 
\begin{equation}\label{LieHeisenberg}
\forall X,X'\in\gR^{2m}\oplus\gR^{\signp+\signq},\forall a,a'\in\gR,\quad
[X+aZ,X'+a'Z]=\omega(X,X')Z,
\end{equation}
where $\om=\om_0\oplus\om_1$ and $\om_1$ is the canonical quadratic form
of signature $(\signp,\signq)$ on $\gR^{\signp+\signq}$. The action of $G_0$ on $\kg_\gR$ is 
the adjoint action 
$$
\Ad_{(x,t)}(X+aZ)=X+(a+\omega(x,X))Z, 
$$
with $(x,t)\in G_0$, and $X+aZ\in\kg_\gR$. 
A graded basis of $\kg_\gR$ is given by $(c_i,e_\alpha, Z)$, where $(c_i)_{1\leq i\leq 2m}$ 
is a basis of $\gR^{2m}$ with degree $0$, $(e_\alpha)_{1\leq \alpha\leq \signp+\signq}$ 
is a basis of $\gR^{\signp+\signq}$ with degree 1, and $\gR Z$ is the center subalgebra (with degree $0$).
The non-zero commutators read then as
\begin{equation}
[c_i,c_j]=\omega_{ij}Z,\qquad [e_\alpha,e_\beta]=\omega_{\alpha\beta}Z,\qquad [c_i,e_\alpha]=0,\label{eq-baseheis}
\end{equation}
where $\omega_{ij}$ and  $\omega_{\alpha\beta}$ are the components of $\om$ in the basis.
Explicitly,  for all $ i,j=1,\ldots, m$, we have $\omega_{ij}=\delta_{j, m+i}-\delta_{j, i-m}$, and 
for all $\beta=1,\ldots,\signp+\signq$, we have $\omega_{\alpha\beta}=\delta_{\alpha\beta}$ 
if $1\leq\alpha\leq \signp$ and $\omega_{\alpha\beta}=-\delta_{\alpha\beta}$ if $\signp<\alpha\leq \signp+\signq$.

\section{Theory of Hilbert superspaces}

Building on \cite{Bieliavsky:2010su}, we introduce a new definition of Hilbert superspaces, 
which generalizes the standard one \cite{ElGradechi:1996,Deligne:1999su,Carmeli:2006,Salmasian:2010}. 
Their basic properties are derived by using the theory of Krein spaces \cite{Bognar:1974}.

\subsection{Definition of Hilbert superspaces} 

A  {\defin Hermitian superspace} is 
a complex $\gZ_2$-graded vector space $\ehH=\ehH_0\oplus\ehH_1$
endowed with a superhermitian inner product $\langle-,-\rangle$, i.e.,
a non-degenerate sesquilinear\footnote{By convention, sesquilinear means
left-antilinear and right-linear.} map
$\langle-,-\rangle:\ehH\times\ehH\to\gC$ which satisfies: 
\begin{equation}\label{Eq:SuperHerm}
\forall x,y\in\ehH_0\cup\ehH_1,\quad \overline{\langle x,y\rangle}=(-1)^{|x||y|}\langle y,x\rangle.
\end{equation}
If the  inner product is a homogeneous map of degree $\sigma(\ehH)\in\gZ_2$,
then $\ehH$ is called a Hermitian superspace of {\defin parity} $\parity(\ehH)$.
If $\parity(\ehH)=0$ the subspaces $\ehH_0$ and $\ehH_1$ are orthogonal,
i.e.\ $\langle \ehH_0,\ehH_1\rangle=\algzero$, and by Equation \eqref{Eq:SuperHerm} we get 
$\langle x,x\rangle\in i^{|x|}\gR$ if $x\in\ehH$ is homogeneous.
If $\parity(\ehH)=1$, the subspaces $\ehH_0$ and $\ehH_1$ are isotropic,
i.e.\ $\langle \ehH_0,\ehH_0\rangle=\langle \ehH_1,\ehH_1\rangle=\algzero$,
and by Equation \eqref{Eq:SuperHerm} we get $\langle x,x\rangle\in\gR$ for all $x\in\ehH$.
Via the formula
\begin{equation*}
\forall x,y\in\ehH,\quad (x,y)_J:=\langle x,J(y)\rangle,
\end{equation*}
a linear isomorphism $J$ of $\ehH$ gives rise to a non-degenerate
sesquilinear map $(-,-)_J:\ehH\times\ehH\to\gC$.

\begin{definition}
Consider a Hermitian superspace  $(\ehH,\langle-,-\rangle)$ of given parity.
\begin{itemize}
\item A {\defin fundamental decomposition} is an orthogonal decomposition
\begin{equation*}
\ehH=\ehH_{[1]}\oplus\ehH_{[i]}\oplus\ehH_{[-1]}\oplus\ehH_{[-i]}
\end{equation*}
such that: $\forall k\in\{0,\dots,3\},\ \forall x\in\ehH_{[i^k]}\setminus\algzero, 
\quad \langle x,x\rangle\in i^k\gR^\ast_+$.
\item A {\defin fundamental symmetry} is an endomorphism $J$  of $\ehH$ such that 
$J^4=\gone$, $\langle J(x),J(y)\rangle=\langle x,y\rangle$ for all $x,y\in\ehH$, 
and $(-,-)_J$ defines a scalar product on $\ehH$
(i.e. $(-,-)_J$ is a positive definite  hermitian map).
\end{itemize}
\end{definition}

A fundamental decomposition allows one to define a fundamental symmetry by setting
\begin{equation*}
J:=P_{[1]}-iP_{[i]}-P_{[-1]}+iP_{[-i]},
\end{equation*}
where $P_{[i^k]}$ is the orthogonal projector on the subspace $\ehH_{[i^k]}$, 
with $k=0,\ldots,3$. Conversely, a fundamental symmetry $J$ defines 
a  fundamental decomposition given by  $\ehH_{[i^k]}=\text{Ker}(J-i^{-k}\gone)$.
Note that $(-,-)_J$ provides a scalar product on $\ehH_{[i^k]}$
equal to $i^{-k}\langle-,-\rangle$.

\begin{definition}\label{Def:SuperHilbert}
A  Hermitian superspace $(\ehH,\langle-,-\rangle)$ of given parity 
is said to be a {\defin Hilbert superspace} if one of the two equivalent 
conditions hold:
\begin{itemize}
\item there exists a fundamental decomposition such that each component 
$\ehH_{[i^k]}$ is intrinsically complete, i.e. complete for the topology defined 
by $i^{-k}\langle-,-\rangle$;
\item there exists a fundamental symmetry $J$ such that $(\ehH,(-,-)_J)$ is a Hilbert space.
\end{itemize}
\end{definition}

The above definition of Hilbert superspace is slightly more general 
that the one provided in \cite{Bieliavsky:2010su}, 
where a fundamental symmetry $J$ is fixed as part of the structure.
This freedom in the choice of $J$ will be important in the following sections.


A {\defin standard Hilbert superspace} is a Hilbert superspace  $\ehH$
of parity $0$ which admits the canonical fundamental decomposition $\ehH_{[1]}=\ehH_0$, 
$\ehH_{[i]}=\ehH_1$ and $\ehH_{[-1]}=\ehH_{[-i]}=\algzero$.
It admits a unique fundamental symmetry given by $J(x)=(-i)^{|x|}x$
on homogeneous elements $x\in\ehH$. Moreover, both $(\ehH_0,\langle-,-\rangle)$ 
and $(\ehH_1,i\langle-,-\rangle)$ are Hilbert spaces.
The notion of standard Hilbert superspace coincides with the usual  
definition of Hilbert superspace, as given in \cite{ElGradechi:1996,Deligne:1999su,Carmeli:2006,Salmasian:2010}.
In other words, the Definition \ref{Def:SuperHilbert} of Hilbert superspace 
is a generalization of the definition introduced in the above references.

The classical notion underlying Definition \ref{Def:SuperHilbert}
is the concept of Krein space \cite{Bognar:1974}.
In our terms, a Krein space is a purely even Hilbert superspace $\ehH$ of parity $0$,
i.e., $\ehH=\ehH_0$.
We have the following obvious proposition.

\begin{proposition}
\label{prop-krein}
Let $\ehH=\ehH_0\oplus\ehH_1$ be a complex $\gZ_2$-graded vector space.
\begin{itemize}
\item $(\ehH,\langle-,-\rangle)$ is a Hilbert superspace of parity 0 if and only if 
$(\ehH_0,\langle-,-\rangle)$ and $(\ehH_1,i\langle-,-\rangle)$ are Krein spaces.
For any fundamental decomposition of $\ehH$, we have 
$\ehH_0=\ehH_{[1]}\oplus\ehH_{[-1]}$ and $\ehH_1=\ehH_{[i]}\oplus\ehH_{[-i]}$.
Fundamental symmetries of $\ehH$, $\ehH_0$ and
$\ehH_1$ are related by the equality $J=J_0\oplus (- i J_1)$, in obvious notations.

\item $(\ehH,\langle-,-\rangle)$ is a Hilbert superspace of parity 1 if and only if
$(\ehH,\langle-,-\rangle)$ is a Krein space admitting $\ehH_0$ and  $\ehH_1$ as isotropic subspaces.
For any fundamental decomposition  of $\ehH$, we have $\ehH_{[i]}=\ehH_{[-i]}=\algzero$. 
\end{itemize}
\end{proposition}

\subsection{Properties of Hilbert superspaces}

A number of important results on Hilbert superspaces can be obtained
as direct consequence of the theory of Krein spaces, by using that a Hilbert
superspace is a Krein space or the direct sum of two Krein spaces (see 
Proposition \ref{prop-krein}). When this is the case, we do not provide a proof 
and simply refer to the corresponding result for Krein spaces.
\medskip

We start with topological properties.

\begin{theorem}[\cite{Bognar:1974}, Theorem V.1.1]
\label{thm-topo}
Let $\ehH$ be a Hilbert superspace. 
The scalar products $(-,-)_J$, with $J$ a fundamental symmetry 
of $\ehH$, all define the same topology.
\end{theorem}

As a consequence, while a Hilbert superspace $\ehH$ has no canonical scalar product,
it admits a canonical Hilbert topology. In particular, all fundamental decompositions of $\ehH$ 
have intrinsically complete components and all fundamental symmetries
$J$ yield unitarily equivalent Hilbert spaces  $(\ehH,(-,-)_J)$.

The following lemma states that the superhermitian inner product 
of a Hilbert superspace is globally continuous for its Hilbert topology.

\begin{lemma}
\label{lem-ineq}
Let $\ehH$ be a Hermitian superspace of given parity.
For any fundamental symmetry $J$, we have
\begin{equation*}
\forall x,y\in\ehH, \quad |\langle x,y\rangle|\leq \norm x\norm_J\norm y\norm_J,
\end{equation*}
where $\norm x\norm_J:=\sqrt{(x,x)_J}$.
\end{lemma}

\begin{proof}
Let $x\in\ehH$. Using the fundamental decomposition of $\ehH$ associated to $J$, 
we have $x=\sum_{k=0}^3x_k$ where $x_k\in\ehH_{[i^k]}$. This yields
\begin{equation*}
|\langle x,y\rangle|= |\sum_k \langle x_k,y_k\rangle|\leq \sum_k |\langle x_k,y_k\rangle|.
\end{equation*}
Since $|\langle x_k,y_k\rangle|=|i^{-k}\langle x_k,y_k\rangle|=|(x_k,y_k)_J|$, 
we can apply the Cauchy-Schwartz inequality for $(-,-)_J$ to obtain:
\begin{equation*}
|\langle x,y\rangle|\leq  \sum_k \norm x_k\norm_J\norm y_k\norm_J. 
\end{equation*}
Using the inequality
$(\norm x_k\norm_J\norm y_l\norm_J- \norm x_l\norm_J\norm y_k\norm_J)^2\geq 0$,
with $k,l=0,\ldots,3$, we deduce that
\begin{equation*}
|\langle x,y\rangle|\leq \Big(\sum_{k,l} \norm x_k\norm_J^2\norm y_l\norm_J^2 \Big)^{\frac 12}
=\norm x\norm_J\norm y\norm_J.
\end{equation*}
\end{proof}
\medskip

We pursue with the study of linear operators.
Let $(\ehH^{(1)},\langle-,-\rangle_1)$ and $(\ehH^{(2)},\langle-,-\rangle_2)$ 
be two Hilbert superspaces. A {\defin bounded operator} between 
$\ehH^{(1)}$ and $\ehH^{(2)}$ is a linear operator 
$T:\ehH^{(1)}\to\ehH^{(2)}$ which is  continuous w.r.t.\
their Hilbert topologies. 
The degree of operators defines 
a $\gZ_2$-grading on the space of bounded operators 
$\caB(\ehH^{(1)},\ehH^{(2)})$ and on the algebra
$\caB(\ehH):=\caB(\ehH,\ehH)$.

\begin{proposition}[\cite{Bieliavsky:2010su}]
\label{prop-superhilbert-adjoint}
Let $T\in\caB(\ehH^{(1)},\ehH^{(2)})$ be of degree  $|T|\in\gZ_2$. 
There exists a unique operator $T^\dag\in\caB(\ehH^{(2)},\ehH^{(1)})$ 
such that: 
\begin{equation*}
\forall x\in\ehH^{(2)},\ \forall y\in\ehH^{(1)},\quad\langle T^\dag(x),y\rangle_1=(-1)^{|T||x|}\langle x,T(y)\rangle_2.
\end{equation*}
\end{proposition}

The operator $T^\dag$ is called the  {\defin superadjoint} operator of $T$.

\begin{definition}[\cite{Bieliavsky:2010su}]
\label{def-superunit}
A {\defin superunitary operator} between $\ehH^{(1)}$ and $\ehH^{(2)}$ 
is a homogeneous operator $\Phi\in\caB(\ehH^{(1)},\ehH^{(2)})$ of degree 0 satisfying: 
\begin{equation*}
\Phi^\dag\Phi=\Phi\Phi^\dag=\gone.
\end{equation*}
\end{definition}

The set of superunitary operators is denoted by $\caU(\ehH^{(1)},\ehH^{(2)})$.
In particular, $\caU(\ehH):=\caU(\ehH,\ehH)$ is  a group. Any superunitary operator
$\Phi\in\caU(\ehH^{(1)},\ehH^{(2)})$ is in particular a linear bjection of degree $0$ and an isometry, that is
\begin{equation}\label{Eq:Isometry}
\langle\Phi(x),\Phi(y)\rangle_2=\langle x,y\rangle_1,
\end{equation}
for all $x,y\in\ehH^{(1)}$.
The converse also holds.

\begin{proposition}\label{prop:isosuperunitary}
Let 
$\Phi:\ehH^{(1)}\to\ehH^{(2)}$ be a linear bijection of degree $0$.
The operator $\Phi$ is superunitary if and only if 
this is an isometry.
\end{proposition}

\begin{proof}
Let $\Phi:\ehH^{(1)}\to\ehH^{(2)}$ be a linear bijection 
of degree $0$. First, assume that $
\langle\Phi(x),\Phi(y)\rangle_2=\langle x,y\rangle_1$
for all $x,y\in\ehH^{(1)}$. If $J_1$ is a fundamental symmetry 
of $\ehH^{(1)}$, then $\Phi\circ J_1\circ\Phi^{-1}$  is a 
fundamental symmetry of $\ehH^{(2)}$.  This implies that $\Phi$ 
and $\Phi^{-1}$ are continuous. Then, one easily proves that 
$\Phi\in\caU(\ehH^{(1)},\ehH^{(2)})$. 
The converse implication is straightforward.
\end{proof}

On a Hilbert superspace $\ehH$, the {\defin parity operator} 
$\gP\in\caB(\ehH)$ is defined by 
\begin{equation}
\forall x_0\in\ehH_0,\ \forall x_1\in\ehH_1,\quad \gP(x_0+x_1)=x_0-x_1.\label{eq-parity}
\end{equation}
This operator is of degree 0, involutive and satisfies
$\gP^\dag=(-1)^\sigma\gP$ if $\ehH$ is of parity $\sigma$. 

\begin{proposition}\label{Prop:J}
Let $J$  be a fundamental symmetry of $\ehH$.
Then, $J$ is an operator of degree $\sigma$,
unitary w.r.t.\ the scalar product $(-,-)_J$
and satisfies $J^2=\gP^{(\sigma+1)}$, $J^\dag=J\gP$.
Moreover, $J$ is superunitary if $\sigma=0$ and $J$ provides
a linear isomorphism between $\ehH_0$ and $\ehH_1$ if $\sigma=1$.
\end{proposition}
\begin{proof}
All stated  properties are obvious, we only prove that
$J^2=\gP^{(\sigma+1)}$. Consider the fundamental decomposition
associated to $J$, so that $\ehH_{[i^k]}$ is the eigenspace  of $J$
with eigenvalue $i^{-k}$. Then, $J^2$ is equal to $\gone$
on $\ehH_{[1]}\oplus\ehH_{[-1]}$ and to $-\gone$ on
$\ehH_{[i]}\oplus\ehH_{[-i]}$. The result follows then from
Proposition~\ref{prop-krein}.
 \end{proof}

Two Hilbert superspaces $\ehH^{(1)}$ and $\ehH^{(2)}$ are 
superunitarily equivalent, i.e. isomorphic, if there exists an 
operator $\Phi\in\caU(\ehH^{(1)},\ehH^{(2)})$. 
Our next goal is to provide a complete invariant for isomorphism classes 
of Hilbert superspaces. Recall that, up to unitary equivalence, 
a Hilbert space is uniquely determined
by its dimension, defined as the cardinal of a Hilbertian basis.

\begin{proposition}
Let $\ehH$ be a Hilbert superspace. The dimensions (as Hilbert spaces) 
of the components of a fundamental decomposition 
$\ehH=\ehH_{[1]}\oplus\ehH_{[i]}\oplus\ehH_{[-1]}\oplus\ehH_{[-i]}$
do not depend on the choice of this decomposition.
\end{proposition}

The proof follows easily from \ref{Eq:Isometry}.
We define the {\defin signature} of $\ehH$ as 
$\text{sgn}(\ehH)=\big(\dim(\ehH_{[1]}),\dim(\ehH_{[i]}),\dim(\ehH_{[-1]}),\dim(\ehH_{[-i]})\big)$.

\begin{theorem}[\cite{Bognar:1974}, Theorem V.1.4]\label{Thm:IsoHilbert}
Two Hilbert superspaces are superunitarily equivalent
if and only if their parities and their signature coincide.
\end{theorem}

According to Proposition \ref{prop-krein} and the theory of Krein spaces,
there exist Hilbert superspaces of parity $0$ and of arbitrary signature.
For instance, there exists four inequivalent classes of Hilbert superspaces
of total dimension $1$, characterized by the value $\langle e,e\rangle\in\{1,i,-1,-i\}$ 
for a normalized generator $e$.
As proved below, there does not exist Hilbert superspaces of parity $1$ with arbitrary signature.
Indeed, they are uniquely determined (up to isomorphism) by their total dimensions. 

\begin{proposition}\label{SgnParity1}
The signature of  a Hilbert superspace $\ehH$ of parity $1$
is given by $\text{sgn}(\ehH)=\big(d,0,d,0\big)$,
where $d=\dim(\ehH_{[1]})=\dim(\ehH_{[-1]})$. Moreover, the parity operator
$\gP$ gives an isomorphism between $\ehH_{[1]}$ and $\ehH_{[-1]}$.
\end{proposition}
\begin{proof}
Consider a fundamental symmetry $J$ and the 
corresponding fundamental decomposition, given by the 
eigenspaces of $J$.
According to Proposition \ref{prop-krein}, we have 
$\dim(\ehH_{[i]})=\dim(\ehH_{[-i]})=0$.
By Proposition \ref{Prop:J}, we know that $J^2=\gone$
and $J\ehH_0=\ehH_1$. We deduce that 
$\ehH_{[1]}=\Ker (\gone-J)= \Imag(\gone+J)=(\gone+J)\ehH_0$, and similarly
 $\ehH_{[-1]}=(1-J)\ehH_0$. Therefore,
$\gP$ gives an isomorphism between $\ehH_{[1]}$ and $\ehH_{[-1]}$.
\end{proof}
\medskip

Finally, we provide ways of constructing new Hilbert superspaces 
out of old ones. Given a Hilbert superspace $(\ehH,\langle-,-\rangle)$, a Hilbert 
sub-superspace $\caF$ is a closed subspace of $\ehH$ which is graded,
i.e. $\caF=\caF\cap\ehH_0\oplus\caF\cap\ehH_1$, and such that 
$(\caF,\langle-,-\rangle|_{\caF\times\caF})$ is a Hilbert superspace.
It is characterized below using the orthogonal of $\caF$ w.r.t.\ the inner product $\langle-,-\rangle$,
defined by $\caF^\perp:=\{x\in\ehH\; | \; \forall y \in\ehH,\;\langle x,y\rangle=0\}$.

\begin{theorem}[\cite{Bognar:1974}, Theorem V.3.4]
\label{thm-subspace}
A graded subspace $\caF\leq \ehH$ is a Hilbert sub-superspace
 if and only if $\caF\oplus \caF^\perp=\ehH$.
\end{theorem}


\begin{corollary}
\label{cor-subspace}
Let $\ehH$ be a Hilbert superspace and $P\in\caB(\ehH)$  be 
an orthogonal projector, i.e. $P^2=P=P^\dag$. Then $\Imag(P)$ 
is a Hilbert sub-superspace of $\ehH$.
\end{corollary}

If $\caF$ is a Hilbert sub-superspace of $\ehH$, the natural embedding is a bounded isometry.
All bounded isometries arise in this way.

\begin{proposition}
\label{prop:ImIso}
Let 
$\Phi:\ehH^{(1)}\to\ehH^{(2)}$ be a bounded isometry.
The operator $\Phi$ satisfies $\Imag\Phi\oplus(\Imag\Phi)^\perp=\ehH^{(2)}$.
Both $\Imag\Phi$ and $(\Imag\Phi)^\perp$ are Hilbert sub-superspaces of $\ehH^{(2)}$
and $\Imag\Phi$ is superunitarily equivalent to $\ehH^{(1)}$.
\end{proposition}

\begin{proof}
Assume $\Phi:\ehH^{(1)}\to\ehH^{(2)}$ is a bounded isometry.
By a direct adaptation of \cite[Theorem VI.3.8]{Bognar:1974}, 
the operator $\Phi$ satisfies $\Imag\Phi+(\Imag\Phi)^\perp=\ehH^{(2)}$.
If $u\in\Imag\Phi\cap(\Imag\Phi)^\perp$, then we have $u=\Phi(x)$ and, for all $y\in\ehH^{(1)}$,
$0=\langle u,\Phi(y)\rangle_2=\langle \Phi(x),\Phi(y)\rangle_2=\langle x,y\rangle_1$. 
Hence, $x=0$, $u=0$ and we deduce that 
$\Imag\Phi\oplus(\Imag\Phi)^\perp=\ehH$. Thanks to Theorem \ref{thm-subspace},
this implies that both $\Imag\Phi$ and $(\Imag\Phi)^\perp$ are Hilbert sub-superspaces of $\ehH^{(2)}$.
Since $\Phi$ is an isometry, this is a linear bijection from $\ehH^{(1)}$ to $\Imag\Phi$. 
By Poposition \ref{prop:isosuperunitary}, both are superunitarily equivalent.
\end{proof}

The direct sum of two Hilbert superspaces 
is easily obtained.

\begin{proposition}[\cite{Bieliavsky:2010su}]
\label{prop-superhilbert-sumprod}
Let $(\ehH^{(1)},\langle-,-\rangle_1)$ and $(\ehH^{(2)},\langle-,-\rangle_2)$ 
be two Hilbert superspaces of same parity $\sigma$. 
The direct sum $\ehH=\ehH^{(1)}\oplus\ehH^{(2)}$ 
endowed with the superhermitian inner product 
\begin{equation*}
\forall x_1,y_1\in\ehH^{(1)}, \ \forall x_2,y_2\in\ehH^{(2)},\quad
\langle x_1\oplus x_2,y_1\oplus y_2\rangle=\langle x_1,y_1\rangle_1+\langle x_2,y_2\rangle_2,
\end{equation*}
is a Hilbert superspace of parity $\sigma$. 
Moreover, the direct sum $J=J^{(1)}\oplus J^{(2)}$ 
of any two fundamental symmetries of $\ehH^{(1)}$
and $\ehH^{(2)}$ gives a fundamental symmetry of $\ehH$.
The component-wise sum of two fundamental decompositions
of $\ehH^{(1)}$ and $\ehH^{(2)}$ gives a fundamental 
decomposition of $\ehH$.
\end{proposition}

Up to our knowledge, the construction of the {\defin tensor product}
has not been developed in the Krein setting. We provide it
for Hilbert superspaces.
Let $(\ehH^{(1)},\langle-,-\rangle_1)$ and $(\ehH^{(2)},\langle-,-\rangle_2)$ 
be two Hilbert superspaces of parity $\sigma_1,\sigma_2\in\gZ_2$.
There is a canonical superhermitian inner product of parity $\sigma_1+\sigma_2$
on the algebraic tensor product $\ehH^{(1)}\otimes\ehH^{(2)}$:
\begin{equation}\label{Eq:TPsuperH}
\forall x_1,y_1\in\ehH^{(1)},\ \forall x_2,y_2\in\ehH^{(2)},\quad
\langle x_1\otimes x_2,y_1\otimes y_2\rangle:=(-1)^{\sigma_1\sigma_2+|x_2||y_1|}\langle x_1,y_1\rangle_1\langle x_2,y_2\rangle_2.
\end{equation}
For any two fundamental symmetries  $J_1$ and $J_2$ 
on $\ehH^{(1)}$ and $\ehH^{(2)}$,
we get a scalar product on $\ehH^{(1)}\otimes\ehH^{(2)}$:
\begin{equation}\label{Eq:scalarproduct12}
(x_1\otimes x_2,y_1\otimes y_2)_{J}:=(x_1,y_1)_{J_1} \, (x_2,y_2)_{J_2},
\end{equation}
which satisfies $(-,-)_{J}=\langle-,J-\rangle$, where $J$ is the operator
\begin{equation}\label{Eq:J}
J(x_1\otimes x_2):=(-1)^{\sigma_1\sigma_2+(\sigma_1+|x_1|)|x_2|}J_1(x_1)\otimes J_2(x_2).
\end{equation}

\begin{proposition}
\label{prop-superhilbert-tensprod}
In same notations as above, we have:
\begin{itemize}
\item all scalar products $(-,-)_J$ as in Equation \eqref{Eq:scalarproduct12}
induce the same pre-Hilbert topology on $\ehH^{(1)}\otimes \ehH^{(2)}$;
\item the inner product $\langle-,-\rangle$ extends continuously to 
the completed tensor product $\ehH^{(1)}\hat{\otimes} \ehH^{(2)}$;
\item $(\ehH^{(1)}\hat{\otimes} \ehH^{(2)},\langle-,-\rangle)$ is 
a Hilbert superspace of parity $\sigma=\sigma_1+\sigma_2$,
and the operators $J$ as in Equation \eqref{Eq:J} are fundamental symmetries for it.
\end{itemize}
\end{proposition}

\begin{proof}
Let $J_1$, $J_{1'}$ be two fundamental symmetries of $\ehH^{(1)}$
and $J_2$, $J_{2'}$ two of $\ehH^{(2)}$. Denote by $(-,-)_{J_{12}}$,
$(-,-)_{J_{1'2}}$, $(-,-)_{J_{12'}}$ and $(-,-)_{J_{1'2'}}$
the four corresponding scalar products on $\ehH^{(1)}\otimes\ehH^{(2)}$. 
Accordingly, we write $\norm\cdot\norm_{J_{12}}$, etc, for the corresponding norms.
We shall prove that they are all equivalent norms.
Any $x\in\ehH^{(1)}\otimes\ehH^{(2)}$ admits a decomposition as follows
$$
x=\sum_{j=1}^k x^{(1)}_j\otimes x^{(2)}_j,
$$
where $k\in\gN$, $(x^{(1)}_j)$ is a family of elements in $\ehH^{(1)}$ and
$(x^{(2)}_j)$ is a family of orthonormal elements in the Hilbert space $(\ehH^{(2)},(-,-)_{J_2})$. 
Then a direct computation shows that
$$
\norm x\norm^2_{J_{12}}=\sum_{j=1}^k \norm x^{(1)}_j\norm^2_{J_1}\cdot\norm x^{(2)}_j\norm^2_{J_2}.
$$
By Theorem \ref{thm-topo}, there exists $c>0$ such that 
$\norm\cdot\norm_{J_1}\leq c\norm\cdot\norm_{J_{1'}}$.
Therefore, we get
$$
\norm x\norm_{J_{12}}\leq c \norm x\norm_{J_{1'2}}.
$$
Repeating the same argument, we obtain that all the four above norms are equivalent.
As a consequence, $\ehH^{(1)}\otimes \ehH^{(2)}$ inherits of a canonical pre-Hilbert 
topology and admits a completion $\ehH^{(1)}\hat{\otimes} \ehH^{(2)}$.

By Lemma \ref{lem-ineq}, $\langle-,-\rangle$
is globally continuous on $\ehH^{(1)}\otimes \ehH^{(2)}$.
As $J$ is unitary w.r.t.\ the scalar product $(-,-)_J$, this is a bounded
operator on $\ehH^{(1)}\otimes \ehH^{(2)}$. Therefore, $\langle-,-\rangle$
and $J$ extend to the completion $\ehH^{(1)}\hat{\otimes} \ehH^{(2)}$ 
and turn it into a Hilbert superspace.
\end{proof}

The two Hilbert superspaces $\ehH^{(1)}\hat{\otimes} \ehH^{(2)}$ 
and $\ehH^{(2)}\hat{\otimes} \ehH^{(1)}$ are isomorphic via
the superunitary operator given by
\begin{equation*}
\forall x_1\in\ehH^{(1)}, \ \forall x_2\in\ehH^{(2)},\quad
x_1\otimes x_2 \mapsto (-1)^{|x_1|(|x_2|+\sigma_2)}x_2\otimes x_1.
\end{equation*}  
We describe the fundamental decomposition of $\ehH^{(1)}\hat{\otimes} \ehH^{(2)}$ below, 
using the operator $\gA=\gone\oplus i\gone$ on $\ehH^{(j)}=\ehH^{(j)}_0\oplus\ehH^{(j)}_1$,
with $j=1,2$. Note that $\gA^2=\gP=\gone\oplus -\gone$  is the parity operator.

\begin{proposition}
The Hilbert superspace $\ehH:=\ehH^{(1)}\hat{\otimes} \ehH^{(2)}$ 
inherits of the following fundamental decomposition, where 
$k,k_1,k_2\in\gZ_4=\{0,1,2,3\}$ and sums are taken modulo $4$, 
\begin{itemize}
\item if $\sigma_1=\sigma_2=0$, 
\begin{equation*}
\ehH_{[i^k]}=\bigoplus_{k_1-(-1)^kk_2=k}\ehH^{(1)}_{[i^{k_1}]}\hat{\otimes} \ehH^{(2)}_{[i^{k_2}]};
\end{equation*}
\item if $\sigma_1=1$ and $\sigma_2=0$, 
\begin{equation*}
\ehH_{[i^k]}=\begin{cases} 
\bigoplus\limits_{k_1-k_2=k} \gA^{k_1}\ehH^{(1)}_{[1]}\hat{\otimes} \ehH^{(2)}_{[i^{k_2}]},  &
\quad\text{if }k\in\{0,2\},\\
0, &\quad\text{if }k\in\{1,3\};
 \end{cases}
\end{equation*}
\item if $\sigma_1=\sigma_2=1$, denoting by $\overline{\mathrm{span}}$ 
the completed linear span of a subset in $\ehH$,
\begin{equation*}
\ehH_{[i^k]}=\overline{\mathrm{span}}\{ x\otimes \gA^{k}y+(-1)^{k}\gP x\otimes\gP \gA^{k}y\ | \ x\in \ehH^{(1)}_{[1]}, \ y\in\ehH^{(2)}_{[1]}\}.
\end{equation*}
\end{itemize}
\end{proposition}

\begin{proof}
For each of the three values of $(\sigma_1,\sigma_2)$, the equations above define
subspaces $\ehH_{[i^k]}$. In the three cases, direct computations using Equation \eqref{Eq:TPsuperH} show that $\langle x,x\rangle\in i^k\gR^\times_+$ 
for all $x\in \ehH_{[i^k]}\setminus\algzero$.
It remains to prove that $\ehH=\bigoplus_k\ehH_{[i^k]}$.
We proceed by a case analysis on the values of $(\sigma_1,\sigma_2)$.

If $\sigma_1=\sigma_2=0$,  the result follows from 
the fundamental decompositions of $\ehH^{(1)}$ and $\ehH^{(2)}$. 

If $\sigma_1=1$ and $\sigma_2=0$, we obtain the announced 
decomposition of $\ehH$ thanks to the fundamental decomposition of
$\ehH^{(2)}$ and to the decompositions
\begin{equation*}
\ehH^{(1)}=\ehH^{(1)}_{[1]}\oplus \gA^2\ehH^{(1)}_{[1]}\quad
\text{and}\quad
\ehH^{(1)}=\gA\ehH^{(1)}_{[1]}\oplus \gA^3\ehH^{(1)}_{[1]}.
\end{equation*}
The latter are consequences of the proof of Proposition \ref{SgnParity1} and
of the equality $\gA^2=\gP$. 

The last case, $\sigma_1=\sigma_2=1$, is more involved.
Let $x\in \ehH^{(1)}_{[1]}$, $y\in\ehH^{(2)}_{[1]}$.
By the proof of Proposition~\ref{SgnParity1}, for $j=1,2$ we have 
$\ehH^{(j)}_{[1]}=(1+J_j)\ehH^{(j)}_0$, i.e., $x=x_0+J_1x_0$
and $y=y_0+J_2y_0$ with $x_0\in \ehH^{(1)}_{0}$, $y_0\in\ehH^{(2)}_{0}$.
Hence, the generators of the subspace $\ehH_{[i^k]}$ 
(as defined in the proposition statement) read as
\begin{equation*}
(1+(-1)^k)(x_0\otimes y_0+i^kJ_1x_0\otimes J_2 y_0)
+(1-(-1)^k)(J_1x_0\otimes y_0+i^kx_0\otimes J_2 y_0).
\end{equation*}
This means that 
\begin{equation*}
\ehH_{[\pm 1]}=\left(\gone\otimes\gone\pm J_1\otimes J_2\right)
\left(\ehH^{(1)}\hat\otimes\ehH^{(2)}\right)
\quad\text{and}\quad
\ehH_{[\pm i]}=\left(J_1\otimes\gone\pm i\gone\otimes J_2\right)
\left(\ehH^{(1)}\hat\otimes\ehH^{(2)}\right).
\end{equation*}
The equalities $J_j\ehH^{(j)}_0=\ehH^{(j)}_1$ for $j=1,2$ imply then the
following identities:
\begin{align*}
\ehH_{[1]}\oplus\ehH_{[-1]}&=\ehH^{(1)}_0\hat\otimes\ehH^{(2)}_0
\oplus\ehH^{(1)}_1\hat\otimes\ehH^{(2)}_1=\ehH_0,\\
\ehH_{[i]}\oplus\ehH_{[-i]}&=\ehH^{(1)}_1\hat\otimes\ehH^{(2)}_0
\oplus\ehH^{(1)}_0\hat\otimes\ehH^{(2)}_1=\ehH_1.
\end{align*}
The decomposition $\ehH=\bigoplus_k\ehH_{[i^k]}$ follows.
\end{proof}

The signature of a tensor product can be easily obtained from the above 
proposition. In particular, if both $\ehH^{(1)}$ and $\ehH^{(2)}$ are of
parity $1$, then the four components of $\ehH$ have the same dimension,
equal to the one of $\ehH^{(1)}_{[1]}\hat\otimes\ehH^{(2)}_{[1]}$.

The dual $\ehH^*$ of a Hilbert superspace $\ehH$ is 
defined as the space of all continuous linear maps $\ehH\to\gC$. 
As proved below, $\ehH^*$ is canonically a Hilbert superspace 
and there is a canonical Riesz anti-isomorphism between $\ehH$ and $\ehH^*$. 

\begin{proposition}
The map $x\in\ehH\mapsto {}^\flat x\in\ehH^*$, defined by 
${}^\flat x(y):={\langle x,y\rangle}$, 
is a homogeneous antilinear bijection. Moreover, endowing $\ehH^*$ 
with the superhermitian inner product
\begin{equation*}
\langle {}^\flat x,{}^\flat y\rangle_{\ehH^*}:=\langle x,y\rangle,
\end{equation*}
the space $\ehH^*$ becomes a Hilbert superspace anti-isomorphic to $\ehH$.
\end{proposition}

\begin{proof}
Let $J$ be a fundamental symmetry of $\ehH$.
By Lemma \ref{lem-ineq}, the antilinear form ${}^\flat x$ is 
continuous so that the map $R:x\mapsto {}^\flat x$ is 
valued in $\ehH^*$. Clearly, $R$ is antilinear, even and injective.
By the Riesz theorem, for any $\varphi\in\ehH^*$,  there 
exists $x\in\ehH$ such that 
$\varphi(y)=(x,y)_J=\langle J^\dag x,y\rangle$. 
Therefore, $\phi={}^\flat(J^\dag x)$ and $R$ is bijective.

We denote by $\sigma$ the parity of $\ehH$. The inner 
product $\langle -,-\rangle_{\ehH^*}$ defined on $\ehH^*$ is of same degree 
that the one on $\ehH$ and non-degenerate. The 
following computation shows that it is also 
superhermitian: for any $x,y\in\ehH$ with $|x|+|y|=\sigma$, 
we have $(-1)^{|{}^\flat x||{}^\flat y|}=
(-1)^{(\sigma+|x|)(\sigma+|y|)}=(-1)^{|x||y|}(-1)^{\sigma(\sigma+|x|+|y|)}=(-1)^{|x||y|}$.
Finally, any fundamental symmetry $J$ of $\ehH$ induces 
a fundamental symmetry of $\ehH^*$ given by 
$\tilde J=R\circ J\circ R^{-1}$. Therefore, $\ehH^*$ is a 
Hilbert superspace anti-isomorphic to $\ehH$ via the superunitary
operator $R$.
\end{proof}

\subsection{Useful examples}

We present three classes of examples of Hilbert superspaces.
They are based on finite dimensional graded spaces and functional spaces
of square integrable functions on real and complex superspaces.

\begin{example}[\cite{Bognar:1974}, page 100]
Any finite-dimensional complex $\gZ_2$-graded space $\ehH$ 
endowed with a homogeneous superhermitian non-degenerate 
inner product $\langle-,-\rangle$ is a Hilbert superspace.
\end{example}


\begin{example}[\cite{Bieliavsky:2010su}]
\label{ex-superhilbert}
We consider the superspace $\caD(\gR^{m|n})$ of smooth functions 
with compact support on $\gR^{m|n}$. 
Its elements read as $\varphi(x,\xi)=\sum_{\a\in(\gZ_2)^n}\widetilde{\varphi_\a}(x)\xi^\a$
with $\varphi_\a\in\caD(\gR^m)$ (see \eqref{Function:Rmn}).
Berezin integration \eqref{eq-berezin} induces the following superhermitian inner product on $\caD(\gR^{m|n})$
\begin{equation}
\langle \varphi,\psi\rangle=\int\dd x\dd\xi\ \overline{\varphi(x,\xi)}\psi(x,\xi)=\sum_{\a\in(\gZ_2)^n}\eps(\a)\int\dd x\ \overline{\varphi_\a(x)}\psi_{\bar{\a}}(x),\label{eq-superman-scalprod}
\end{equation}
where $\bar{\a}=(1,\ldots,1)-\a\in(\gZ_2)^n$ and $\eps(\a)=\pm 1$
is determined by the relation $\xi^\a\xi^{\bar{\a}}=\eps(\a)\xi^1\cdots \xi^n$.
The operator $J$, defined as:
\begin{equation*}
J(\varphi_\a(x)\xi^\a)=\eps(\a)\varphi_\a(x)\xi^{\bar{\a}},
\end{equation*}
is a fundamental symmetry. The completion 
of $\caD(\gR^{m|n})$, w.r.t.\ the scalar product 
$\langle-,J-\rangle$ is a Hilbert superspace of parity $n$ mod $2$, 
which coincides with the Lebesgue superspace
$L^2(\gR^{m|n})\simeq L^2(\gR^m,\gC)\otimes\bigwedge\gR^n$ 
introduced after \eqref{eq-berezin}.

If $m=0$, the above construction provides a Hilbert 
superspace structure on $\ehH:=\bigwedge \gC^n$.
The $\gZ_2$-grading is the standard one,
$\ehH_0=\bigoplus_k \bigwedge^{2k}\gC^n$
and $\ehH_1=\bigoplus_k \bigwedge^{2k+1}\gC^n$,
and $J$ coincides with the Hodge dual.
If $n$ is odd, $\ehH$ is of parity $1$ and signature 
$(2^{n-1},0,2^{n-1},0)$. 
If $n\geq 2$ is even, $\ehH$ is of parity $0$ and the four eigenspaces 
of $J$ are of same dimensions. They
correspond to spaces of differential forms over a point 
which are even  or odd, self-dual or anti self-dual. Hence, 
$\ehH$ is of signature $(2^{n-2},2^{n-2},2^{n-2},2^{n-2})$.
If $n=0$, then $\ehH=\gC$ is of parity $0$ and signature $(1,0,0,0)$.
\end{example}


\begin{example}\label{ex-holfct}
Consider coordinates $(q,p,\xi,\eta)\in\gR^{2m|2n}$,
with $q,p\in\gR^{m|0}$
and $\xi,\eta\in\gR^{0|n}$.
The complex superspace $\gC^{m|n}$ admits holomorphic coordinates $(z,\zeta)$
given by $z=q+ip$ and $\zeta=\xi+i\eta$.
The canonical K\"ahler form on $\gC^{m|n}$ reads as
$\omega=\dd p \dd q+\frac{1}{2}\left(\dd\xi^2+\dd\eta^2\right)
=\frac{\dd z \dd\bar{z}}{2i}+\frac{\dd\zeta \dd\bar{\zeta}}{4}$.
It derives from a K\"ahler potential: 
$\omega=\frac{i}{2}\partial\bar{\partial}\left(|z|^2+\frac{i}{2}\bar{\zeta}\zeta\right)$.
By definition, holomorphic functions are smooth functions $\varphi\in\caC^\infty(\gR^{2m|2n})$ such that $\partial_{\bar{z}}\varphi=0=\partial_{\bar{\zeta}}\varphi$. 
Hence, the algebra 
of holomorphic functions satisfies $\Hol(\gC^{0|n})=\bigwedge\gC^n$
and $\Hol(\gC^{m|n})\simeq\Hol(\gC^m)\otimes\bigwedge\gC^n$.

The {\defin Segal-Bargmann} (or Bargmann-Fock) {\defin superspace} is
$$
\ehH:=\{\varphi\in\Hol(\gC^{m|n}) \;| \; |\langle \varphi,\varphi\rangle|<+\infty\},
$$
where $\langle -,-\rangle$ is the superhermitian inner product given by 
$$
\langle \varphi,\psi\rangle=\frac{(2i)^n}{\pi^m}\int \dd z \dd\bar{z}\dd\zeta\dd\bar{\zeta} \
\overline{\varphi(z,\zeta)}\psi(z,\zeta) e^{-\left(|z|^2+\frac{i}{2}\bar{\zeta}\zeta\right)},
$$
with
$\dd\zeta\dd\bzeta=\dd\zeta^n\dd\bzeta^n\cdots\dd\zeta^1\dd\bzeta^1$, 
i.e.\
$\int\dd\zeta\dd\bar{\zeta}\ \left(\bar{\zeta}^1\zeta^1\cdots \bar{\zeta}^n\zeta^n\right)=1$
(see \eqref{eq-berezin} and below).

The Hermitian superspace $(\ehH,\langle -,-\rangle)$ is a standard Hilbert superspace,
as justified below.
Expanding holomorphic functions as 
$\varphi(z,\zeta)=\sum_{\a\in(\gZ_2)^n} \widetilde{\varphi_\a}(z) \zeta^\a$,
 with $\varphi_\a\in Hol(\gC^m)$, we get that
$$
\langle \varphi,\psi\rangle=
\frac{1}{\pi^m}\sum_{\a\in(\gZ_2)^n}(-1)^{\frac{|\a|(|\a|-1)}{2}}
\left(2i\right)^{|\a|}\int \dd z \dd\bar{z}\ \overline{\varphi_{\a}(z)}\psi_{\a}(z)e^{-|z|^2},
$$
where $|\a|=\a_1+\cdots+\a_n\in\gN$ is the length of the multi-index $\a$.
Setting
\begin{equation*}
J\varphi(z,\zeta)=\sum_{\a\in(\gZ_2)^n}(-1)^{\frac{|\a|(|\a|-1)}{2}}
\left(-i\right)^{|\a|}\varphi_\a(z)\zeta^\a,
\end{equation*}
we obtain a scalar product $(-,-)_J$ on $\ehH$.
Moreover, the operator $J$ satisfies $J\varphi=\varphi$ if $\varphi$ is even
and $J\varphi=i\varphi$ if $\varphi$ is odd. 
Therefore, the expressions 
$\sqrt{|\langle \varphi,\varphi\rangle|}$ and $\sqrt{\langle \varphi,J\varphi\rangle}$ 
define equivalent norms and $\ehH$ is complete w.r.t. the
scalar product $(-,-)_J$. We conclude that $(\ehH,\langle -,-\rangle)$ is a Hilbert superspace,
with fundamental symmetry $J$. Since $J$ is equal to $\gone$ on even functions 
and $i\gone$ on odd functions, $(\ehH,\langle -,-\rangle)$ is a standard Hilbert superspace.

If $m=0$, the above construction turns $\ehH=\bigwedge \gC^n$ 
into a Hilbert superspace of parity $0$ and signature $(2^{n-1},2^{n-1},0,0)$, 
whose $\gZ_2$-grading is the standard one.
The normalization has been chosen such that $\Hol(\gC^{0|n})$ 
can be embedded isometrically into $L^2(\gR^{0|2n})$,
via the map 
$\varphi(\zeta)\mapsto \tilde\varphi(\xi,\eta)=e^{\frac{i}{4}\zeta\bzeta}\varphi(\zeta)$, 
with $\zeta=\xi+i\eta$. However, $\Hol(\gC^{0|n})$ is not 
preserved by the Hodge dual in $L^2(\gR^{0|2n})$.

In the sequel, we will have to deal with a slightly modified 
inner product on $\Hol(\gC^{0|n})$, namely
\begin{equation}\label{Hol-scalprod}
\langle \varphi,\psi\rangle=(2i)^n\int\dd\zeta\dd\bzeta\ \overline{\varphi(\zeta)}\psi(\zeta) e^{-\frac{i\hbar\eps}{2}\bzeta\zeta},
\end{equation}
where $\hbar\in\gR^\times$ and $\eps=\pm 1$. Then, each summand of
the inner product is modified by the factor $(\hbar\eps)^{n-|\a|}$ and, 
depending on the sign of $\hbar\eps$, 
there are two Hilbert superspaces $\Hol_\pm(\gC^{0|n})$.
Obviously, $\Hol_+(\gC^{0|n})$ (i.e. $h\eps>0$) is the same Hilbert superspace as above 
whereas $\Hol_-(\gC^{0|n})$ has signature 
$(2^{n-1},0,0,2^{n-1})$ if $n$ is even and $(0,2^{n-1},2^{n-1},0)$ if $n$ is odd.
\end{example}

\section{Superunitary representations}

In this section, we introduce the notion of superunitary representation 
on Hilbert superspaces (SUR for short). We provide independently a 
definition of SUR for super Harish-Chandra pairs and for Lie supergroups
and prove that they are equivalent.
Our definition of SUR generalizes the one in \cite{Carmeli:2006,Salmasian:2010}.

\subsection{SUR of Lie groups}

A {\defin superunitary representation} (SUR) of a Lie group $G_0$ on a Hilbert 
superspace $\ehH$ (see Definition \ref{Def:SuperHilbert}) is defined as a group morphism $\pio:G_0\to\caU(\ehH)$ such 
that, for all $\vv\in\ehH$, the maps $\pio^\vv:g\mapsto\pio(g)\vv$ 
are continuous on $G_0$. Since superunitary operators on $\ehH$ are bounded
for the Hilbert topology of $\ehH$, SUR's fall into the class of strongly continuous 
representations, whose theory is well-developed \cite{Warner:1972}.
In particular, if $\ehH$ is a Hilbert space, then SUR's $(\ehH,\pio)$
are unitary representations (see e.g. \cite{Warner:1972}),
while if $\ehH$ is a Krein space they are Krein-unitary 
representations as defined in \cite{Naimark:1968b}. 

Consider $(\ehH^{(i)},\pio^{(i)})$ with $i=1,2$, two SUR's of the Lie group $G_0$.
{\defin Intertwiners}, or intertwining operators, between these two SUR's
are bounded operators $T\in\caB(\ehH^{(1)},\ehH^{(2)})$ which 
intertwine the representation morphisms: 
\begin{equation}\label{Intertwine:pio}
\forall g\in G_0,\quad T\circ\pio^{(1)}(g)=\pio^{(2)}(g)\circ T.
\end{equation}
The morphism space $\Hom(\pio^{(1)},\pio^{(2)})$ between two SUR's is the 
space of even intertwiners.
Two SUR's are said to be equivalent if they are related by 
a superunitary intertwiner $T\in\caU(\ehH^{(1)},\ehH^{(2)})$ (see Definition \ref{def-superunit}). 
A SUR is {\defin unitarizable} if it is equivalent to a unitary representation.

\begin{remark}
A superunitary representation $(\ehH,\pio)$ of a Lie group $G_0$
is unitarizable if and only if there exists a 
fundamental symmetry $J$ of $\ehH$ which is preserved by the representation operators,
i.e., $\pio(g)\circ J = J\circ\pio(g)$ for all $g\in G_0$. The representation
is then unitary w.r.t.\ the scalar product $(-,-)_J$.
This is automatically the case if $\ehH$ is a standard Hilbert superspace.
\end{remark}

\begin{example}\label{Ex:nonunitarizable}
Finite dimensional irreducible representations of non-compact simple Lie groups are never 
unitarizable (except the trivial representation), however there might be SUR's. 
For instance, if $k,l$ are positive integers, this is the case of the standard representation of the pseudo-unitary group 
$\mathrm{SU}(k,l)$ on $\gC^{k+l}$ and of the representation of the spin groups 
$\mathrm{Spin}(k,l)$ on the complex spinor module, see Section \ref{Sec:meta}.
\end{example}

As mentioned above, the theory of strongly continuous 
representations applies to SUR's. Results in \cite[pages 252-253]{Warner:1972}
particularize as follows.
For any SUR $(\ehH,\pio)$ of $G_0$,
the space of {\defin smooth vectors},
$$
\ehH^\infty:=\{\vv\in\ehH\;|\;
\pio^\vv\in\caC^\infty(G_0,\ehH)\},
$$ 
is dense in~$\ehH$ and stable under the action 
of the operators $\pio(g)$ for all $g\in G_0$. Examples of smooth vectors are
given by 
\begin{equation*}
\pio(f)\vv:=\int_{G_0}\dd g\ f(g)\pio(g)\vv \ \in\ehH^\infty,
\end{equation*}
where $\vv\in\ehH$ and $f\in\caC^\infty(G_0)$ has compact support.
If $X\in \Lie(G_0)$ and $\vv\in\ehH^\infty$, then the function
$t\mapsto \pio(e^{tX})\vv\in\ehH^\infty$ is smooth in a neighborhood 
of zero and its derivative, $d\pio(X)\vv:=\frac{d}{dt}\big|_{t=0}\left(\pio(e^{tX})\vv\right)$,
defines a Lie algebra morphism $d\pio:\Lie(G_0)\to\End(\ehH^\infty)$.
It extends canonically to the universal enveloping algebra $\mathfrak{U}(\Lie(G_0))$ 
as an algebra morphism. Both $d\pio$ and its extension to $\mathfrak{U}(\Lie(G_0))$
are referred to as the {\defin infinitesimal representation} associated to $(\ehH,\pio)$.
The relative topology induced by the inclusion $\ehH^\infty\subseteq\caC^\infty(G_0,\ehH)$
turns $\ehH^\infty$ into a Fr\'echet space. Once a Hilbert norm $\norm -\norm$ is chosen 
on $\ehH$, the seminorms of $\ehH^\infty$ read as $\norm d\pio(X) (-) \norm$, 
for $X$ running over a basis of $\mathfrak{U}(\Lie(G_0))$.

If $H_0$ is a Lie subgroup of $G_0$, then a SUR of $G_0$ defines 
a SUR of $H_0$  by restriction. Its space of smooth vectors $\ehH^\infty_{H_0}$ 
is a Fr\'echet space which satisfies $\ehH^\infty\subseteq\ehH^\infty_{H_0}\subseteq\ehH$ 
and both inclusions are strict in general. If $H_0$ is the trivial group
we have $\ehH^\infty_{H_0}=\ehH$ as topological spaces.
The density property of $\ehH^\infty$ in $\ehH$ can be generalized as follows.

\begin{proposition}\label{prop:density}
Let $(\ehH,\pio)$ be a unitarizable SUR of $G_0$
and $H_0\trianglelefteq G_0$  a normal subgroup.
Then, $\ehH^\infty$ is dense in $\ehH_{H_0}^\infty$. 
\end{proposition}

\begin{proof}
We need a lemma.

\begin{lemma}\label{Lem:continuitypio} 
Let $(\ehH,\pio)$ be a unitarizable SUR of $G_0$
and $H_0\trianglelefteq G_0$  a normal subgroup.
For all $v\in \ehH^\infty_{H_0}$ and $g\in G_0$,
we have $\pio(g)v\in\ehH^\infty_{H_0}$.
Moreover, $\pio$ restricts as a continuous map 
$\pio:G_0\times\ehH^\infty_{H_0}\to\ehH^\infty_{H_0}$ 
for the Fr\'echet topology of $\ehH^\infty_{H_0}$.
\end{lemma}
\begin{proof}
Since $\pio$ is unitarizable, the topology on $\ehH$ is given by a Hilbert norm
which is preserved by all the representation operators $\pio(g)$ with $g\in G_0$.
This norm is denoted by  $\norm - \norm$.

Let $v\in \ehH^\infty_{H_0}$ and $g\in G_0$.
By definition of $d\pio$, we have
\begin{equation*}
d\pio(X)\pio(g)v=\pio(g)d\pio(\Ad_g(X))v,
\end{equation*}
for all $X\in\mathfrak U(\Lie(H_0))$.
Since $\pio$ is unitarizable and $\Ad_g(X)\in\mathfrak U(\Lie(H_0))$ (recall that $H_0$ is normal),
we have $\norm\pio(g)d\pio(\Ad_g(X))v\norm<+\infty$. Hence,
the element $\pio(g)v\in\ehH$ has finite seminorms, 
$\norm d\pio(X)\pio(g)v\norm <+\infty$ for all $X\in\mathfrak U(\Lie(H_0))$,
and we have $\pio(g)v\in\ehH^\infty_{H_0}$.

Consider $v\in \ehH^\infty_{H_0}$, $g\in G_0$ and a sequence $(g_j)$ in $G_0$ converging to $g$. 
As $\pio$ is unitarizable, for all $X\in\mathfrak U(\Lie(H_0))$, 
we get
\begin{align*}
\norm d\pio(X)\pio(g_j)v-&d\pio(X)\pio(g)v\norm\\
&\leq\norm \pio(g_j)d\pio (\Ad_{g_j}X)v-\pio(g)d\pio(\Ad_{g}X)v\norm\\
&\leq \norm \pio(g_j)\left(d\pio(\Ad_{g_j}X)v-d\pio(\Ad_{g}X)v\right)\norm
+ \norm (\pio(g_j)-\pio(g))d\pio(\Ad_{g}X)v\norm \\
&\leq\norm d\pio(\Ad_{g_j}X)v-d\pio(\Ad_{g}X)v\norm + \norm (\pio(g_j)-\pio(g))d\pio(\Ad_{g}X)v\norm.
\end{align*}
Since 
$\pio$ is strongly continuous, we have
$\lim_{j\to\infty}\norm (\pio(g_j)-\pio(g))d\pio(\Ad_{g}X)v\norm= 0$. 
We also have $\lim_{j\to\infty}\norm d\pio(\Ad_{g_j}X)v-d\pio(\Ad_{g}X)v\norm= 0$,
as a consequence of $\lim_{j\to\infty}\Ad_{g_j}X= \Ad_gX$
and of the continuity of the linear map $\Lie(H_0)\ni Y\mapsto d\pio(Y)v\in\ehH$.
Therefore, the map 
\begin{equation*}
G_0\ni g\quad\mapsto\quad \pio(g)v\in\ehH^\infty_{H_0},
\end{equation*}
is continuous. Since $\ehH^\infty_{H_0}$ is a Fr\'echet space, 
this implies the joint continuity of the map 
$\pio:G_0\times\ehH^\infty_{H_0}\to\ehH^\infty_{H_0}$ (see \cite[page 219]{Warner:1972}).
\end{proof}

We are now ready to prove the proposition. Let $v\in\ehH^\infty_{H_0}$ and 
choose a decreasing sequence of compact subsets $(K_j)_{j\in\gN}$ 
of $G_0$ converging to $\{e_{G_0}\}$. Then, there exists a sequence of smooth 
positive functions $(f_j)$ such that $\text{Supp}(f_j)\subseteq K_j$ and $\int_{G_0} f_j=1$. 
Applying Proposition 4.1.1.2. in \cite{Warner:1972} to the strongly continuous
representation $\pio:G_0\times\ehH^\infty_{H_0}\to\ehH^\infty_{H_0}$,
we get that $\lim_{j\to\infty}\pio(f_j)v= v$ in the topology 
of $\ehH^\infty_{H_0}$.
Since $(\pio(f_j)v)\in\ehH^\infty$, this proves
that $\ehH^\infty$ is dense in $\ehH^\infty_{H_0}$.
\end{proof}

\subsection{SUR of super Harish-Chandra pairs}
\label{subsec-surhc}

For a super Harish-Chandra pair $(G_0,\kg_\gR)$, representations by bounded operators are 
too restrictive \cite{Carmeli:2006}. Instead, one should fix a common domain 
of definition for operators representing the odd elements of the $\gR$-Lie superalgebra $\kg_\gR$. 
A natural choice is the space of smooth vectors $\ehH^\infty$ for the underlying 
superunitary representation of the Lie group $G_0$. 
Recall that $\kg_\gR=(\kg_\gR)_0\oplus (\kg_\gR)_1$ and $(\kg_\gR)_0=\Lie(G_0)$.

\begin{definition}
\label{def-reprhc}
Let $(G_0,\kg_\gR)$ be a super Harish-Chandra pair.
A {\bf super\-unitary representation} (SUR) of $(G_0,\kg_\gR)$ 
is a triple $(\ehH,\pio,\piG)$ such that
\begin{itemize}
\item $\ehH$ is a Hilbert superspace;
\item $\pio:G_0\to \caU(\ehH)$ is a superunitary representation of the Lie group $G_0$ on $\ehH$;
\item $\piG:\kg_\gR\to\End(\ehH^\infty)$ is a $\gR$-Lie superalgebra morphism
such that $\piG=d\pio$ on $(\kg_\gR)_0$ and
\begin{align}\nonumber
\forall g\in G_0,\ \forall X\in(\kg_\gR)_1, \qquad \piG(X)^\dag&=-\piG(X),\\  \label{Eq:HC-rep} \piG(\Ad_g(X))&=\pio(g)\piG(X)\pio(g)^{-1},
\end{align}
where $\ehH^\infty$ is the space of smooth vectors  of the representation $\pio$ and
 $\Ad:G_0\times\kg_\gR\to\kg_\gR$ is the defining  action of the pair $(G_0,\kg_\gR)$.
\end{itemize}
\end{definition}


Since $\piG=d\pio$ on $(\kg_\gR)_0$, the display \eqref{Eq:HC-rep}
holds for all $X\in\kg_\gR$. If $G_0$ is connected,
the second line of the display is redundant, this is a consequence
of the fact that $\piG$ is a $\gR$-Lie superalgebra morphism.
If $\ehH$ is a standard Hilbert superspace, i.e.\ 
$\ehH_{[-1]}=\ehH_{[-i]}=\algzero$,
the above definition reduces exactly to the definition 
of a superunitary representation in \cite{Carmeli:2006}.
Note that $\pio$ is not (yet) asked to be unitarizable.

\begin{remark}\label{Rmk:piGcontinuous}
Given a SUR $(\ehH,\pio,\piG)$ of $(G_0,\kg_\gR)$, consider the operators $\piG(X)$
with $X\in\kg_\gR$. The condition $\piG(X)^\dag=-\piG(X)$ allows to prove easily that
$\piG(X)$ is closable.  Since $\piG(X)$ is defined on the whole Fr\'echet space $\ehH^\infty$,
it is continuous w.r.t.\ the Fr\'echet topology of $\ehH^\infty$,
by the closed graph theorem for Fr\'echet spaces.
As a consequence of this result, the operators $\piG(X)$ are bounded 
if $\ehH^\infty=\ehH$.
\end{remark}

Consider $(\ehH^{(j)},\pio^{(j)},\piG^{(j)})$ with $j=1,2$, two SUR's of 
the super Harish-Chandra pair $(G_0,\kg_\gR)$. 
Clearly, 
an intertwiner $T$
of the underlying $G_0$-representations 
 (see Equation \eqref{Intertwine:pio}) preserves the spaces of smooth vectors, $T((\ehH^{(1)})^\infty)\subseteq (\ehH^{(2)})^\infty$.
The operator $T$ is an intertwiner of the full SUR's 
of the super Harish-Chandra pair if it also satisfies
\begin{equation}\label{Intertwine:piG}
\forall X\in\kg_\gR,\  \forall\vv\in(\ehH^{(1)})^\infty,\quad
\big(T\circ \piG^{(1)}(X)\big)\vv=\big(\piG^{(2)}(X)\circ T\big)\vv.
\end{equation}
{\defin Morphisms} between two SUR's are the even intertwiners.
If there exists a superunitary intertwiner, $T\in\caU(\ehH^{(1)},\ehH^{(2)})$, then 
$(\ehH^{(1)},\pio^{(1)},\piG^{(1)})$ and $(\ehH^{(2)},\pio^{(2)},\piG^{(2)})$
are said to be isomorphic or equivalent. As usual, to be isomorphic is an equivalence relation 
and one can speak of {\defin equivalence classes of SUR's}.
If $\kg_\gR=(\kg_\gR)_0$, SUR's of $(G_0,\kg_\gR)$  coincide
with SUR's of the Lie group $G_0$, as well as their intertwiners. 
\\

Given a SUR $(\ehH,\pio,\piG)$ ,
a closed subspace $\ehF\leq\ehH$  is said $(\pio,\piG)$-invariant if 
$\pio(g)\ehF\subseteq\ehF$ for all $g\in G_0$ and
$\piG(X)\big(\ehF\cap\ehH^\infty\big)\subseteq\ehF\cap\ehH^\infty$ 
for all $X\in\kg_\gR$. 
If $\ehF$ is in addition a Hilbert sub-superspace of $\ehH$,
then $(\ehF,\pio,\piG)$ is a SUR
called a sub-representation of $(\ehH,\pio,\piG)$.

\begin{definition}
Let $(\ehH,\pio,\piG)$ be a SUR of a 
super Harish-Chandra pair $(G_0,\kg_\gR)$. It is called 
{\defin indecomposable} if it admits no proper sub-representation
and {\defin (graded-)irreducible} if there exists no proper closed $(\pio,\piG)$-invariant (graded) subspace $\ehF\leq\ehH$.
\end{definition}

The orthogonal of a sub-representation is clearly a sub-representation.
Hence, a SUR can always be decomposed as a direct sum
of indecomposable SUR's. A graded-irreducible SUR is always indecomposable 
but the converse is not always true,
see Example \ref{IndecomposableReducible} below.
There exist graded-irreducible SUR's which are not irreducible,
see Example \ref{GradedIrrReducible} below.
\\

If $(\ehH^{(j)},\pio^{(j)},\piG^{(j)})$ with  $j=1,2$ are two SUR's of 
the super Harish-Chandra pair $(G_0,\kg_\gR)$, 
one can form their {\defin direct sum} by considering the triple
\begin{equation*}
\big(\ehH^{(1)}\oplus\ehH^{(2)}, \pio^{(1)}\oplus\pio^{(2)}, \piG^{(1)}\oplus \piG^{(2)}\big).
\end{equation*}
This is straightforwardly a SUR of $(G_0,\kg_\gR)$ since $(\ehH^{(1)}\oplus\ehH^{(2)})^\infty=(\ehH^{(1)})^\infty\oplus(\ehH^{(2)})^\infty$. 

Let $(\ehH^{(j)},\pio^{(j)},\piG^{(j)})$ with  $j=1,2$ be two SUR's of the super Harish-Chandra pairs $(G_0^{(j)},\kg_\gR^{(j)})$. We want to construct their {\defin tensor product}. First, we set $\ehH:=\ehH^{(1)}\hat\otimes \ehH^{(2)}$ by using the Hilbert tensor product (see Proposition \ref{prop-superhilbert-tensprod}), and $\pi_0(g_1,g_2):=\pio^{(1)}(g_1)\otimes\pio^{(2)}(g_2)$, for $g_1\in G_0^{(1)}$ and $g_2\in G_0^{(2)}$. The pair $(\ehH,\pio)$ is a SUR of the Lie group $G_0:=G_0^{(1)}\times G_0^{(2)}$. 
Then, the natural definition of the representation $\piG$ 
would be
\begin{equation*}
\piG(X_1,X_2)(v_1\otimes v_2):=\big(\piG^{(1)}(X_1)v_1\big)\otimes v_2 + v_1\otimes\big( \piG^{(2)}(X_2)v_2\big).
\end{equation*}
The operators $\piG(X_1,X_2)$ are well-defined on the algebraic tensor product $(\ehH^{(1)})^\infty\otimes (\ehH^{(2)})^\infty$,
which is included in $\ehH^\infty$.
But in general, they are not defined on the whole of $\ehH^\infty$ and do not stabilize it. 
However, if both conditions are satisfied, the resulting representation $(\ehH, \pi_0,\piG)$ 
is a SUR of $(G_0^{(1)}\times G_0^{(2)},\kg_\gR^{(1)}\oplus\kg_\gR^{(2)})$.

One is often interested in the tensor product of two SUR's of the same 
Harish-Chandra pair $(G_0,\kg_\gR)$, viewed as a representation of the original pair. 
However, restriction of a SUR to a sub-pair is a delicate notion, discussed below.
In particular, even if the tensor product exists as a SUR of $(G_0\times G_0,\kg_\gR\oplus\kg_\gR)$, 
its restriction to the diagonal sub-pair $(G_0,\kg_\gR)$ might not be a SUR, see Example \ref{Example:tensor}.
\\



Let $(\ehH,\pio,\piG)$ be a SUR of a super Harish-Chandra pair $(G_0,\kg_\gR)$.
Consider a sub-pair $(H_0,\kh_\gR)$, i.e., $H_0$ is a subgroup of $G_0$,
$\kh_\gR$ is a $\gR$-Lie sub-superalgebra of $\kg_\gR$ and the action $H_0\times\kh_\gR\to\kh_\gR$
is the restriction of the defining action $G_0\times\kg_\gR\to\kg_\gR$ of $(G_0,\kg_\gR)$.
Recall that the spaces of smooth vectors on $G_0$ and $H_0$  
satisfy then $\ehH^\infty\subseteq\ehH^\infty_{H_0}$.
The SUR $(\ehH,\pio,\piG)$ admits a {\defin restriction} to $(H_0,\kh_\gR)$ if there exists a map 
$\piH:\kh_\gR\to \End(\ehH^\infty_{H_0})$ such that $(\ehH,\pio,\piH)$ is a SUR of $(H_0,\kh_\gR)$
and
\begin{equation}\label{Eq:restriction}
\forall X\in \kh_\gR,\ \forall \vv\in\ehH^\infty,\quad \piH(X)\vv=\piG(X)\vv.
\end{equation}
If $\ehH^\infty_{H_0}=\ehH^\infty$, then the restriction always exists and is unique.
The general case is more involved.
For $X\in(\kh_\gR)_0$, one can set $\piH(X):=d(\left.\pio\right|_{H_0})(X)$
which is a continuous operator on $\ehH^\infty_{H_0}$. However, if $X\in(\kh_\gR)_1$, 
there is no natural way to define $\piH(X)$ on $\ehH_{H_0}^\infty$ in general. Thus, restriction 
to a sub-pair might not exist, see Example~\ref{Example:tensor}.
We provide below conditions for a restriction to exist and be unique when $\pio$ is unitarizable.

\begin{proposition}\label{Prop:restriction}
Let $(\ehH,\pio,\piG)$ be a SUR of a super Harish-Chandra pair $(G_0,\kg_\gR)$, with $\pio$ unitarizable,
and $(H_0,\kh_\gR)$ be a sub-pair with $H_0$ a normal subgroup of $G_0$.
The SUR $(\ehH,\pio,\piG)$ admits a restriction to $(H_0,\kh_\gR)$ if and only if 
the operators $\piG(X)$ are continuous w.r.t.\ the Fr\'echet topology
of $\ehH^\infty_{H_0}$ for all $X\in(\kh_\gR)_1$.
If it exists, the restriction is unique and given by $(\ehH,\pio,\piH)$
with $\piH(X)$ the continuous extension of $\piG(X)$ to $\ehH^\infty_{H_0}$.
\end{proposition}
\begin{proof}
Assume that $(\ehH,\pio,\piG)$ admits a restriction to $(H_0,\kh_\gR)$, 
denoted by $(\ehH,\pio,\piH)$. By Remark \ref{Rmk:piGcontinuous}, 
for all $X\in(\kh_\gR)_1$, the operators $\piH(X)\in\End(\ehH^\infty_{H_0})$ 
are continuous w.r.t.\ the Fr\'echet topology of $\ehH^\infty_{H_0}$.
Hence, by Equation \eqref{Eq:restriction}, the same holds for the operators $\piG(X)$.

Conversely, assume that the operators $\piG(X)\in\End(\ehH^\infty)$ are continuous w.r.t.\ 
the Fr\'echet topology of $\ehH^\infty_{H_0}$ for all $X\in(\kh_\gR)_1$. By Lemma \ref{Lem:continuitypio},
for all $X\in(\kh_\gR)_0$, the operator $\piG(X)=d\pio(X)$ can be extended as a continuous endomorphism 
of $\ehH^\infty_{H_0}$, denoted by $\piH(X)$.
We obtain a map $\piH:\kh_\gR\to \End(\ehH^\infty_{H_0})$ such that $(\ehH,\pio,\piH)$ is a restriction
of the SUR $(\ehH,\pio,\piG)$.

Assume that $(\ehH,\pio,\piG)$ admits two restrictions to $(H_0,\kh_\gR)$, 
denoted by $(\ehH,\pio,(\piH)^{(j)})$, with $j=1,2$. For all $X\in\kh_\gR$, the Equation \eqref{Eq:restriction} yields the operator equality $(\piH)^{(1)}(X)=(\piH)^{(2)}(X)$
 on $\ehH^\infty$. Since $\ehH^\infty$ is dense in $\ehH^\infty_{H_0}$ (see Proposition \ref{prop:density}) 
and the two operators are continuous, we conclude that $(\piH)^{(1)}(X)=(\piH)^{(2)}(X)$ on the whole of 
$\ehH^\infty_{H_0}$. This means that the two restrictions coincide.
\end{proof}


\subsection{Strong SUR of super Harish-Chandra pairs}
\label{subsec-strongSURHC}

Now, we introduce a stronger notion of representation for a 
super Harish-Chandra pair $(G_0,\kg_\gR)$. 
It relies on the {\defin odd derived  pair} $(\DGR,\DgR)$, where 
$\DgR:=[(\kg_\gR)_1,(\kg_\gR)_1]\oplus(\kg_\gR)_1$ 
and  $\DGR$ is the connected Lie subgroup of $G_0$ with Lie algebra 
$[(\kg_\gR)_1,(\kg_\gR)_1]$. Therefore, the $\gR$-Lie superalgebra $\DgR$ is an ideal of $\kg_\gR$
satisfying $(\DgR)_1=(\kg_\gR)_1$ and the Lie group $\DGR$ is a normal subgroup of $G_0$.

\begin{definition}
A {\bf strong super\-unitary representation} (strong SUR) of  $(G_0,\kg_\gR)$ 
is a SUR $(\ehH,\pio,\piG)$ such that
\begin{itemize}
\item $(\ehH,\pio)$ is unitarizable;
\item $(\ehH,\pio,\piG)$ admits a  restriction to $(\DGR,\DgR)$.
\end{itemize}
\end{definition}


By Proposition \ref{Prop:restriction}, the restriction to $(\DGR,\DgR)$ is unique.
Moreover, the analog of Remark \ref{Rmk:piGcontinuous} for strong SUR rephrases as follows.

\begin{proposition}
\label{rmk-bounded}
Let $(\ehH,\pio,\piG)$ be a strong SUR of $(G_0,\kg_\gR)$. Then, for any 
$X\in(\kg_\gR)_1$, the operator $\piG(X)$ 
is continuous w.r.t.\ the Fr\'echet topology of $\ehH^\infty_{\DGR}$.
If in addition $\ehH^\infty_{\DGR}=\ehH$, then $\piG(X)$ extends 
uniquely as a bounded operator on $\ehH$.
\end{proposition}

A SUR $(\ehH,\pio,\piG)$ might fail to be a strong SUR  because
$\pio$ is not unitarizable, see Example \ref{Ex:nonunitarizable},
or because there exists $X\in(\kg_\gR)_1$ such that $\piG(X)\in\End(\ehH^\infty)$ 
is not continuous w.r.t. to the Fr\'echet topology of $\ehH^\infty_{\DGR}$, 
see Example \ref{Example:tensor}. The case where $\ehH$ is a standard Hilbert superspace 
is remarkable in this respect.

\begin{proposition}\label{Prop:Pre-repToRep}
Let $\ehH$ be a standard Hilbert superspace. 
Then, any SUR $(\ehH,\pio,\piG)$ of any super Harish-Chandra pair $(G_0,\kg_\gR)$ 
is a strong SUR.
\end{proposition}

\begin{proof}
This is a direct consequence of Proposition $2$ in \cite{Carmeli:2006}.
\end{proof}

The next Lemma provides an alternative description for strong SUR's.

\begin{lemma}\label{lem:SURong1}
Let $(G_0,\kg_\gR)$ be a super Harish-Chandra pair and $(\ehH,\pio)$ a unitarizable SUR of $G_0$.
Assume there exists a map $\rho:(\kg_\gR)_1\to \End(\ehH^\infty_{D(G_0)})$ such that
\begin{enumerate}[label=(\roman*)]
\item $\forall X,Y \in(\kg_\gR)_1,\; v\in\ehH^\infty, \quad [\rho(X),\rho(Y)]) v = d\pio([X,Y]) v$ ;
\item $\forall X \in(\kg_\gR)_1, \quad\rho(X)^\dag=-\rho(X)$ ;
\item $\forall g\in G_0, \; X\in (\kg_\gR)_1,\quad \pio(g)\rho(X)\pio(g^{-1})=\rho(\Ad_g X)$ .
\end{enumerate} 
Then, there exists a unique map $\piG$ such that $(\ehH,\pio,\piG)$ is a strong SUR
and $\piG(X) v = \rho(X) v$ for all $X\in(\kg_\gR)_1$ and $v\in\ehH^\infty$. Conversely, if $(\ehH,\pio,\piG)$ is a strong SUR 
with $(\ehH,\pio,\piG^{D(G_0)})$ its restriction to 
the odd derived pair,
then $\rho:=(\piG^{D(G_0)})_{|(\kg_\gR)_1}$ satisfies 
$(i)$, $(ii)$ and $(iii)$.
\end{lemma}
\begin{proof}
We first prove the uniqueness of $\piG$. 
If $X\in (\kg_\gR)_0$, by definition of a SUR, we should have $\piG(X):=d\pio(X)\in\End(\ehH^\infty)$.
If $X\in (\kg_\gR)_1$, by hypothesis, we should have $\piG(X)v:=\rho(X)v$ for all $v\in\ehH^\infty$.

We prove now that $(\ehH,\pio,\piG)$ is indeed a strong SUR if $\piG$ is defined as above.
Let $v\in\ehH^\infty$ and $X\in V$. Consider the map
\begin{equation}\label{Eq:rhosmooth}
g\in G_0\quad\mapsto\quad \pi_0(g)\rho(X)v.
\end{equation}
By hypothesis $(iii)$, we have 
$\pi_0(g)\rho(X)v=\rho(\Ad_g X)\pi_0(g)v$.
The map $g\mapsto \pi_0(g)v$ is smooth because $v$ is a smooth vector,
the adjoint action $g\mapsto \Ad_g X$ is smooth and the map $\rho$
is linear. Hence, the map in \eqref{Eq:rhosmooth} is smooth. 
This means that $\piG(X)v:=\rho(X)v\in\ehH^\infty$.
Therefore, $\piG$ defines a map $\piG:\kg_\gR\to\End(\ehH^\infty)$.
The hypothesis $(i)$ and $(iii)$ ensure that $\piG$ is a $\gR$-Lie superalgebra morphism.
Then, it is direct to check that $(\ehH,\pio,\piG)$ is a SUR.
By $(ii)$ and Remark \ref{Rmk:piGcontinuous}, 
for all $X\in (\kg_\gR)_1$, the operators $\piG(X)$ 
are continuous with respect to the Fr\'echet topology of $\ehH ^\infty_{D(G_0)}$.
By applying Proposition \ref{Prop:restriction}, we deduce that $(\ehH,\pio,\piG)$ is a strong SUR of $(G_0,\kg_\gR)$.

The converse statement is obvious.
\end{proof}

Morphisms of strong SUR's are simply morphisms of the underlying SUR's.
The notion of equivalence is the same for SUR's and for strong SUR's.
Note that, if there exists a superunitary intertwiner between a SUR 
$(\ehH^{(1)},\pio^{(1)},\piG^{(1)})$ and
a strong SUR $(\ehH^{(2)},\pio^{(2)},\piG^{(2)})$, then the SUR 
$(\ehH^{(1)},\pio^{(1)},\piG^{(1)})$ is a strong SUR.
In other words, the category of strong SUR's is a strictly full subcategory of 
the category of SUR's. By Proposition \ref{Prop:Pre-repToRep} and Theorem \ref{Thm:IsoHilbert}, 
the category of SUR's on standard Hilbert superspaces, introduced in \cite{Carmeli:2006},
is a strictly full subcategory of the category of strong SUR's.

A strong SUR is indecomposable, graded-irreducible or irreducible if and only if the underlying
SUR does. 

{\defin Schur's Lemma} states that the intertwiners of an irreducible representation
are the scalar multiples of identity. As is well-known, it holds for finite dimensional representations
and for unitary representations of Lie groups. It also holds for strong SUR's
of super Harish-Chandra pairs on standard Hilbert superspaces \cite{Carmeli:2006}.
However, up to our knowledge, no analog of Schur's Lemma is known 
for Krein-unitary representations, and a fortiori for SUR's.
We restrict our attention to strong SUR's $(\ehH,\pio,\piG)$ 
of a super Harish-Chandra pair $(G_0,\kg_\gR)$, where $\ehH$ is a general
Hilbert superspace. Then, by definition,
the $G_0$-representation $(\ehH,\pio)$ admits a unitarization.
Hence, it can be decomposed into a direct integral of unitary
irreducible $G_0$-representations. In the following, we assume
that the involved measure is discrete, i.e.,
\begin{equation}\label{decompoHpio}
\ehH=\bigoplus_{i\in I} \ehH_i\hat{\otimes}\ehF_i,
\end{equation}
where $I$ is a given set (of arbitrary cardinality) and 
we have: 
$(\ehH_i)_{i\in I}$ is a sequence of Hilbert spaces carrying inequivalent unitary irreducible
representations of $G_0$ and $(\ehF_i)_{i\in I}$ is a sequence of Hilbert spaces carrying the trivial 
representation of $G_0$.

\begin{proposition}[Schur's Lemma]\label{Schur}
Let $(\ehH,\pio,\piG)$ be a strong SUR 
which decomposes as in \eqref{decompoHpio} and 
$T$ a homogeneous intertwiner of $(\ehH,\pio,\piG)$.
Assume there exists $i\in I$ such that $\dim\ehF_i<\infty$. If  
$(\ehH,\pio,\piG)$ is irreducible 
then $T$ is a scalar multiple of the identity operator.
If $(\ehH,\pio,\piG)$ is graded-irreducible then $T$ is 
is a scalar multiple of the identity operator $\gone$ or of a fixed involutive odd operator~$A$.
\end{proposition}
\begin{proof}
Assume that $T\in\caB(\ehH)$ is an intertwiner of the irreducible SUR $(\ehH,\pio,\piG)$.
By Corollary $4.3.1.3$ in \cite{Warner:1972}, $T$ maps irreducible 
unitary $G_0$-representations to equivalent ones.
As the spaces $\ehH_i$ carry inequivalent $G_0$-representations 
for different $i\in I$, we have 
$T(\ehH_i\hat{\otimes}\ehF_i)\subseteq \ehH_i\hat{\otimes}\ehF_i$,  for all $i\in I$.
Hence, $T$ admits a block-diagonal decomposition
$T=\bigoplus_{i\in I} T_i$ with $T_i\in\caB(\ehH_i\hat{\otimes}\ehF_i)$. 
Applying classical Schur's Lemma to the unitary representations $\ehH_i\hat{\otimes}\ehF_i$
of $G_0$, we get that $T_i=\gone\otimes t_i$ with $t_i\in\caB(\ehF_i)$. 

Consider $i\in I$ such that $\dim\ehF_i<\infty$. Then, the linear operator
$t_i$ admits at least one complex eigenvalue $\l\in\gC$. The eigenspace
$$
\ehH_\l:=\{v\in\ehH \; | \; T v=\l v\}
$$
is then non-zero. As $T$ is a continuous homogeneous operator,
$\ehH_\l$ is a closed subspace of $\ehH$.
Since $T$ intertwines the $(G_0,\kg_\gR)$-action, the space $\ehH_\l$
is $(\pio,\piG)$-stable. Therefore, $(\ehH,\pio,\piG)$ being irreducible, we conclude that
$\ehH_\l=\ehH$ and $T=\l\gone$.

If $(\ehH,\pio,\piG)$ is graded-irreducible, we can formulate the same arguments 
for an intertwiner $T$ of degree 0. Then, $\ehH_\lambda$ is a closed graded subspace 
of $\ehH$, which is $(\pio,\piG)$-stable. So, $\ehH_\l=\ehH$ and $T=\l\gone$.
The composition of two odd intertwiners $T,T'$ is an even intertwiner, 
so that $T$ and $T'$ differ by a scalar multiple.
Hence, every odd intertwiner is a scalar multiple of a fixed involutive odd operator $A$.
\end{proof}


A direct sum  of strong SUR's is automatically a strong SUR. As for the tensor product of 
strong SUR's, it might not be a strong SUR and even not a SUR. 
The restriction of a strong SUR of  $(G_0,\kg_\gR)$ to a sub-pair
$(H_0,\kh_\gR)$ exists and is unique if 
$D(G_0)\subseteq H_0\subseteq G_0$ and it is a strong SUR.
However, restrictions to general sub-pairs might not exist.

Sufficient conditions for a tensor product to be a strong SUR are stated in the next proposition.
 
\begin{proposition}
\label{prop-tensprodSUR}
Let $(\ehH^{(j)},\pi_0^{(j)},\pi_*^{(j)})$ be two strong SUR's of 
$(G_0^{(j)},\kg_\gR^{(j)})$, for $j=1,2$. Set $\ehH:=\ehH^{(1)}\hat\otimes\ehH^{(2)}$
and $(G_0,\kg_\gR):=(G_0^{(1)}\times G_0^{(2)},\kg_\gR^{(1)}\oplus\kg_\gR^{(2)})$.
 If the equalities $(\ehH^{(j)})^\infty_{D(G_0^{(j)})}=\ehH^{(j)}$ hold for $j=1,2$, 
then there exists a unique strong SUR $(\ehH,\pi_0,\pi_*)$ of 
$(G_0,\kg_\gR)$ on $\ehH$ such that
\begin{itemize}
\item for all $(g_1,g_2)\in G_0$ and $v_1\otimes v_2\in \ehH^{(1)}\otimes\ehH^{(2)}$, we have
\begin{equation}\label{Tensor-pio}
\pio(g_1,g_2)\, v_1\otimes v_2=\pio^{(1)}(g_1)v_1\otimes\pio^{(2)}(g_2)v_2\, ;
\end{equation}
\item for all $(X_1,X_2)\in \kg_\gR$ and  $v_1\otimes v_2\in (\ehH^{(1)})^\infty\otimes(\ehH^{(2)})^\infty$, we have
\begin{equation}\label{Tensor-piG}
\piG(X_1,X_2)\, v_1\otimes v_2=\piG^{(1)}(X_1)v_1\otimes v_2+v_1\otimes\piG^{(2)}(X_2) v_2.
\end{equation}
\end{itemize}
The strong SUR $(\ehH,\pi_0,\pi_*)$ is called the tensor product representation.
\end{proposition}
\begin{proof}
Since $(\ehH^{(j)},\pio^{(j)})$ are unitarizable for $j=1,2$, there exists a unique 
map $\pio:G_0\to\ehH$ satisfying \eqref{Tensor-pio}. Clearly, $(\ehH,\pio)$ 
is a unitarizable SUR of $G_0$.

As $(\ehH^{(j)})^\infty_{D(G_0^{(j)})}=\ehH^{(j)}$, the operators $\pi_*^{(j)}(X_j)$ 
are bounded on $\ehH^{(j)}$ for all $X_j\in \Lie\left(D(G_0^{(j)})\right)$, with $j=1,2$. 
In view of Formula \eqref{Tensor-piG}, if $X\in \Lie\left(D(G_0)\right)$,
the operator $\pi_*(X)$ is then bounded on 
the dense subspace $(\ehH^{(1)})^\infty\otimes(\ehH^{(2)})^\infty\subseteq\ehH$.
This implies that $\piG(X)\in\caB(\ehH)$ 
and $\ehH^\infty_{D(G_0)}=\ehH$.
By Proposition \ref{rmk-bounded}, for any $X_j\in(\kg_\gR^{(j)})_1$ 
the operators $\pi_*^{(j)}(X_j)$ extend as bounded operators on $\ehH^{(j)}$. 
Therefore, so does $\pi_*(X)$ if $X\in(\kg_\gR)_1$ and we get
 a linear map $\rho:(\kg_\gR)_1\to \caB(\ehH)$ which satisfies
$\rho(X) v=\piG(X) v $ for all $X\in(\kg_\gR)_1$ and $v\in\ehH^\infty$. 
As $(\ehH^{(j)},\pi_0^{(j)},\pi_*^{(j)})$ are SUR's, for $j=1,2$,
the map $\rho$ satisfies the three hypotheses of Lemma \ref{lem:SURong1} on
$(\ehH^{(1)})^\infty\otimes(\ehH^{(2)})^\infty$. Since the latter space is dense in $\ehH$,
Lemma \ref{lem:SURong1} applies and we conclude that there exists a unique strong SUR $(\ehH,\pi_0,\pi_*)$
satisfying Equations \eqref{Tensor-pio}-\eqref{Tensor-piG}.
\end{proof}

Sufficient conditions to have restrictions of a strong SUR to any sub-pair
 are stated in the next proposition.

\begin{proposition}\label{prop:restriction-strong}
If $(\ehH,\pi_0,\pi_*)$  is a strong SUR of  $(G_0,\kg_\gR)$ 
such that $\ehH^\infty_{D(G_0)}=\ehH$, then it admits a
restriction to any sub-pair of  $(G_0,\kg_\gR)$. 
\end{proposition}
\begin{proof}
Consider $(H_0,\kh_\gR)$ a sub-pair of  $(G_0,\kg_\gR)$.
The SUR $(\ehH,\pio)$ of $G_0$ automatically restricts as a SUR of $H_0$
and we have $\ehH^\infty_{D(G_0)}\subseteq\ehH^\infty_{D(H_0)}\subseteq\ehH$.
As $\ehH^\infty_{D(G_0)}=\ehH$, we deduce that 
$\ehH^\infty_{D(H_0)}=\ehH$.
By Proposition \ref{rmk-bounded}, if $X\in(\kg_\gR)_1$ 
the operator $\piG(X)$ extends as a bounded operator on $\ehH$,
denoted by $\rho(X)$. 
Hence, we get a linear map $\rho:(\kg_\gR)_1\to \caB(\ehH)$. 
One easily checks that Lemma \ref{lem:SURong1} applies. As a result, we get a restriction of 
$(\ehH,\pi_0,\pi_*)$ to $(H_0,\kh_\gR)$ which is a strong SUR.
\end{proof}

\subsection{SUR and strong SUR of Lie supergroups}

In this section, we provide a definition
of (strong) SUR for  Lie supergroups
and show that it agrees with the one for super Harish-Chandra pairs via the equivalence 
between both structures.

As a Lie supergroup is a patching of (even part of) 
graded $\superA$-vector spaces, representations of $G$ should be valued 
in graded $\superA$-vector spaces (see section \ref{sec-prelim}). 
Given a Hilbert superspace $(\ehH,\langle-,-\rangle)$,
the superhermitian inner product is extended by 
$\superA$-bilinearity to $\superA\otimes\ehH$:
\begin{equation}
\langle av,bw\rangle=(-1)^{|a|\sigma+|b|(\sigma+|v|)}\overline{a}b\langle v,w\rangle,\label{eq-abilin}
\end{equation}
for $a,b\in\superA_\gC$, $v,w\in \ehH$ and $\sigma$ is the parity of $\ehH$. 
This means that the degree $\sigma$ of the inner product is carried by its left bracket, 
as for an inner product defined by an integral (see Examples \ref{ex-superhilbert} and \ref{ex-holfct}).
If $\ehF$ is a dense subspace of $(\ehH,\langle-,-\rangle)$, 
we will consider the group of superunitary operators 
$\caU_\superA(\superA\otimes\ehF)$, which consists of even invertible 
$\superA$-linear maps $B:\superA\otimes\ehF\to \superA\otimes\ehF$ such that
\begin{equation*}
\forall\vv,\vw\in\ehF,\quad \langle B\vv,
B\vw\rangle=\langle\vv,\vw\rangle.
\end{equation*}
Note that no continuity assumptions are made on the maps $B$. 

\begin{definition}\label{Def:prerep}
A {\bf superunitary representation} (SUR) of 
a Lie supergroup $G$ is a triple $(\ehH,\pi_0,\pi)$ such that
\begin{itemize}
\item $\ehH$ is a Hilbert superspace;
\item $\pio:\gB G\to \caU(\ehH)$ is a superunitary representation 
of $\gB G$ on $\ehH$;
\item $\pi:G\to \caU_\superA(\superA\otimes\ehH^\infty)$ is a group morphism
such that
\begin{align}\nonumber
\forall \vv\in\ehH^\infty,&\quad  
\pi^\vv: g\mapsto\pi(g)\vv \in\caC^\infty(G,\superA\otimes\ehH^\infty),\\  \label{Eq:SG-rep}
\forall g\in G,& \quad \pi(\gB g)\vv=\pio(\gB g)\vv,
\end{align}
where $\ehH^\infty$ is the space of smooth vectors  of the representation $\pio$.
\end{itemize}
\end{definition}

A {\defin morphism} between two SUR's $(\ehH^{(j)},\pio^{(j)},\pi^{(j)})$, with $j=1,2$,
 of a Lie supergroup $G$ is a morphism 
of the underlying $\gB G$-representations $T\in\Hom(\pio^{(1)},\pio^{(2)})$ 
such that
\begin{equation*}
\forall g\in G,\ \forall \vv\in\superA\otimes(\ehH^{(1)})^\infty,\quad
(\gone\otimes T)\circ\pi^{(1)}(g)\vv=\pi^{(2)}(g)\circ(\gone\otimes T)\vv,
\end{equation*}
where $\gone$ denotes the identity operator on $\superA$.

Let $G$ be a Lie supergroup and $H$ be a Lie sub-supergroup.  
A SUR $(\ehH,\pio,\pi)$ of $G$ admits a {\defin restriction} to $H$ if there exists a map 
$\pi^H:H\to \caU_\superA(\superA\otimes\ehH^\infty_{\gB H})$
such that $(\ehH,\pio,\pi^H)$ is a SUR of $H$ and
\begin{equation*}
\forall h\in H,\ \forall \vv\in\superA\otimes\ehH^\infty,\quad \pi^H(h)\vv=\pi(h)\vv.
\end{equation*}
Here, $\ehH^\infty_{\gB H}$ denotes the space of smooth vectors on $\gB H\leq\gB G$. 

Let $G$ be a Lie supergroup. The {\defin odd derived supergroup} $\DG\leq G$
is the connected Lie sub-supergroup of $G$ with Lie superalgebra 
$\Dg:=[\kg_1,\kg_1]\oplus\kg_1\leq \kg=\Lie(G)$. 

\begin{definition}
\label{def-reprsg}
A {\bf strong superunitary representation} (strong SUR)
of  a Lie supergroup $G$ is a SUR $(\ehH,\pio,\pi)$ such that
\begin{itemize}
\item $(\ehH,\pio)$ is unitarizable;
\item $(\ehH,\pio,\pi)$ admits a  restriction to $\DG$.
\end{itemize}
\end{definition}

A morphism between two strong SUR's of a Lie supergroup $G$ is a morphism 
of the underlying SUR's.

\medskip

Recall that a Lie supergroup $G$, of dimension $m|n$, induces
a Lie group $\gB G$, of dimension $m$, and a Lie supergroup
$G^\wod$, of dimension $m|0$, obtained from $\gB G$ by $\superA_0$-scalar
extension. Under the equivalence between Lie supergroups 
and super Harish-Chandra pairs, $G$ corresponds to $(\gB G,\gB\kg)$, 
$G^\wod$ corresponds to $(\gB G,\gB\kg_0)$ and 
the odd derived Lie sub-supergroup $\DG$ corresponds to 
the odd derived sub-pair $(\D(\gB G),\D(\gB\kg))$. 
Recall that $\Lie(\gB G)=\gB\kg_0$, $\Lie(G^\wod)=\superA\otimes\gB\kg_0$, 
$\Lie(G)=\kg=\superA\otimes\gB \kg$ and
$\kg_0=\superA_0\otimes\gB\kg_0\oplus\superA_1\otimes\gB\kg_1$.

\begin{theorem}\label{thm:Equivalence}
Let $G$ be a Lie supergroup and $(\gB G,\gB\kg)$ the corresponding super 
Harish-Chandra pair. 
\begin{enumerate}[label=(\roman*)]
\item The categories of SUR's of $G$ and $(\gB G,\gB\kg)$ are equivalent.
\item The categories of strong SUR's of $G$ and $(\gB G,\gB\kg)$ are equivalent.
\end{enumerate}
\end{theorem}

\begin{proof}
We need lemmas.  

\begin{lemma}\label{lem:pitopio}
The SUR's of  $\gB G$ and $G^\wod$ are 
in $1:1$ correspondence via 
the map $(\ehH,\pio)\mapsto(\ehH,\pio,\piwod)$,
where $\piwod$ is given by
\begin{equation}\label{piotopi}
\forall \vv\in\ehH^\infty,\ \forall g\in G^\wod,\quad \piwod(g)\vv:=
\widetilde{\pio^\vv}(g),
\end{equation}
with $\widetilde{\pio^\vv}\in\caC^\infty(G^\wod,\superA\otimes\ehH^\infty)$ 
the smooth extension of the map $\pio^\vv\in\caC^\infty(\gB G,\ehH^\infty)$, see \eqref{Function:Rm0}.
\end{lemma}

\begin{proof}
If $(\ehH,\pio,\pi)$ is a SUR of $G^\wod$
then $(\ehH,\pio)$ is obviously a SUR of $\gB G$.
Conversely, assume that $(\ehH,\pio)$ is a SUR of $\gB G$.
Let $\piwod$ be defined by Equation \eqref{piotopi}.
For all $g,g'\in\gB G$ and $\vv,\vw\in\ehH^\infty$, we have
\begin{equation*}
\piwod(g)(\vv+\vw)=\piwod(g)\vv+\piwod(g)\vw,\quad
\piwod(g)(\piwod(g')\vv)=\piwod(gg')\vv \quad \text{and}\quad
\langle\piwod(g)\vv,\piwod(g)\vw\rangle=\langle\vv,\vw\rangle.
\end{equation*}
By definition, two smooth functions in 
$\caC^\infty(G^\wod,\superA\otimes\ehH^\infty)$
that coincide on $\gB G$ are equal. Hence, the above equalities
hold for any $g,g'\in G^\wod$ and 
$(\ehH,\pio,\piwod)$ is then a SUR of $G^\wod$.
Moreover, if $(\ehH,\pio,\pi)$ is a SUR of $G^\wod$,
then, by definition, $\pi(-)\vv$ must be a smooth extension
of $\pio(-)\vv$. This means that $\pi=\piwod$. Therefore, the correspondence
between SUR's of $\gB G$ and $G^\wod$ is one-to-one.
\end{proof}

\begin{lemma}\label{Lem:expandpi}
Let $(\ehH,\pio,\pi)$ be a SUR of $G$.
There exists a unique map $d\pi:\gB \kg\to\End(\ehH^\infty)$
whose $\superA$-linear extension to $\kg$ satisfies
\begin{equation*}
\forall X\in \kg_0,\ \forall \vv\in\ehH^\infty,\quad d\pi(X)\vv
:=\left.{\frac{d}{dt}}\right|_{0} \pi(e^{tX})\vv.
\end{equation*}
The map $d\pi$ is a $\gR$-Lie superalgebra morphism which satisfies 
Equation \eqref{Eq:HC-rep}. 
\end{lemma}
\begin{proof}
Consider $v\in\ehH^\infty$. Since the exponential map $X\mapsto e^X$ 
is smooth in a neighborhood of $0\in\kg_0$, the map $X\mapsto d\pi(X)\vv$ is smooth on $\kg_0$.
 Routine arguments, similar as for strongly continuous 
representations of Lie groups, show that 
$d\pi:\kg_0\to\End(\superA\otimes\ehH^\infty)$
is a Lie superalgebra morphism which satisfies Equation \eqref{Eq:HC-rep}. 
By Remark \ref{rmk:smoothlinear}, $d\pi$ 
 is the $\superA$-linear extension of a map
$\gB \kg\to\End(\ehH^\infty)$ which shares the same properties.
Namely, this map is a $\gR$-Lie superalgebra morphism satisfying 
Equation \eqref{Eq:HC-rep}. 
\end{proof}

We are ready to prove the theorem. We first show $(i)$. 

Let $(\ehH,\pio,\pi)$ be a SUR of $G$.
By Lemma \ref{Lem:expandpi}, $\pi$ induces a linear map 
$d\pi:\gB\kg\to\End(\ehH^\infty)$ such that $(\ehH,\pio,d\pi)$ 
is a SUR of $(\gB G,\gB\kg)$. 

Conversely, let $(\ehH,\pio,\piG)$ be a SUR of $(\gB G,\gB \kg)$. 
Applying Lemma \ref{lem:pitopio},
one gets a SUR $(\ehH,\pio,\piwod)$ of $G^\wod$.
Extending $\piG$ by $\superA$-linearity, each element 
$X\in\kg^{(1)}=\superA_1\otimes\gB\kg_1$ defines a nilpotent operator
 $\piG(X)$, satisfying $\piG(X)^\dag=-\piG(X)$.
As a result, exponentiation gives a well-defined operator 
$\exp(\piG(X))\in\caU_\superA(\superA\otimes\ehH^\infty)$. 
According to Proposition \ref{prop-diffeom}, 
for each $g\in G$ there exists a unique couple  $(g_0,X)\in G^\wod\times\kg^{(1)}$
such that $g=g_0\, e^{X}$. We define the map 
$\pi:G\to \caU_\superA(\ehH^\infty\otimes\superA)$ by
\begin{equation}\label{piopiH:pi}
\pi(g):=\piwod (g_0)\exp(\piG(X)).
\end{equation}
Consider $g'\in G$ and $(g_0',X')\in G^\wod\times\kg^{(1)}$
such that $g'=g_0'\, e^{X'}$.
We use successively: Equation \eqref{piopiH:pi}, the extension  to $G^\wod$ of 
the second line in Equation \eqref{Eq:HC-rep}, Proposition \ref{Prop:BCH},
the property $\exp(d\piwod(X))=\piwod(e^X)$ if 
$X\in\superA_0\otimes(\kg_\gR)_0$ and Equation \eqref{GroupLaw:BCH},
to get
\begin{align*}
\pi(g)\pi(g')&=\piwod(g_0)\exp\big(\pi_*(X)\big)\piwod(g_0')\exp\big(\pi_*(X')\big),\\ 
&=\piwod(g_0)\piwod(g_0')\exp\big(\pi_*(\Ad_{(g_0')^{-1}}X)\big)\exp\big(\pi_*(X')\big),\\ 
&= \piwod(g_0g_0') \, \exp\big(\pi_*(\BCH_0(\Ad_{(g_0')^{-1}}X,X'))\big) \exp\big(\pi_*(\BCH_1(\Ad_{(g_0')^{-1}}X,X'))\big),\\
&= \piwod\Big(g_0g_0' \, e^{\BCH_0(\Ad_{(g_0')^{-1}}X,X')}\Big) \exp\big(\pi_*(\BCH_1(\Ad_{(g_0')^{-1}}X,X'))\big),\\
&=\pi(gg').
\end{align*}
This means that $\pi$ is a group morphism. Then, one can trivially check
that $(\ehH,\pio,\pi)$ is a SUR of $G$.

We prove now that the two constructions are inverse of each other. 
Start with a SUR $(\ehH,\pio,\piG)$ of a super Harish-Chandra
pair and let $\pi$ be given by Equation \eqref{piopiH:pi}, so that
$(\ehH,\pio,\pi)$ is a SUR of the corresponding Lie supergroup. 
By definition of SUR's of Lie supergroups and of super Harish-Chandra pairs, 
on $\gB \kg_0$, we have $d\pi=d\pi_0=\piG$. On $\gB \kg_1$, $d\pi$ is also equal to $\piG$.
Indeed, by definition of $\pi$, for all $X\in\gB \kg_1$, we have 
$$
\forall \vv\in\ehH^\infty,\quad d\pi(X)\vv=\left.{\frac{d}{dt}}\right|_{0} \pi(e^{tX})\vv=
\left.{\frac{d}{dt}}\right|_{0} \exp(\piG(tX))\vv=\piG(X)\vv.
$$ 
Conversely, consider a SUR $(\ehH,\pio,\pi)$ of a Lie
supergroup and let $\piG:=d\pi$, so that $(\ehH,\pio,\piG)$ is a SUR of 
the corresponding super Harish-Chandra pair. Then, the map defined by Equation \eqref{piopiH:pi}
coincides with $\pi$ on $G^\wod$. It coincides with $\pi$ on all of $G$ because
$\exp(d\pi(X))=\pi(e^X)$ for all $X\in\kg^{(1)}$. 

It remains  to check that, under the above correspondence 
of representations, morphisms of SUR's go to morphisms of SUR's. 
Let $(\ehH^{(j)},\pio^{(j)},\piG^{(j)})$ with $j=1,2$ be two SUR's of a super Harish-Chandra
pair and let $\pi^{(j)}$ be given by Equation \eqref{piopiH:pi}, so that
$(\ehH^{(j)},\pio^{(j)},\pi^{(j)})$ are SUR's of the corresponding Lie supergroups. 
Consider an even bounded operator $T\in\caB(\ehH^{(1)},\ehH^{(2)})$. Then, one can easily
prove the following equivalences
\begin{align*}
\forall g\in G^\wod,&\qquad 
T\circ\pio^{(1)}(\gB g)=\pio^{(2)}(\gB g)\circ T \,\Leftrightarrow\,
T\circ\piwod^{(1)}(g)=\piwod^{(2)}(g)\circ T,\\
\forall X\in \kg^{(1)}=\superA_1\otimes\gB \kg_1,&\quad 
T\circ\piG^{(1)}(\gB X)=\piG^{(2)}(\gB X)\circ T \,\Leftrightarrow\,
T\circ\exp(\piG^{(1)}(X))=\exp(\piG^{(2)}(X))\circ T.
\end{align*}
In view of Equation \eqref{piopiH:pi}, we deduce that $T$
is a morphism between the two SUR's $(\ehH^{(j)},\pio^{(j)},\piG^{(j)})$, with $j=1,2$,
if and only if $T$ is a morphism between the two SUR's $(\ehH^{(j)},\pio^{(j)},\pi^{(j)})$.
Therefore $(i)$ is proved.
\\

Strong SUR's are SUR's such that the underlying Lie group representation  
is unitarizable and the restriction to the odd derived
Lie sub-supergroup/super Harish-Chandra sub-pair is a SUR.
Hence, $(ii)$ follows from $(i)$.
\end{proof}

In the following, we freely use the equivalence of categories between 
Lie supergroups and Harish-Chandra super-pairs, and between their (strong) SUR's.  
They are denoted by $(\ehH,\pio,\pi)$ or $(\ehH,\pio,\piG)$
depending on the context and we freely refer to the other map $\piG$ or $\pi$ 
respectively.
Along the equivalence of categories stated in Theorem \ref{thm:Equivalence}, one can transport 
from super Harish-Chandra pairs to Lie supergroups
the definitions of irreducible, graded-irreducible and indecomposable (strong) SUR 
as well as direct sum or tensor product of (strong) SUR's.
This is obvious and left to the reader.

\subsection{Examples}

We illustrate definitions and properties above on examples. 
The illustrated notions are explicitly mentioned right before
each example. 
\medskip

In the next three examples we use the Abelian supergroup $\gR^{m|n}$,
whose super Harish-Chandra pair is $(\gR^m,\gR^m\oplus\gR^n)$.
Its odd derived supergroup is $\gR^{0|n}$, with super Harish-Chandra pair 
$(\algzero,\gR^n)$.

The necessity of $\superA$-scalar extension for the map $\pi$ in Definition 
\ref{Def:prerep} of SUR is illustrated on the left regular representation of $\gR^{m|n}$,
with $n=0$.

\begin{example}\label{Ex:A-extension}
The Hilbert superspace $\ehH:=L^2(\gR^{m|0})$ 
is isomorphic to $L^2(\gR^m)$ (see Example \ref{ex-superhilbert}).
Each element $\varphi\in L^2(\gR^m)$ corresponds to an element 
$\tilde \varphi\in L^2(\gR^{m|0})$. If $\varphi$ is a smooth function, 
$\tilde\varphi$ is given by Equation \eqref{Function:Rm0}.
The superspace $\ehH$ carries a SUR of the Abelian Lie group $\gR^{m}$, given by 
\begin{equation*}
\forall y\in\gR^m,\quad (\pio(y)\tilde \varphi):=\widetilde{\tau_y\varphi},
\end{equation*}
where $\tau_y\varphi:x\mapsto \varphi(x-y)$.
The space of smooth vectors $\ehH^\infty$ is then the space of smooth 
functions on $\gR^{m|0}$ whose all derivatives are square integrable.
The left regular representation of $\gR^{m|0}$ is the triple $(\ehH,\pio,\pi)$
where $\pi$ is given by
\begin{equation*}
\forall y\in\gR^{m|0},\quad (\pi(y)\tilde \varphi):=\tau_y\tilde{\varphi}=\sum_{\nu\in\gN^m}
\frac{\cN(y)^\nu}{\nu!}\,\widetilde{\big(\partial^\nu \tau_{\gB y}\varphi\big)}.
\end{equation*}
Therefore, $\pi(y)\tilde \varphi$ is not in $\ehH^\infty$ but in $\superA\otimes\ehH^\infty$.
One easily checks that $(\ehH,\pio,\pi)$ is a strong SUR of $\gR^{m|0}$. 
The corresponding infinitesimal representation of $\gR^m$ reads as 
$\piG(e_i)=-\partial_{x^i}$, where $(e_i)$ is the Cartesian basis of 
$\gR^m$ and $(x^i)$ its Cartesian coordinate system.
\end{example}

An indecomposable but reducible strong SUR 
is given by the left regular representation of the Abelian Lie supergroup
$\gR^{m|n}$, with $m=0$.

\begin{example}\label{IndecomposableReducible}
Functions in the Hilbert superspace $\ehH:=L^2(\gR^{0|n})\simeq\bigwedge \gR^n\otimes\gC$
read as $\varphi:\xi\mapsto \varphi(\xi)=\sum_\a \varphi_\a\xi^\a$ in multi-index notation,
with $\xi=(\xi^1,\ldots,\xi^n)\in\gR^{0|n}$.
We set
\begin{equation*}
\forall \eta\in\gR^{0|n},\quad (\pi(\eta) \varphi)(\xi):=\varphi(\xi-\eta).
\end{equation*}
One easily checks that $(\ehH,\gone,\pi)$ is a strong SUR of $\gR^{0|n}$. 
In the Cartesian basis $(e_i)$ of $\gR^{n}$, 
the infinitesimal representation is given by $\pi_*(e_i)=-\partial_{\xi^i}$.
Therefore, constant functions pertain to all $(\pio,\piG)$-invariant subspace
of $L^2(\gR^{0|n})$. In particular, the representation is indecomposable.
As the $1$-dimensional space of constant functions is $(\pio,\piG)$-invariant and of degree~$0$,
the representation is not graded-irreducible if $n\geq 2$.
\end{example}

There are also graded-irreducible strong SUR that admit proper closed invariant subspaces (but they are not graded).
Such SUR are not irreducible.

\begin{example}\label{GradedIrrReducible}
Functions in the Hilbert superspace $\ehH:=L^2(\gR^{0|1})\simeq\bigwedge \gC$
read as $\varphi:\xi\mapsto \varphi(\xi)= \varphi_0+\varphi_1\xi$,
with $\xi\in\gR^{0|1}$. We set
\begin{equation*}
\forall \eta\in\gR^{0|1},\quad (\pi(\eta) \varphi)(\xi):=(1-\eta\xi)\varphi(\xi-\eta).
\end{equation*}
One easily checks that $(\ehH,\gone,\pi)$ is a strong SUR of $\gR^{0|1}$,
whose infinitesimal representation is given by $\pi_*(t)=-t(\xi+\partial_{\xi})$.
Therefore, the space $\{\varphi_0(1+\xi)\,|\,\varphi_0\in\gC\}$ is the only proper 
$(\pio,\piG)$-invariant subspace of $L^2(\gR^{0|1})$. As it is non-graded,
 the representation is graded-irreducible but not irreducible.
\end{example}

The next three examples concern the Heisenberg Lie supergroup $\rH_{0|1,0}$,
whose super Harish-Chandra pair is $(\gR,(\kh_{0|1,0})_\gR)$. 
It is also known as the $1|1$-dimensional one-parameter Lie supergroup and
its odd derived supergroup is itself.
We start with the description of the left regular representation, which is a strong SUR.

\begin{example}
The left regular representation of $\rH_{0|1,0}$ on itself is 
the strong SUR $(\ehH, \pio,\pi)$ with $\ehH=L^2(\gR^{1|1})$,
\begin{align}
\forall \varphi\in \ehH, \forall y\in\gR,&\quad \pio(y)\varphi(x,\xi):=\varphi(x-y,\xi),\\
\forall \varphi\in \ehH^\infty, \forall (y,\eta)\in\rH_{0|1,0},&\quad 
\pi(y,\eta)\varphi(x,\xi):=\varphi(x-y-\frac{1}{2}\eta\xi,\xi-\eta).
\end{align}
Here, the space of smooth vectors $\ehH^\infty$ of $\pio$ is the space of 
smooth functions on $\gR^{1|1}$ whose all derivatives are in $L^2(\gR^{1|1})$.
The associated infinitesimal representation is given by 
$\piG(t,\tau)=-t\partial_x-\tau(\partial_\xi-\frac{1}{2}\xi\partial_x)$, 
for all $(t,\tau)\in(\kh_{0|1,0})_\gR$.
\end{example}

There exists strong SUR's which do not admit restriction to a Lie sub-supergroup.
In particular, we show below that the tensor product of the left regular representation of $\rH_{0|1,0}$
with itself is not a SUR for the diagonal embedding of $\rH_{0|1,0}$ into $\rH_{0|1,0}\times\rH_{0|1,0}$.

\begin{example}\label{Example:tensor}
The tensor product of the left regular representation of $\rH_{0|1,0}$ with itself 
turns out to be a strong SUR of $\rH_{0|1,0}\times\rH_{0|1,0}$.
Its restriction to the diagonal embedding of $\rH_{0|1,0}$ 
reads as 
\begin{align*}
\forall \varphi\in \ehH, \forall y\in\gR,&\quad \pio(y)\varphi(x_1,x_2;\xi_1,\xi_2):=\varphi(x_1-y,x_2-y;\xi_1,\xi_2),\\
\forall (t,\tau)\in(\kh_{0|1,0})_\gR,&\quad 
\piG(t,\tau):=-t(\partial_{x_1}+\partial_{x_2})
-\tau(\partial_{\xi_1}+\partial_{\xi_2})-\frac{1}{2}\tau(\xi_1\partial_{x_1}+\xi_2\partial_{x_2}).
\end{align*}
Positing $x_+=(x_1+x_2)/2$, $x_-=(x_1-x_2)/2$, $\xi_+=(\xi_1+\xi_2)/2$, 
and $\xi_-=(\xi_1-\xi_2)/2$ we get
\begin{align*}
\forall \varphi\in \ehH, \forall y\in\gR,&\quad \pio(y)\varphi(x_+,x_-;\xi_1,\xi_2):=\varphi(x_+-y,x_-;\xi_+,\xi_-),\\
\forall (t,\tau)\in(\kh_{0|1,0})_\gR,&\quad 
\piG(t,\tau):=-t\partial_{x_+}-\tau\big(\partial_{\xi_+}-\frac{1}{2}(\xi_+\partial_{x_+}+\xi_-\partial_{x_-}\big).
\end{align*}
This does not define a SUR of $\rH_{0|1,0}$.
Indeed, the space of smooth vectors on $\rH_{0|1,0}$ is the space of functions
on $\gR^{2|2}$ whose all derivatives along $\partial_{x_+}$ are square-integrable,
while the unbounded operators $\piG(t,\tau)$ are not well defined on this space in general
as they involved the derivatives $\partial_{x_-}$.
\end{example}

There are examples of representations of $\rH_{0|1,0}$ which are SUR or not depending 
on the considered Hilbert superspace structure.

\begin{example}
Let  $(\ehH,\pio)$ be the SUR of the group $\gR$ on $\ehH=L^2(\gR^{1|1})$
given by $\pio(t)=\exp(i\hbar t)\gone$ for all $t\in\gR$.
Its space of smooth vectors is $\ehH^\infty=\ehH$.
We set $\mathcal{D}:=\{\varphi_0+\xi\varphi_1'\ | \ \varphi_0,\varphi_1\in H^1(\gR)\}$, with 
$H^1(\gR)$ the Sobolev space of square integrable functions with square integrable derivative.
We define the map $\piG:(\kh_{0|1,0})_\gR\to \End(\mathcal{D})$ by setting
$$
\piG(t,\tau)\varphi:=t(i \hbar \varphi) + \tau(-\tfrac{i\hbar}{2}\varphi_1+\xi\varphi_0'),
$$
for all $\varphi=\varphi_0+\xi\varphi_1'\in\mathcal{D}$ and $(t,\tau)\in(\kh_{0|1,0})_\gR$.
The domain of $\piG(t,\tau)$ is equal to $\mathcal{D}$ and cannot be extended to $\ehH$,
so that $(\ehH,\pio,\piG)$ does not define a SUR of $\rH_{0|1,0}$ in any way.
Nevertheless, $\mathcal{D}$ is a Hilbert superspace for the inner product
$$
\langle\varphi_0+\xi\varphi_1',\psi_0+\xi\psi_1'\rangle_{\mathcal{D}}:=\int_\gR i(\varphi_0\psi_1-\varphi_1\psi_0+\varphi_0'\psi_1'-\varphi_1'\psi_0'),
$$
and one easily checks that $(\mathcal{D},\pio,\piG)$ is a strong SUR of $\rH_{0|1,0}$.
\end{example}


\section{Schr\"odinger representations of the Heisenberg supergroups}

\subsection{Construction as an induced representation in the even case}

In this section, we construct an induced strong SUR of the Heisenberg supergroup $\rH_{2m|\signp,\signq}$
via the orbit method of Kirillov, in the case $\signp+\signq$ is even.

First, we  build the smooth induced representation 
 of $\rH_{2m|\signp,\signq}$ which is associated with 
the coadjoint orbit $\caO_{\mu_0}$, passing through 
the even linear form $\mu_0=\hbar\dual{Z}$ with $\hbar\in\gR^\times$ (see \eqref{coad-action}).
By Equations \eqref{heis:normal} and \eqref{decompo:E},
we have the decomposition
\begin{equation}\label{decompo:g}
\kh_{2m|\signp,\signq}=(W\oplus W^*)\oplus V\oplus\superA Z,
\end{equation}
where $W=\superA^{m|r}$ and $V=\superA^{0|s'}$, with $r,s'\in\gN$ as in \eqref{def:rs}.
As $\signp+\signq$ is even, $s'=2s$ is also even.
A complex structure on $V$, given by $I\in\End(V)$ such that $I^2=-\gone$,
provides an eigenspace decomposition $V\otimes\gC=V_I\oplus\overline{V_I}$
with $V_I=(\gone+iI)(V\otimes\gC)$.  
We define a complex polarization $\kb$ subordinate to $\mu_0$
by setting
\begin{equation}\label{polab}
\kb:=W^*\otimes\gC\oplus \overline{V_I}\oplus\superA_\gC Z,
\end{equation}
i.e., $\kb$ is a maximal complex Lie sub-superalgebra $\kb\leq(\kh_{2m|\signp,\signq}\otimes\gC)$ 
such that $\mu_0([X,Y])=0$ for all $X,Y\in\kb$.
The real part of the polarization is given by  $\kb\cap\kh_{2m|\signp,\signq}=W^*\oplus\superA Z$
and exponentiates as the Abelian Lie supergroup $B=(W^*)_0\times\gR^{1|0}$.
The latter admits a unitary character $\chi:B\to\superA_\gC$ determined by $\mu_0$:
\begin{equation*}
\forall p\in(W^*)_0,\ \forall t\in\gR^{1|0},\quad \chi(p,t):=e^{i\mu_0(p+tZ)}=e^{i\hbar t}.
\end{equation*}
This permits to define the space of $B$-equivariant functions on $G:=\rH_{2m|\signp,\signq}$,
\begin{equation*}
\caC^\infty(G)^{B}:=\{\hat\varphi\in\caC^\infty(G)\;| \; \forall g\in G,\ 
\forall b\in B,\ \hat\varphi(gb)=\chi(b^{-1})\hat\varphi(g)\}.
\end{equation*}
The left regular $G$-action, given by $(\l(g)\hat\varphi)(g')=\hat\varphi(g^{-1} g')$ 
for all $g,g'\in G$, does not preserve the space $\caC^\infty(G)^{B}$ itself,
but it extends as a $\superA$-linear $G$-action on 
$\superA\otimes\caC^\infty(G)^{B}$.
 This representation can be restricted to the 
subspace of $\superA$-valued $\kb$-equivariant functions, that is the tensor product
of $\superA$ with the space
\begin{equation*}
\caC^\infty(G)^{B,\kb}:=\{\hat\varphi\in\caC^\infty(G)^B\;| \; \forall X\in\kb_0,\ \hat X\hat\varphi+i\mu_0(X)\hat\varphi=0\}.
\end{equation*}
Here, $\hat X$ denotes the left invariant vector field which acts
as $(\hat X\hat\varphi)(g)=\left.\frac{d}{dt}\right|_{0}\hat\varphi( g\, e^{tX})$,
for all $g\in G$.
As the functions $\hat\varphi$ are $B$-equivariant, it is equivalent 
to impose the $\kb$-equivariance condition only for $X\in \big(\overline{V_I}\big)_0$.
The resulting $G$-representation on $\superA\otimes\caC^\infty(G)^{B,\kb}$
is independent of the choice of polarization $\kb$ 
and called the {\defin $\mu_0$-induced representation}.
It admits the following explicit description.

\begin{proposition}\label{Prop:U}
Let $\hbar\in\gR^\times$ and $m,\signp,\signq\in\gN$ with $\signp+\signq$ even.
Set $\mu_0=\hbar\dual{Z}\in(\kh_{2m|\signp,\signq})^*$ and 
$\varepsilon,r,s'$  as in \eqref{def:rs}, so that $s=s'/2\in\gN$.
The $\mu_0$-induced representation is equivalent to
the $\rH_{2m|\signp,\signq}$-representation on 
$\superA\otimes\left(\caC^\infty(\gR^{m|r})\otimes \Hol(\gC^{0|s})\right)$ given by $U:g\mapsto U(g)$ with
\begin{equation}\label{eq-schro}
\left(U(g)\varphi\right)(q_0,\zeta_0):=
e^{i\hbar \Big(t+(\frac12q-q_0)p+\frac{\eps}{2}(\frac12\zeta-\zeta_0)\bzeta\Big)}
\varphi(q_0-q,\zeta_0-\zeta).
\end{equation}
Here, 
$\varphi\in\superA\otimes\left(\caC^\infty(\gR^{m|r})\otimes \Hol(\gC^{0|s})\right)$,
$(q_0,\zeta_0)\in \gR^{m|r}\oplus \gC^{0|s}$ and $g\in \rH_{2m|\signp,\signq}$ decomposes as
\begin{equation}\label{decompog}
g=(\xx,t)\in\gR^{2m|\signp+\signq}\times\gR^{0|1},
\quad \xx=(q,p,\xi)\in \gR^{m|r}\times(\gR^{m|r})^*\times\gR^{0|2s}, \quad \zeta=\xi+iI\xi\in \gC^{0|s},
\end{equation}
with $I=\left(
\begin{smallmatrix} 0 & 1\\ -1 & 0\end{smallmatrix}\right)$
on $\gR^{0|s}\oplus\gR^{0|s}$ and $\gC^{0|s}$ is identified with $(\gone+iI)(\gR^{0|2s})$.
\end{proposition}

\begin{proof}
We proceed in two steps and use same notations as above the proposition.
First we analyze the $G$-action on
$\superA\otimes\caC^\infty(G)^{B}$ and then on 
$\superA\otimes\caC^\infty(G)^{B,\kb}$.

Let $C:=W_0\oplus V_0$. The group multiplication on $G\times G$
restricts as a global diffeomorphism $C\times B \to G$.
The latter induces an isomorphism
\begin{equation}\label{iso:GBC}
\caC^\infty(C)\to\caC^\infty(G)^{B},\quad \tilde\varphi\mapsto \hat \varphi,
\end{equation}
 where $\hat \varphi(cb)=\chi(b^{-1})\tilde\varphi(c)$ for all $c\in C$, $b\in B$.
This permits to transfer the left regular action 
on $\superA\otimes\caC^\infty(G)^{B,\kb}$ into an action 
$\tilde U$ on $\superA\otimes\caC^\infty(C) $. 
Taking $c_0=(q_0,\xi_0)\in W_0\oplus V_0$
and $cb=g=(q,p,\xi,t)\in W_0\oplus(W^*)_0\oplus V_0\oplus\gR^{1|0}$,
we have the equality
$\l(cb)\hat\varphi(c_0)=\chi(c_0^{-1}cbc^{-1}c_0)\hat\varphi(c^{-1}c_0)$.
This yields
\begin{equation*}
\left(\tilde U(g)\tilde\varphi\right)(q_0,\xi_0)=e^{i\hbar\Big(t+(\frac12q-q_0)p+\frac{\eps}{2}\xi\xi_0\Big)}\tilde\varphi(q_0-q,\xi_0-\xi).
\end{equation*}

By definition of $\mu_0$, we have $\mu_0(X)=0$ if $X\in\overline{V_I}$.
Via the map \eqref{iso:GBC}, the subspace $\caC^\infty(G)^{B,\kb}$
is isomorphic to $\caC^\infty(C)^{V_I}$, defined as the subspace of functions 
$\tilde\varphi\in\caC^\infty(C)$ such that $\hat X \tilde\varphi=0$ for all $X\in (V_I)_0$.
At the point $(q,\xi)\in C\subseteq G$, one can compute that 
$\hat X= X+\frac{\varepsilon }{4}\, \om(\zeta, X)$, 
where $X$ is identified with a constant vector field on $C$ and $\zeta=\xi+iI\xi$.
Hence, we end up with an isomorphism
\begin{equation*}
\caC^\infty(W_0)\otimes \Hol((V_I)_0)\to\caC^\infty(C)^{V_I},
\quad \varphi\mapsto \tilde \varphi,
\end{equation*}
where $\tilde\varphi(q,\xi)=e^{\frac{i\eps \hbar}{4}\zeta\bzeta}\varphi(q,\zeta)$.
Now, recall that $W_0=\gR^{m|r}$, $V=\gR^{0|2s}$ and choose
$I=\left(\begin{smallmatrix} 0 & 1\\ -1 & 0\end{smallmatrix}\right)$
as complex structure on $V$. A direct computation shows that the action 
$\tilde{U}$ on $\tilde\varphi$ induces the action $U$ on $\varphi$, as given by \eqref{eq-schro}.
\end{proof}

Using Proposition \ref{Prop:U}, we define now a SUR of the classical
Heisenberg group $\gB \rH_{2m|\signp,\signq}=\rH_{2m}$. 
Decomposing an element $g\in\rH_{2m}$ as 
$g=(q,p,t)\in \gR^m\oplus(\gR^m)^*\oplus \gR$, the operator $U_0(g):=U(g)$ 
defined in Equation \eqref{eq-schro} becomes
\begin{equation}\label{Eq:U0}
\left(U_0(g)\varphi\right)(q_0,\zeta_0)=
e^{i\hbar \Big(t+(\frac12q-q_0)p\Big)}
\varphi(q_0-q,\zeta_0).
\end{equation}
This is a well-defined operator on any function $\varphi$ in the Hilbert superspace
\begin{equation}\label{HS}
\ehH_S:=L^2(\gR^{m|r})\otimes Hol(\gC^{0|s}).
\end{equation}
According to Examples \ref{ex-superhilbert} and \ref{ex-holfct} and to Proposition 
\ref{prop-superhilbert-tensprod}, $\ehH_S$ is of same parity as $r$
and its inner product takes the following form 
\begin{equation}
\langle \varphi,\psi\rangle=(2i)^s\int\dd q\dd\zeta\dd\bzeta\ \overline{\varphi(q,\zeta)}\psi(q,\zeta)e^{\frac{i\hbar\eps}{2}\zeta\bzeta}.\label{eq-innerschro}
\end{equation}
Clearly, $U_0(g)$ is a superunitary operator and the map
$
U_0:\rH_{2m}\to \caU(\ehH_S)
$ 
defines a SUR of $\rH_{2m}$. 
We determine its space of smooth vectors.
Recall that any smooth superfunction $\varphi\in\caC^\infty(\gR^{m|n})$ is determined by its components $(\varphi_\a)_{\a\in(\gZ_2)^n}$ in the decomposition \eqref{Function:Rmn} along odd variables. These are smooth functions over $\gR^m$.

\begin{definition}
The {\defin Schwartz space} $\caS(\gR^{m|n})$ is the Fr\'echet space of 
smooth superfunctions $\varphi\in\caC^\infty(\gR^{m|n})$ whose
components $(\varphi_\a)_{\a\in(\gZ_2)^n}$ are Schwartz functions
on $\gR^m$. Its topology is provided by the seminorms
\begin{equation*}
\norm \varphi\norm_{\nu,\nu'}:=\sum_{\a\in(\gZ_2)^n}\sup_{x\in\gR^m}|x^\nu\partial^{\nu'} \varphi_\a(x)|,
\end{equation*}
indexed by multi-indices $\nu,\nu'\in\gN^m$.
\end{definition}


\begin{proposition}
\label{prop-smoothschwartz}
The space $\ehH_S^\infty$ of smooth vectors of the SUR
$(\ehH_S,U_0)$ of $\rH_{2m}$ is the Schwartz space $\caS(\gR^{m|r+s})\simeq
\caS(\gR^{m|r})\otimes Hol(\gC^{0|s})$.
\end{proposition}
\begin{proof}
In view of the action \eqref{Eq:U0} of $\rH_{2m}$, the smooth vectors are the functions
$\varphi\in\ehH_S$ whose components $\varphi_\a\in L^2(\gR^m)$ are smooth vectors for the
usual Schr\"odinger representation of $\rH_{2m}$, i.e. Schwartz functions 
(see e.g. \cite{Reed:1980}). 
\end{proof}

We go back to the full Heisenberg supergroup $\rH_{2m|\signp,\signq}$.
One easily shows that the operator $U(g)$, defined by Equation \eqref{eq-schro},
preserves $\superA\otimes\ehH_S^\infty$ and is superunitary. Hence, we get
a group morphism $U:\rH_{2m|\signp,\signq}\to\caU_\superA(\superA\otimes\ehH_S^\infty)$.
Note that the odd derived supergroup satisfies 
$\D(\rH_{2m|\signp,\signq})\simeq\rH_{0|\signp,\signq}$, so that its body is 
the center group $Z(\rH_{2m|\signp,\signq})=\gR^{1|0}$.

\begin{theorem}
\label{thm-suprepheis}
Let $\hbar\in\gR^\times$ and $m,\signp,\signq\in\gN$ with $\signp+\signq$ even.
Set $\varepsilon,r,s'$  as in \eqref{def:rs}. The Equations \eqref{HS}, \eqref{Eq:U0} and \eqref{eq-schro}
define a strong SUR $(\ehH_S,U_0,U)$ of the Heisenberg supergroup $\rH_{2m|\signp,\signq}$, 
called the {\defin Schr\"odinger representation} with parameter $\hbar$.
\end{theorem}
\begin{proof}
We already know that $(\ehH_S,U_0)$ is a SUR of $\rH_{2m}$, 
$U$ is a group morphism $U:\rH_{2m|\signp,\signq}\to\caU_\superA(\superA\otimes\ehH_S^\infty)$
and by definition  $U(g)\varphi=U_0(g)\varphi$ if $g\in\rH_{2m}$ and $\varphi\in\ehH_S^\infty$.
Moreover, by Equation \eqref{eq-schro}, the function $g\mapsto U(g)\varphi$ is smooth on 
all of $\rH_{2m|\signp,\signq}$ if $\varphi\in\ehH_S^\infty$. 
Hence, the triple $(\ehH_S,U_0,U)$ is a SUR of $\rH_{2m|\signp,\signq}$. 

The center $Z(\rH_{2m|\signp,\signq})=\gR^{1|0}$ acts by $t\mapsto e^{i\hbar t}\gone$.
Therefore, the space of smooth vectors on the center is $\ehH_S$ itself.
Using Equation \eqref{eq-schro},
we deduce that the operator $U(g)$ is well-defined on all of $\ehH_S$ if $g\in \D(\rH_{2m|\signp,\signq})\simeq\rH_{0|p,q}$.
Hence, the map $U$ restricts as a group morphism
$U^{\D(\rH)}:\D(\rH_{2m|\signp,\signq})\to\caU_\superA(\superA\otimes\ehH_S)$.
Using again Equation \eqref{eq-schro}, one proves easily that 
$g\mapsto U^{\D(\rH)}(g)\varphi$ is a smooth function on 
$\D(\rH_{2m|\signp,\signq})$ if $\varphi\in\ehH_S$. As a result, the triple 
$(\ehH_S,U_0,U)$ is a strong SUR of $\rH_{2m|\signp,\signq}$.
\end{proof}

By Lemma \ref{Lem:expandpi}, the Schr\"odinger representation generates an infinitesimal representation 
$U_*:=dU:\kh_{2m|\signp,\signq}\to\End(\superA\otimes\ehH_S^\infty)$. 
In terms of the decomposition \eqref{decompo:g}, we have 
$\ehH_S^\infty\subseteq\caC^\infty(W_0)\otimes Hol\big((V_I)_0\big)$
and
\begin{equation*}
\forall\varphi\in\ehH_S^\infty,\quad U_*(X)\varphi=
\begin{cases}X\varphi, \quad &\text{if}\; X\in W\otimes\gC\oplus V_I,\\
i\hbar\om(X,-)\varphi, 
\quad &\text{if}\; X\in W^*\otimes\gC\oplus \overline{V_I},\\
i\hbar\varphi,\quad &\text{if}\; X=Z,
\end{cases}
\end{equation*}
where $X$ acts as a constant vector field in the first line and 
$\om(X,-)\in\superA\otimes\caC^\infty(W_0)\otimes Hol\big((V_I)_0\big)$. 
More precisely, we have $\om(X,-):(q_0,\zeta_0)\mapsto\om(X,(q_0,0,\zeta_0,0))$, with 
$(q_0,0,\zeta_0,0)\in(W\oplus W^*)\oplus (V_I\oplus\overline{V_I})$
and $\om$ the $\gC$-linear extension of the defining symplectic form 
of $\kh_{2m|\signp,\signq}$.

The following result will be useful.

\begin{proposition}\label{Prop:pibounded}
Let $(\ehH,\pio,\piG)$ be a strong SUR of $\rH_{2m|\signp,\signq}$ such that $\pio(0,t)=e^{i\hbar t}\gone$
for all $(0,t)\in Z(\rH_{2m|\signp,\signq})$. Then, we have
\begin{itemize}
\item the operators $\piG(X)$ are bounded on $\ehH$ for all $X\in\gB(\kh_{2m|\signp,\signq})_1$,
\item $(\ehH,\pio,\piG)$ restricts to any Lie sub-supergroup of $\rH_{2m|\signp,\signq}$ as a strong SUR,
\item the associated supergroup representation $\pi:\rH_{2m|\signp,\signq}\to\caU_\superA(\superA\otimes\ehH)$ satisfies
$$
\pi(\xx,t)=e^{i\hbar t}\,\widetilde{\pio}(x^\ev)\sum_{\g\in(\gZ_2)^{\signp+\signq}}(x^\od)^\g\, \pi_\g,
$$
where $\xx=x^\ev+x^\od$, with $x^\ev\in\gR^{2m|0}$ and $x^\od\in\gR^{0|\signp+\signq}$, $\widetilde{\pio}(x^\ev)$ is defined in Lemma
\ref{lem:pitopio} and for all $\g\in(\gZ_2)^{\signp+\signq}$
we have $\pi_\g\in\caB(\ehH)$ and $\pi_\g(\ehH^\infty)\subseteq\ehH^\infty$.
\end{itemize}
\end{proposition}
\begin{proof}
The body of the odd derived supergroup $\D(\rH_{2m|\signp,\signq})$ 
is the center of $\rH_{2m|\signp,\signq}$.
By hypothesis, the space of smooth vectors on $Z(\rH_{2m|\signp,\signq})$ is the whole space $\ehH$.
Therefore, Propositions \ref{rmk-bounded} and \ref{prop:restriction-strong}
apply. We deduce that $\piG(X)\in\caB(\ehH)$ 
for all $X\in\gB(\kh_{2m|\signp,\signq})_1$ and $(\ehH,\pio,\piG)$ restricts to 
any Lie sub-supergroup of $\rH_{2m|\signp,\signq}$ as a strong SUR.
By Equation \eqref{piopiH:pi}, the couple $(\pio,\piG)$ defines a supergroup representation
given by
$$
\pi(\xx,t)=e^{i\hbar t}\,\widetilde{\pio}(x^\ev)\exp(\piG(x^\od)).
$$
Choosing a basis $(e_i)$ of $\gB(\kh_{2m|\signp,\signq})_1\cong\gR^{\signp+\signq}$,
we get a decomposition of $x^\od\in\gR^{0|\signp+\signq}$ as 
$x^\od=\sum_{i=1}^{\signp+\signq}(x^\od)^ie_i$, with $(x^\od)^i\in\caA_1$.
Since $((x^\od)^i\pi_*(e_i))_i$ is a family of bounded, nilpotent and commutative operators,
the exponential series admits a finite expansion:
$$
\exp(\piG(x^\od))=\sum_{\g\in(\gZ_2)^{\signp+\signq}}(x^\od)^\g\, \pi_\g,
$$
where, for all $\g\in(\gZ_2)^{\signp+\signq}$,
 $\pi_\g=\pm\pi_*(e_1)^{\g_1}\ldots\pi_*(e_{\signp+\signq})^{\g_{\signp+\signq}}$ 
are bounded operators which satisfy $\pi_\g(\ehH^\infty)\subseteq\ehH^\infty$, 
just as the operators $\pi_*(e_i)$.
This concludes the proof.
\end{proof}

We write explicitly the Schr\"odinger representation for $(\signp,\signq)=(1,1)$ and $(2,0)$ or $(0,2)$.

\begin{example}
If $(\signp,\signq)=(1,1)$, then we have the decomposition $\rH_{0|1,1}=W_0\oplus(W^*)_0\oplus\gR^{1|0}$, with $W_0\simeq (W^*)_0\simeq \gR^{0|1}$. The Hilbert superspace is $\ehH_S:=L^2(\gR^{0|1})$ of parity $1$ 
and its elements $\varphi$ can be decomposed as $\varphi(q_0):=\varphi_0+\varphi_1q_0$, with $q_0\in W_0$. 
The inner product has the form
\begin{equation*}
\langle\varphi,\psi\rangle=\overline{\varphi_0}\psi_1+\overline{\varphi_1}\psi_0,
\end{equation*}
and a fundamental symmetry is provided by $J\varphi(q_0)=\varphi_1+\varphi_0q_0$. 
The Schr\"odinger representation reads as $U(q,p,t)=e^{i\hbar t}(\gone+q\, U_{10}+p\, U_{01}+qp\, U_{11})$, for $q\in W_0$, $p\in(W^*)_0$, $t\in\gR^{1|0}$, with
\begin{equation*}
U_{10}:=-\partial_{q_0}=-U_{10}^\dag,\qquad U_{01}:=i\hbar q_0, 
\qquad U_{11}:=\frac{i\hbar}{2}(1-2q_0\partial_{q_0}).
\end{equation*}
The infinitesimal representation satisfies $U_*(q,p,t)=q\, U_{10}+p\, U_{01}+i\hbar t\gone$.
\end{example}

\begin{example}
If $(\signp,\signq)=(2,0)$ or $(0,2)$, then we have the decomposition 
$\rH_{0|\signp,\signq}=V\oplus\gR^{1|0}$, with $V\simeq \gR^{0|2}$. 
The Hilbert superspace is $\ehH_S:=Hol(\gC^{0|1})$ and its elements 
can be decomposed as $\varphi(\zeta_0):=\varphi_0+\varphi_1\zeta_0$,
where $\zeta_0=\xi_0+iI\xi_0\in V_I\simeq\gC^{0|1}$. 
The inner product has the form
\begin{equation*}
\langle\varphi,\psi\rangle=\eps \hbar\overline{\varphi_0}\psi_0+2i\overline{\varphi_1}\psi_1,
\end{equation*}
where $\varepsilon=1$ if $(\signp,\signq)=(2,0)$ and $\varepsilon=-1$ if $(\signp,\signq)=(0,2)$.
 A fundamental symmetry is provided by $J\varphi(\zeta_0)=\eps\varphi_0-i\varphi_1\zeta_0$. 
Moreover, $\ehH_S$ is a Hilbert superspace of parity 0, with $\text{sgn}(\ehH_S)=(1,1,0,0)$ if $\hbar\eps>0$
and $\text{sgn}(\ehH_S)=(0,1,1,0)$ if $\hbar\eps<0$. 
The representation reads as $U(\zeta,\bzeta,t)=e^{i\hbar t}(\gone+\zeta\, U_{10}
+\bzeta\, U_{01}+\zeta\bzeta\, U_{11})$, with
\begin{equation*}
U_{10}:=-\partial_{\zeta_0}=-U_{01}^\dag,\qquad U_{01}:=\frac{i\hbar\eps}{2}\zeta_0,\qquad U_{11}:=\frac{i\hbar\eps}{4}(1-2\zeta_0\partial_{\zeta_0}).
\end{equation*}
The infinitesimal representation satisfies $U_*(\zeta,\bzeta,t)=\zeta\, U_{10}
+\bzeta\, U_{01}+i\hbar t\gone$.
\end{example}

\subsection{Definition in the odd case}
\label{subsec-oddalg}

In this section, we consider the Heisenberg supergroup $\rH_{2m|\signp,\signq}$ 
with the odd dimension $\signp+\signq$ which is odd. 
The notations of Equation \eqref{decompo:g} for the decomposition of 
the Heisenberg Lie superalgebra $\kh_{2m|\signp,\signq}$ should be adapted as follows
$$
\kh_{2m|\signp,\signq}=(W\oplus W^*)\oplus V\oplus V'\oplus\superA Z,
$$
with $W=\superA^{m|r}$, $V=\superA^{0|2s}$ and  $V'=\superA^{0|1}$, where the integers $r,s'=2s+1$ and $\eps$ are as in \eqref{def:rs}.

Performing the Orbit Method of Kirillov with respect to the same complex polarization $\kb$ (see \eqref{polab}), 
we obtain a superunitary representation map
\begin{equation}
U(g)\varphi(q_0,\zeta_0,\tau_0)=e^{i\hbar\Big(t+(\frac12q-q_0)p+\frac{\eps}{2}(\frac12\zeta-\zeta_0)\bzeta+\frac{\eps}{2}\tau\tau_0\Big)}\varphi(q_0-q,\zeta_0-\zeta,\tau_0-\tau),\label{eq-schro1}
\end{equation}
where $g=(q,p,\xi,\tau,t)\in W_0\oplus(W^*)_0\oplus V\oplus V'\oplus\gR^{1|0}$ 
and $\zeta=\xi+iI\xi\in V_I$ as in Proposition \ref{Prop:U}. The extra variable $\tau$
can be considered either as a real or as a complex variable. Accordingly, the Hilbert superspace
of this representation is either $L^2(\gR^{m|r+1})\otimes Hol(\gC^{0|s})$, of parity $r+1$, 
or $L^2(\gR^{m|r})\otimes Hol(\gC^{0|s+1})$, of parity $r$.
In each case the obtained representation is a strong SUR. If $\tau$ is real and $\eps=1$ (resp. $\eps=-1$), 
this is the restriction of the Schr\"odinger representation of the supergroup $\rH_{2m|\signp+1,\signq}$ 
(resp. $\rH_{2m|\signp,\signq+1}$) to $\rH_{2m|\signp,\signq}$.
If $\tau$ is complex  and $\eps=1$ (resp. $\eps=-1$), this is the restriction of the Schr\"odinger representation of the supergroup 
$\rH_{2m|\signp,\signq+1}$ (resp. $\rH_{2m|\signp+1,\signq}$) to $\rH_{2m|\signp,\signq}$. 
Accordingly, we introduce the following definition.

\begin{definition}\label{Def:SchrodingerRep}
Let $\hbar\in\gR^\times$ and $m,\signp,\signq\in\gN$ with $\signp+\signq$ odd. 
\begin{itemize}
\item 
The {\defin Schr\"odinger representation} of $\rH_{2m|\signp,\signq}$ 
with parameter $\hbar$ and parity $0$ 
is the restriction to $\rH_{2m|\signp,\signq}$ 
of the Schr\"odinger representation  with same parameter of $\rH_{2m|\signp,\signq+1}$  if $\max(k,\ell+1)$ is even
and of $\rH_{2m|\signp+1,\signq}$ otherwise.
\item 
The {\defin Schr\"odinger representation} of $\rH_{2m|\signp,\signq}$ 
with parameter $\hbar$ and parity $1$ 
is the restriction to $\rH_{2m|\signp,\signq}$ of the Schr\"odinger representation 
 with same parameter of $\rH_{2m|\signp+1,\signq}$ if $\max(k,\ell+1)$ is odd
and of $\rH_{2m|\signp,\signq+1}$ otherwise.
\end{itemize}
\end{definition}

We write explicitly the Schr\"odinger representation for $(\signp,\signq)=(1,0)$ or $(0,1)$.

\begin{example}
\label{ex-oddschro}
If $(\signp,\signq)=(1,0)$ or $(0,1)$, then we have the decomposition $\rH_{0|\signp,\signq}=V'\oplus\gR^{0|1}$ with $V'=\gR^{0|1}$ 
and $\eps=1$ or $-1$ respectively. The elements of the representation superspace $\ehH_S$ decompose as $\varphi(\tau_0):=\varphi_0+\varphi_1\tau_0$ and the representation map reads as $U(\tau,t)=e^{i\hbar t}(\gone+\tau \, U_{1})$, with
\begin{equation*}
U_{1}:=-\partial_{\tau_0}+\frac{i\hbar}{2}\eps \tau_0.
\end{equation*}
It is superunitary with respect to two inner products of opposite parity $\sigma$. 
If $\sigma=0$, the inner product has the form
\begin{equation*}
\langle\varphi,\psi\rangle=\eps\hbar \overline{\varphi_0}\psi_0+2i\overline{\varphi_1}\psi_1,
\end{equation*}
so that $\ehH_S=Hol(\gC^{0|1})$ is of parity $0$.  If $\sigma=1$, the inner product has the form
\begin{equation*}
\langle\varphi,\psi\rangle= \overline{\varphi_0}\psi_1+\overline{\varphi_1}\psi_0,
\end{equation*}
so that $\ehH_S=L^2(\gR^{0|1})$  is of parity $1$.
In both parity, $\ehH_S$ admits exactly two $(U_0,U)$-invariant subspaces given by 
$\{(1\pm\sqrt{\frac{-i\hbar\eps}{2}}\tau_0)\varphi_0 \; | \; \varphi_0\in\gC\}$. They are not graded. Hence, in both parity the Schr\"odinger representation $(\ehH_S,U_0,U)$ is graded-irreducible but not irreducible.
\end{example}

\subsection{Algebraic properties}
\label{subsec-spinor}

In this section, we show that the Schr\"odinger representation space $\ehH_S$ of the 
Heisenberg supergroup $G=\rH_{0|\signp,\signq}$, with $\signp+\signq=2n$,
carries a natural action of the complex Clifford algebra $\mathbb{C}\ell(2n)$
and identifies with its spinor module. Then, we prove
that the Schr\"odinger representations of any Heisenberg supergroup $\rH_{2m|\signp,\signq}$ are irreducible if $\signp+\signq$ is even and graded-irreducible if $\signp+\signq$ is odd.

The infinitesimal Schr\"odinger representation of $\kg_\gR:=\gB\kh_{0|\signp,\signq}$,
with parameter $\hbar\in\gR^\times$ and parity $\sigma\in\gZ_2$, induces algebra morphisms
\begin{align*}
U_*&:\kU(\kg_\gR)\rightarrow\End_\gC(\ehH_S),\\
U_*\otimes\gone_\gC&:\kU(\kg_\gR\otimes\gC)\rightarrow\End_\gC(\ehH_S),
\end{align*}
where $\kU$ stands for the universal enveloping algebra. 

\begin{proposition}\label{Schrodinger-Clifford}
Let $\signp+\signq\in 2\gN$. The algebra morphisms $U_*$ and $U_*\otimes\gone_\gC$ are surjective, 
and induce the following algebra isomorphisms 
$$
\forall \hbar\in\gR^\times,\quad\kU(\kg_\gR)/\langle Z^2+\hbar^21_\kU\rangle\simeq\End_\gC(\ehH_S)\simeq\kU(\kg_\gR\otimes\gC)/\langle Z-i\hbar1_\kU\rangle\simeq \mathbb{C}\ell(\signp+\signq),
$$
where $Z\in Z(\kg_\gR)$ is the generator of the center of $\kg_\gR$ and $1_\kU$ is the unit of $\kU(\kg_\gR)$ and $\kU(\kg_\gR\otimes\gC)$.
\end{proposition}
\begin{proof}
The isomorphism $\kU(\kg_\gR)/\langle Z^2+\hbar^21_\kU\rangle\simeq
\kU(\kg_\gR\otimes\gC)/\langle Z-i\hbar1_\kU\rangle$ is obtained by mapping $Z$ to $i\hbar1_\kU$.
The two other isomorphisms are obtained via the map $U_*\otimes  \gone_\gC$ as follows.
Since $U_*(Z)=i\hbar\gone_{\ehH_S}$, the kernel of $U_*\otimes  \gone_\gC$
contains the ideal $\langle Z-i\hbar\gone\rangle$.
By the Heisenberg commutation relation \eqref{LieHeisenberg}, 
the Clifford relations 
$$
U_*(X)U_*(Y)+U_*(Y)U_*(X)=i\hbar\om(X,Y)
$$
hold for any $X,Y\in \kg_\gR\otimes\gC/\langle Z\rangle$.
By universal property of $\mathbb{C}\ell(\signp+\signq)$, the image of 
$U_*\otimes  \gone_\gC$ is isomorphic to $\mathbb{C}\ell(\signp+\signq)$. 
This means there exists a surjection from $\kU(\kg_\gR\otimes\gC)/\langle Z-i\hbar1_\kU\rangle$
to $\mathbb{C}\ell(\signp+\signq)$ and an injection from $\mathbb{C}\ell(\signp+\signq)$
to $\End(\ehH_S)$.
The result follows then by dimension counting.
\end{proof}

We still assume that $\signp+\signq=2n$ is even. 
As in \eqref{def:rs}, we consider $\varepsilon\in\{-1,+1\}$
and $(r,s)\in\gN^2$ such that $\varepsilon=1$ and $(\signp,\signq)=(r+2s,r)$, or, $\varepsilon=-1$ and 
$(\signp,\signq)=(r,r+2s)$. 
In both cases, the Schr\"odinger representation space of $\rH_{0|\signp,\signq}$
is $\ehH_S=L^2(\gR^{0|r})\otimes Hol(\gC^{0|s})$, endowed with 
the inner product \eqref{eq-innerschro} which depends on $\varepsilon$ and the parameter $\hbar\in\gR^\times$.
As a consequence of Proposition~\ref{Schrodinger-Clifford}, 
the Hilbert superspace $\ehH_S$ satisfies $\End(\ehH_S)\simeq\mathbb{C}\ell(2n)$,
i.e.\ $\ehH_S$ is the spinor module of~$\mathbb{C}\ell(2n)$. 
There are four classes of isomorphisms of such Hilbert superspaces $\ehH_S$,
depending on the signs of $\signp\signq$ and $\eps\hbar$ :
\begin{itemize}
\item if $\signp\signq=0$ and $\hbar\eps>0$, then $\text{sgn}(\ehH_S)=(2^{n-1},2^{n-1},0,0)$;
\item if $\signp\signq=0$ and $\hbar\eps<0$, then $\text{sgn}(\ehH_S)=(2^{n-1},0,0,2^{n-1})$
if $n$ is even and $\text{sgn}(\ehH_S)=(0,2^{n-1},2^{n-1},0)$ if $n$ is odd;
\item if $\signp\signq>0$ and $\signp$ odd, then $\ehH_S$ is of parity $1$;
\item if $\signp\signq>0$ and $\signp$ even, then $\text{sgn}(\ehH_S)=(2^{n-2},2^{n-2},2^{n-2},2^{n-2})$.
\end{itemize}

\begin{remark}
The superhermitian product on $\ehH_S$ is symmetric on $(\ehH_S)_0$ and antisymmetric 
on $(\ehH_S)_1$. Therefore, this inner product is in general not a bilinear invariant of 
the spinor representation (as classified in \cite{Alekseevsky:1997}), since it is neither 
symmetric nor antisymmetric.
\end{remark}

Recall that the center of the Heisenberg supergroup is 
$Z(\rH_{2m|\signp,\signq})=\gR^{1|0}$. Consider the antidiagonal
subgroup 
\begin{equation}
K=\{(t,-t)\in Z(\rH_{2m|\signp,\signq})\times Z(\rH_{2m'|\signp',\signq'})\ | \ t\in\gR^{1|0}\}.\label{eq-antidiag}
\end{equation}
Then, we have the group isomorphism
$$
(\rH_{2m|\signp,\signq}\times \rH_{2m'|\signp',\signq'})/K\simeq \rH_{2(m+m')|\signp+\signp',\signq+\signq'}.
$$
Therefore, if a strong SUR of the direct product $\rH_{2m|\signp,\signq}\times \rH_{2m'|\signp',\signq'}$
has kernel $K$, it descends to a strong SUR of $\rH_{2(m+m')|\signp+\signp',\signq+\signq'}$.

\begin{proposition}
\label{prop-tensprodheis}
The tensor product of the Schr\"odinger representations of $\rH_{2m|\signp,\signq}$
and $\rH_{2m'|\signp',\signq'}$, with same parameter $\hbar\in\gR^\times$ and parities $\sigma$ and $\sigma'$ repsectively, is a strong SUR. It descends as a strong SUR of $\rH_{2(m+m')|\signp+\signp',\signq+\signq'}$,
 isomorphic to its Schr\"odinger representation with same parameter $\hbar$ and parity $\sigma+\sigma'$. 
\end{proposition}
\begin{proof}
According to Proposition \ref{prop-tensprodSUR}, the tensor product of
any two Schr\"odinger representations of $\rH_{2m|\signp,\signq}$
and $\rH_{2m'|\signp',\signq'}$ is a strong SUR
of the direct product $\rH_{2m|\signp,\signq}\times \rH_{2m'|\signp',\signq'}$. 
If both Schr\"odinger representations have the same parameter $\hbar$, the centers
of $\rH_{2m|\signp,\signq}$ and $\rH_{2m'|\signp',\signq'}$ act  in the same way and the 
representation descends to the quotient group $\rH_{2(m+m')|\signp+\signp',\signq+\signq'}$.
A direct computation shows that the tensor product representation is indeed
the Schr\"odinger representation for the latter supergroup, with parameter $\hbar$ and parity $\sigma+\sigma'$.
\end{proof}

As a particular case of the above proposition, the Schr\"odinger representations of $\rH_{2m|\signp,\signq}$
are tensor products of Schr\"odinger representations of $\rH_{2m}$
and $\rH_{0|\signp,\signq}$.
The first one is the classical Schr\"odinger representation of $\rH_{2m}$ on $L^2(\gR^{m})$, see e.g. \cite{Folland:1989}. By Proposition \ref{Schrodinger-Clifford}, if $\signp+\signq$ is even,
the second one is the restriction to $\kh_{0|\signp,\signq}$ of the classical spinor representation of $\mathbb{C}\ell(\signp+\signq)$.
The spinor module is denoted by $S_\gC$ in the following.

\begin{theorem}
\label{thm-irreduct}
Let $\signp+\signq\in 2\gN$. The Schr\"odinger representation of $\rH_{2m|\signp,\signq}$ with parameter $\hbar\in\gR^\times$ is  irreducible.
\end{theorem}
\begin{proof}
Let $(\ehH_S,U_0,U)$ be the Schr\"odinger representation of $\rH_{2m|\signp,\signq}$.
Consider $\varphi\in\ehH_S\setminus\algzero$ and $\ehF$ the smallest
closed subspace of $\ehH_S$ containing $\varphi$ and stable under
$U_0$ and $U_*$. By definition of $\ehH_S$, we have $\varphi=
\sum_\a\varphi_\a \otimes s^\a$ with $\varphi_\a\in L^2(\gR^{2m})$ and $(s^\a)_\a$
a basis of $S_\gC$. By Proposition \ref{Schrodinger-Clifford}, stability under $U_*$
means that $\ehF$ is stable under $\End(S_\gC)$. Hence, for all index $\a$ and $s\in S_\gC$, 
we have $\varphi_\a \otimes s\in\ehF$. The representation $U_0$, restricted to 
$L^2(\gR^{m})\otimes\langle s\rangle$, coincides with the Schr\"odinger representation 
of $\rH_{2m}$, which is irreducible. Considering $\a$ such that $\varphi_\a\neq 0$,
we obtain that $L^2(\gR^{m})\otimes\langle s\rangle\subseteq\ehF$ for all $s\in S_\gC$. 
Hence, $\ehF=\ehH_S$ and the result follows.
\end{proof}

\begin{corollary}
\label{cor-irreduct}
Let $\signp+\signq\in 2\gN+1$. The Schr\"odinger representation of $\rH_{2m|\signp,\signq}$ with parameter $\hbar\in\gR^\times$ and parity $\sigma\in\gZ_2$ is  graded-irreducible.
\end{corollary}
\begin{proof}
We assume that $\signp\neq 0$.
According to Example \ref{ex-oddschro}, Schr\"odinger representations of $\rH_{0|1,0}$ of both parity are graded-irreducible. Moreover, by Theorem \ref{thm-irreduct}, the Schr\"odinger representation of $\rH_{2m|\signp-1,\signq}$ is irreducible. 
Hence, its tensor product with any of the two Schr\"odinger representations of $\rH_{0|1,0}$ 
is a graded-irreducible strong SUR of $\rH_{2m|\signp-1,\signq}\times \rH_{0|1,0}$. 
By Proposition \ref{prop-tensprodheis}, both descend as the Schr\"odinger representations of $\rH_{2m|\signp,\signq}$, 
of parity $0$ and $1$, which are then graded-irreducible.

If $\signp= 0$, Schr\"odinger representations of $\rH_{2m|\signp,\signq}$ are obtained as tensor products of Schr\"odinger representations of $\rH_{0|0,1}$ 
and $\rH_{2m|\signp,\signq-1}$ and the conclusion is the same.
\end{proof}

\subsection{Analytic properties}\label{Sec:AnalyticProp}

We focus now on analytic properties of the Schr\"odinger representation $(\ehH_S,U_0,U)$
of the Heisenberg supergroup $\rH_{2m|\signp,\signq}$ with $\signp+\signq$ even.

Using the manifold decomposition $\rH_{2m|\signp,\signq}=E_0\times \gR^{1|0}$,
with $E_0=\gR^{2m|\signp+\signq}$,
we set $U(\xx):=U(\xx,0)$ for all $\xx\in E_0$.
According to the group law \eqref{GroupLaw}, we have
\begin{equation}\label{Eq:UU}
\forall \xx,\xx'\in E_0,\quad U(\xx)U(\xx')=e^{\frac{i\hbar}{2}\omega(\xx,\xx')}U(\xx+\xx').
\end{equation}
Hence, $\xx\mapsto U(\xx)$ defines a projective representation of the Abelian 
Lie supergroup $E_0$. 

We recall notation in \eqref{decompog}. Set $r,s',\eps$ as in \eqref{def:rs} and $s=\frac{s'}{2}\in\gN$. Each element $\xx\in E_0$ can be decomposed as $\xx=(q,p,\xi)$
with $q=(q^1,\ldots,q^{m+r})\in\gR^{m|r}$, $p=(p^1,\ldots,p^{m+r})\in\gR^{m|r}$
and $\xi=(\xi^1,\ldots,\xi^{2s})\in\gR^{0|2s}$.
Further, $\zeta=(\zeta^1,\ldots,\zeta^{s})\in\gC^{0|s}$  is defined by 
$\zeta^\a=\xi^\a+i\xi^{\a+s}$  for all $\a=1,\ldots,s$.
Then, the invariant measure on $E_0$ is chosen as
\begin{equation}\label{Eq:ddx}
\dd \xx:=(-1)^{\frac{s(s-1)}{2}}\dd q\dd p\dd\xi = (2i)^s\dd q\dd p\dd\zeta\dd\bzeta,
\end{equation}
where $\dd q\dd p:=\dd q^{m+r}\cdots \dd q^1\dd p^{m+r}\cdots \dd p^1$, 
$\dd\xi:=\dd\xi^{2s}\cdots\dd\xi^{1}$ and 
$\dd\zeta\dd\bzeta:=\dd\zeta^s\dd\bzeta^s\cdots\dd\zeta^1\dd\bzeta^1$. 
The corresponding Berezin integration, defined in \eqref{eq-berezin}, reads as 
$$\int\dd\xx \ f(\xx):=
(2i)^s\int_{\gR^{2m}}\dd q^{m}\dd p^{m}
\cdots \dd q^1 \dd p^1\ \left(\partial_{q^{m+r}}\cdots \partial_{q^{m+1}}
\partial_{p^{m+r}}\cdots\partial_{p^{m+1}}\partial_{\zeta^s}\partial_{\bzeta^s} 
\cdots\partial_{\zeta^1}\partial_{\bzeta^1} \, f(\xx)\right)$$
with $f\in L^1(E_0)$.
The next Lemma is obtained by direct computations.
\begin{lemma}\label{lem:integrals}
Let $\kappa_\hbar=\left( \frac{2\pi}{\hbar}\right)^m(i\hbar)^r (\eps\hbar)^s(-1)^{\frac{r(r+1)}{2}}$.
The following identities hold
\begin{align}
\forall \varphi\in\caS(\gR^{m|r}),\quad\int\dd q\dd p\ e^{i\hbar qp}\varphi(q)&=\frac{\kappa_\hbar}{(\eps\hbar)^s}\varphi(0)\label{eq-dirac0},\\
\forall \varphi\in Hol(\gC^{0|s}),\quad (2i)^s\int\dd\zeta\dd\bzeta\ e^{\frac{i\hbar\eps}{2}\zeta\bzeta} \varphi(\zeta)&=(\hbar\eps)^s\varphi(0),\label{eq-dirac1}\\
\forall f\in \caS(\modE_0),\quad \int \dd \xx\dd \yy\ e^{i\hbar\omega(\xx,\yy)}f(\xx)&=\kappa_\hbar^2 f(0)\label{eq-dirac2}.
\end{align}
\end{lemma}

For $\varphi,\psi\in\ehH_S=L^2(\gR^{m|r})\otimes Hol(\gC^{0|s})$, 
with $\varphi$ or $\psi$ in $\ehH_S^\infty=\caS(\gR^{m|r})\otimes Hol(\gC^{0|s})$, 
the associated {\defin Wigner function} on $E_0$ is defined as
\begin{equation}
V\big(\varphi,\psi\big)(\xx):=\langle \varphi,U(\xx)\psi\rangle,
\label{eq-wigner}
\end{equation}
where $\langle-,-\rangle$ is the inner product of $\ehH_S$ defined in~\eqref{eq-innerschro}.
We derive several results about Wigner functions.
The next Lemma concerns their regularity.

\begin{lemma}
\label{lem-gaussian}
Let $\varphi,\psi\in\ehH_S^\infty$. Then, the Wigner function $V(\varphi,\psi)$ is in $L^1(\modE_0)\cap L^2(\modE_0)$.
\end{lemma}
\begin{proof}
According to \eqref{Function:Rmn}, any function $\varphi\in\ehH_S^\infty$ 
can be decomposed along odd variables:
\begin{equation*}
\varphi(q,\zeta):=\sum_{\a\in(\gZ_2)^{r+s}}\widetilde{\varphi_\a}(q^\ev)
\Upsilon^\a, 
\end{equation*}
where $q^\ev=(q^1,\dots,q^m)$ denote the even variables of $q$
and $\Upsilon:=(q^{m+1},\dots,q^{m+r},\zeta^1,\dots,\zeta^s)$ 
the odd variables of $(q,\zeta)$. That is $\Upsilon^\a=(q_{m+1})^{\a_{1}}\cdots(q_{m+r})^{\a_{r}}(\zeta_{1})^{\a_{r+1}}\cdots(\zeta_{s})^{\a_{r+s}}$. Each $\xx\in\modE_0$ 
decomposes as $\xx=x^\ev+x^\od$, with $x^\ev\in\gR^{2m|0}$ 
and $x^\od\in\gR^{0|\signp+\signq}$. By Proposition \ref{Prop:pibounded}, we have
\begin{equation*}
U(\xx)=\widetilde{U_0}(x^\ev)\,\sum_{\gamma\in(\gZ_2)^{2(r+s)}} (x^\od)^\gamma \, U_\gamma,
\end{equation*}
where $U_\gamma$ is an operator preserving $\ehH_S^\infty$. If $\psi\in\ehH_S^\infty$, the superfunction $U_\gamma\psi$ is then smooth and decomposes along odd variables as $U_\gamma\psi(q,\zeta) = \sum_{\a\in(\gZ_2)^{2(r+s)}}\widetilde{(U_\gamma\psi)_\a}(q^\ev)\Upsilon^\a$. After integration of the odd variables $\Upsilon$, the Wigner function reads then as
\begin{equation*}
V\big(\varphi,\psi\big)(\xx)=\sum_{\a,\gamma\in(\gZ_2)^{2(r+s)}} 
\varepsilon(\a)\,(x^\od)^\gamma\, \widetilde{\big(V_0(\varphi_\a,(U_\gamma\psi)_{\bar{\a}})\big)}(x^\ev),
\end{equation*}
where $\varepsilon(\a)=\pm 1$, $\bar{\a}=(1,\ldots,1)-\a\in(\gZ_2)^{2(r+s)}$ and  
$$
V_0\big(\varphi_\a,(U_\gamma\psi)_{\bar{\a}}\big)(\gB \xx):=\int_{\gR^m} \dd\, \gB q\ 
\overline{\varphi_\a}(q^\ev)\,\big(U_0(\gB\xx)\, (U_\gamma\psi)_{\bar{\a}}\big)(q^\ev)
$$
is the classical Wigner function over $\gR^{2m}$.
Since $\varphi,\psi\in\ehH_S^\infty$, the components 
$\varphi_\a$ and $(U_\gamma\psi)_{\bar{\a}}$ are in $\caS(\gR^{m|0})\simeq\caS(\gR^m)$. 
According to \cite[chap 11]{Groechenig:2001}, we have then 
$V_0(\varphi_\a,(U_\gamma\psi)_{\bar{\a}})\in L^1(\gR^{2m})$ for all
possible indices $\a,\gamma\in(\gZ_2)^{2(r+s)}$. Therefore, 
the Wigner function $V(\varphi,\psi)$ is in $L^1(\modE_0)$.

Let $J$ be the fundamental symmetry of $\ehH_S$, obtained from the ones in Examples \ref{ex-superhilbert} and \ref{ex-holfct} 
via the tensor product rule \eqref{Eq:J}. After integration of the odd variables $x^\od$, we get
\begin{multline*}
\int\dd\xx\ \overline{V(\varphi,\psi)(\xx)} JV(\varphi,\psi)(\xx) 
\leq \sum_{\alpha,\gamma,\alpha',\gamma'}  \int_{\gR^{2m}}\dd x\ |V_0(\varphi_\alpha,(U_\gamma \psi)_{\bar{\a}})(x)| | V_0(\varphi_{\alpha'},(U_{\gamma'}\psi)_{\bar{\a'}})(x)|.
\end{multline*}
The functions $\varphi_\alpha$, $(U_\gamma \psi)_{\bar{\a}}$, $\varphi_{\alpha'}$,
$(U_{\gamma'}\psi)_{\bar{\a'}}$ are in $\caS(\gR^m)$ for all $\a,\a',\g,\g'\in(\gZ_2)^{2(r+s)}$. Hence, according to \cite[chap 11]{Groechenig:2001}, 
the Wigner functions $V_0(\varphi_\alpha,(U_\gamma \psi)_{\bar{\a}})$ and 
$V_0(\varphi_{\alpha'},(U_{\gamma'}\psi)_{\bar{\a'}})$ are in $L^2(\gR^{2m})$. 
Therefore, the left hand side of the display is bounded and $V(\varphi,\psi)$ is in $L^2(\modE_0)$.
\end{proof}

The next result is of crucial importance for the proof of the Stone-Von Neumann Theorem.

\begin{proposition}[Resolution of the identity]
\label{prop-resol}
Let $\varphi,\psi\in\caS(\gR^{m|r})$.
For any $\varphi',\psi'\in\ehH_S$, we have the following identities: 
\begin{equation}\label{Eq:resolutionId}
\int\dd \xx\ V\big(\varphi,\varphi'\big)(-\xx)\,V\big(\psi',\psi\big)(\xx)
=\int\dd \xx\ \langle U(\xx)\varphi,\varphi'\rangle\langle\psi',U(\xx)\psi\rangle
=\kappa_\hbar\langle\varphi,\psi\rangle\langle\psi',\gP\varphi'\rangle,
\end{equation}
where $\gP\varphi':=(-1)^{|\varphi'|}\varphi'$ is  the parity operator.
\end{proposition}

Note that the first equality in \eqref{Eq:resolutionId} results directly from the definition of Wigner functions while the second one is commonly named a resolution of identity.

\begin{proof}
We show below the resolution of the identity for $\varphi',\psi'\in\ehH^\infty_S$. Since the inner product $\langle -,-\rangle$ is continuous, see Lemma \ref{lem-ineq}, it suffices to prove \eqref{Eq:resolutionId} for $\varphi',\psi'\in\ehH_S^\infty$.
The extension to $\varphi',\psi'\in\ehH_S$ follows then by a density argument. 

Assume $\varphi,\psi,\varphi',\psi'\in\ehH_S^\infty$.
By Lemma \ref{lem-gaussian}, the Wigner functions $V\big(\varphi,\varphi'\big)$ and 
$V\big(\psi',\psi\big)$ are in $L^2(\modE_0)$.
The left hand side of the identity \eqref{Eq:resolutionId} is then well-defined.
Due to the definitions \eqref{eq-schro}, \eqref{eq-innerschro} and \eqref{Eq:ddx}, 
 it reads as
\begin{multline*}
LHS=(2i)^{3s}(-1)^{r(|\varphi|+|\varphi'|)}\int\dd q\dd p\dd\zeta\dd\bzeta \dd q_0\dd\zeta_0\dd\bzeta_0\dd q_1\dd\zeta_1\dd\bzeta_1\\
 e^{\frac{i\hbar}{2}\big(2(q_0-q_1)p-\eps(\frac12\bzeta-\bzeta_0)\zeta+ \eps(\frac12\zeta-\zeta_1)\bzeta+\eps\zeta_0\bzeta_0+\eps\zeta_1\bzeta_1\big)}
 \overline{\varphi(q_0-q)}\varphi'(q_0,\zeta_0)\overline{\psi'(q_1,\zeta_1)}\psi(q_1-q,\zeta_1-\zeta).
\end{multline*}
We perform the change of variables $(q_1,p)\mapsto(q_1+q_0,p)$ and integrate over
the variables $(q_1,p)$. By Equation \eqref{eq-dirac0}, we obtain  
\begin{multline*}
LHS=\frac{\kappa_\hbar}{(\eps \hbar)^s}(2i)^{3s}(-1)^{r(|\varphi|+|\varphi'|)}(-1)^r
\int\dd q\dd\zeta\dd\bzeta \dd q_0\dd\zeta_0\dd\bzeta_0\dd\zeta_1\dd\bzeta_1\\
 e^{\frac{i\hbar\eps}{2}\big((\zeta-\zeta_1) (\bzeta-\bzeta_0)+(\zeta_0-\zeta_1)\bzeta_0+\zeta_1\bzeta_1\big)}
\overline{\varphi(q_0-q)}\varphi'(q_0,\zeta_0)\overline{\psi'(q_0,\zeta_1)}\psi(q_0-q,\zeta_1-\zeta).
\end{multline*}
By Equation \eqref{eq-dirac1}, the change of variables 
$(\zeta_0,\zeta,\bzeta)\mapsto (\zeta_0+\zeta_1,-\zeta+\zeta_1,-\bzeta+\bzeta_0)$ 
and integration over $(\zeta_0,\bzeta_0)$ give
\begin{multline*}
LHS=\kappa_\hbar(2i)^{2s}(-1)^{r(|\varphi|+|\varphi'|)}(-1)^r
\int\dd q\dd\zeta\dd\bzeta \dd q_0\dd\zeta_1\dd\bzeta_1\\
 e^{\frac{i\hbar\eps}{2}\big(\zeta\bzeta+\zeta_1\bzeta_1\big)}
\overline{\varphi(q_0-q)}\varphi'(q_0,\zeta_1)\overline{\psi'(q_0,\zeta_1)}\psi(q_0-q,\zeta).
\end{multline*}
Then, the change of variables $(q,q_0)\mapsto (q_0-q,q_1)$ yields 
\begin{equation*}
LHS=\kappa_\hbar(2i)^{2s}(-1)^{r(|\varphi|+|\varphi'|)}
\int\dd q\dd\zeta\dd\bzeta \dd q_1\dd\zeta_1\dd\bzeta_1\ \left(
 e^{\frac{i\hbar\eps}{2}\big(\zeta\bzeta+\zeta_1\bzeta_1\big)}
\overline{\varphi(q)}\varphi'(q_1,\zeta_1)\overline{\psi'(q_1,\zeta_1)}\psi(q,\zeta)\right).
\end{equation*}
By definition \eqref{eq-innerschro} of the inner product, this means
$$
LHS=\kappa_\hbar(-1)^{r(|\varphi|+|\varphi'|)}(-1)^{r|\varphi|} (-1)^{|\psi|(r+|\varphi'|+|\psi'|)} (-1)^{|\varphi'| |\psi'|}
\langle\varphi,\psi\rangle \langle\psi',\varphi'\rangle.
$$
Since $\langle\varphi',\psi'\rangle$ is zero unless the degrees satisfy $|\varphi'|+|\psi'|=r$, 
we obtain the announced resolution of identity.
\end{proof}

Using the parity operator $\gP$, introduced in Equation \eqref{eq-parity}, we define 
the operator $\gM$, which acts on $\varphi\in \ehH_S$ by
\begin{equation}\label{Def:M}
(\gM\varphi)(q,\zeta):=\gP\varphi(-q,-\zeta).
\end{equation}
If we split $q\in\gR^{m|r}$ as $q=\qev+\qod$ with $\qev\in\gR^{m|0}$ 
and $\qod\in\gR^{0|r}$, we have $(\gM\varphi)(q,\zeta)=\varphi(-\qev,\qod,\zeta)$ 
and $(\gP\varphi)(q,\zeta)=\varphi(\qev,-\qod,-\zeta)$.

\begin{lemma}
\label{lem-minus}
Let $\varphi,\psi\in\ehH_S$ with $\varphi$ or $\psi$ in $\ehH_S^\infty$. The Wigner function satisfies the following property:
\begin{equation*}
\forall \xx\in\modE_0,\quad V\big(\varphi,\psi\big)(-\xx)
=(-1)^{r+|\varphi|+|\psi|}V\big(\gM\varphi,\gM\psi\big)(\xx). 
\end{equation*}
\end{lemma}
\begin{proof}
Let $\xx=(q,p,\xi)\in\modE_0$ and $\zeta=\xi+i I\xi\in\gC^{0|s}$.
By Equations \eqref{eq-wigner}, \eqref{eq-innerschro} and \eqref{eq-schro}, we have
$$
V\big(\varphi,\psi\big)(-\xx)=(2i)^s\int \dd q_0\dd\zeta_0\dd\bzeta_0\
 e^{\frac{i\hbar}{2}\big( \zeta_0\bzeta_0+(q+2q_0)p+(\frac12\zeta+\zeta_0)\bzeta \big)} 
\overline{\varphi(q_0,\zeta_0)}
\psi(q_0+q,\zeta_0+\zeta).
$$
The result follows from the change of variables $(q_0,\zeta_0,\bzeta_0)\mapsto(-q_0,-\zeta_0,-\bzeta_0)$.
\end{proof}

The {\defin twisted convolution} on $\modE_0$ is defined as
\begin{equation}\label{Def:TwistedConvol}
\big(f^{(1)}\convol f^{(2)}\big)(\xx):=\int_{\modE_0}\dd \yy\ f^{(1)}(\yy) f^{(2)}(\xx-\yy)
 e^{\frac{i\hbar}{2}\omega(\yy,\xx)},
\end{equation}
for any smooth functions with compact support $f^{(1)},f^{(2)}\in \caD(E_0)$ 
(see \cite{Folland:1989} for the non-graded case).

\begin{lemma}
\label{lem-twistedl1}
The twisted convolution extends uniquely to $L^1(\modE_0)$ and turns it into a Banach algebra.
\end{lemma}
\begin{proof}
Let $f^{(1)},f^{(2)}\in \caD(\modE_0)$ and recall that $\modE_0=\gR^{2m|\signp+\signq}$. 
Using the decomposition \eqref{Function:Rmn}
of functions along odd variables, we get 
\begin{gather*}
f^{(1)}(\xx)=\sum_{\b\in(\gZ_2)^{\signp+\signq}}\widetilde{f^{(1)}_\b}(x^\ev)\,(x^\od)^\b,\qquad
f^{(2)}(\xx)=\sum_{\g\in(\gZ_2)^{\signp+\signq}}\widetilde{f^{(2)}_\g}(x^\ev)\,(x^\od)^\g, \\
f^{(1)}*f^{(2)}(\xx)=\sum_{\a\in(\gZ_2)^{\signp+\signq}}\widetilde{(f^{(1)}*f^{(2)})_\a}(x^\ev)\,(x^\od)^\a, 
\end{gather*}
where $\xx\in \modE_0$ is decomposed as $\xx=x^\ev+x^\od$ with $x^\ev\in\gR^{2m|0}$ and $x^\od\in\gR^{0|\signp+\signq}$. In view of \eqref{Def:TwistedConvol}, for any $\a\in(\gZ_2)^{\signp+\signq}$, there exist positive constants $c_{\a,\b,\g}$ such that
\begin{multline*}
\int_{\gR^{2m}}\dd x\ |\big(f^{(1)}\ast f^{(2)}\big)_\a (x)|\leq \sum_{\b,\g}c_{\a,\b,\g}\int_{\gR^{4m}} \dd x\dd y\ |f^{(1)}_\b(y)||f^{(2)}_\g(x-y)| \\
\leq \sum_{\b,\g}c_{\a,\b,\g}\int_{\gR^{4m}} \dd x\dd y\ |f^{(1)}_\b(y)| |f^{(2)}_\g(x)| = \sum_{\b,\g}c_{\a,\b,\g}\norm f^{(1)}_\b\norm_{L^1} \norm f^{(2)}_\g\norm_{L^1}.
\end{multline*}
This exactly means that $f^{(1)}\ast f^{(2)}$ is in $L^1(\modE_0)$ (see Equation \eqref{eq-berezin} and below).
\end{proof}

On Wigner functions, the twisted convolution can be explicitly computed
in some cases.

\begin{lemma}
\label{lem-vphi}
Let $\varphi,\psi\in\caS(\gR^{m|r})$. For any $\varphi',\psi'\in\ehH_S$, we have
\begin{equation*}
\forall \xx\in \modE_0,\quad \Big(V(\varphi',\psi)\convol V(\varphi,\psi')\Big)(\xx)
=\kappa_\hbar \langle\varphi',\psi'\rangle\ V(\varphi,\gP\psi)(\xx).
\end{equation*}
\end{lemma}
\begin{proof}
By definition, the left hand side of the above display is equal to
\begin{equation*}
LHS=\int\dd \yy\ V\big(\varphi',\psi\big)(\yy)\, V\big(\varphi,\psi'\big)(\xx-\yy)\, 
e^{\frac{i\hbar}{2}\omega(\yy,\xx)}.
\end{equation*}
Using the definition of Wigner functions, Equation \eqref{Eq:UU} and 
$\omega(\yy,\xx)=-\omega(\xx,\yy)$, we get
$$
V\big(\varphi,\psi'\big)(\xx-\yy)\, e^{\frac{i\hbar}{2}\omega(\yy,\xx)}=\langle\varphi,U(\xx)U(-\yy)\psi'\rangle
=V\big(U(-\xx)\varphi,\psi'\big)(-\yy).
$$
The resolution of the identity \eqref{Eq:resolutionId} 
yields then
$
LHS=\kappa_\hbar \langle\varphi',\psi'\rangle\ \langle U(-\xx)\varphi,\gP\psi\rangle
$
and the result follows.
\end{proof}

\section{Stone-von Neumann theorem}

From now on, $\hbar$ is a fixed non-zero real.
The aim of this section is to prove the theorem below.

\begin{theorem}[extended Stone-von Neumann]\label{thm:StonevN}
Let $\ehH$ be a Hilbert superspace of parity $\sigma\in\{0,1\}$ 
and $(\ehH,\pio,\pi)$ be a strong SUR of the Heisenberg supergroup  $\rH_{2m|\signp,\signq}$
such that central elements act by $\pi(0,t)=e^{i\hbar t}\gone$, for all $t\in\gR$.
Consider the Schr\"odinger representation $(\ehH_S,U_0,U)$ of $\rH_{2m|\signp,\signq}$, 
with parameter $\hbar$, and parity $\sigma$ if $\signp+\signq$ is odd.
 Then, $(\ehH,\pio,\pi)$ is equivalent to the tensor product of $(\ehH_S,U_0,U)$ 
with the trivial representation on a Hilbert superspace $\ehH_R$.
\end{theorem}

The next section deals with the notion of integrated representation. 
This and the results obtained in Section \ref{Sec:AnalyticProp} are the foreground
for the proof of the theorem in the case $\signp+\signq$ even,
which is performed in Section \ref{subsec-even}. Eventually, in Section \ref{subsec-odd}, 
we reduce the case $\signp+\signq$ odd to the case $\signp+\signq$ even, which concludes the proof.

\subsection{Preliminaries}

We refer to elements in the Heisenberg supergroup $\rH_{2m|\signp,\signq}$
as couples $(\xx,t)\in\modE_0\times\gR^{1|0}$, with $\modE_0=\gR^{2m|\signp+\signq}$ 
and we suppose that $\signp+\signq$ is even.
Let $(\ehH,\pio,\pi)$ be a strong SUR of $\rH_{2m|\signp,\signq}$
such that, for all $t\in\gR$,
$\pi(0,t)=e^{i\hbar t}\gone$, with $\hbar \in\gR^\times$. 
We set $\pi(\xx):=\pi(\xx,0)$ 
for all $\xx\in\modE_0$. According to the group law \eqref{GroupLaw} of $\rH_{2m|\signp,\signq}$, this
defines a projective representation of the Abelian supergroup $\modE_0$ such that
\begin{equation}\label{Eq:pipi}
\pi(\xx)\pi(\xx')=e^{\frac{i\hbar}{2}\omega(\xx,\xx')}\pi(\xx+\xx').
\end{equation}
The {\defin integrated representation} 
$\pi:\caD(E_0)\to\End(\ehH^\infty)$ is the map defined by the following expression,
\begin{equation*}
\pi(f):=\int_{\modE_0}\dd \xx\ f(\xx)\pi(\xx).
\end{equation*}
\begin{lemma}
\label{lem-extendpi}
The integrated representation extends uniquely as a Banach algebra morphism $\pi:L^1(E_0)\to \caB(\ehH)$. 
\end{lemma}
\begin{proof}
The decomposition along the odd variables 
of $\pi(\xx)$ reads as (see Proposition \ref{Prop:pibounded}):
$\pi(\xx)=\pi(x^{ev})\sum_{\a\in(\gZ_2)^{\signp+\signq}}\,(x^{od})^\a\pi_\a$. 
After integration of the odd variables $x^\od$, this yields to
\begin{equation*}
\pi(f)=\sum_{\a\in(\gZ_2)^{\signp+\signq}} \eps(\alpha)\Big(\int_{\gR^{2m}}\dd x\  f_\alpha(x)\pi_0(x)\Big)\pi_{\bar{\a}},
\end{equation*}
where $\bar{\a}=(1,\ldots,1)-\a\in(\gZ_2)^{\signp+\signq}$ and 
$\eps(\alpha)=\pm 1$ is defined in Example \ref{ex-superhilbert}.
Since $(\ehH,\pio,\pi)$ is a strong SUR, there exists a fundamental symmetry $J$ such that $(\ehH,\pio)$ is unitary w.r.t.\ the scalar product $(-,-)_J$. Denoting by $\norm-\norm$ the operator norm associated to $(-,-)_J$, we have $\norm\pi_0(x)\norm=1$ for all $x\in\gR^{2m}$ and then
\begin{equation*}
\norm\pi(f)\norm\leq \sum_{\a\in(\gZ_2)^{\signp+\signq}} \int_{\gR^{2m}}\dd x\ |f_\alpha(x)| \norm\pi_{\bar{\a}}\norm \, \leq \, C\sum_{\alpha\in(\gZ_2)^{\signp+\signq}}\norm f_\alpha\norm_{L^1},
\end{equation*}
where $C=\sup_{\a}\norm\pi_\a\norm$ is a finite number by Proposition \ref{Prop:pibounded}. 
This shows both that $\pi(f)$ is a well-defined bounded operator and that $\pi$ is a continuous map.

By definition, for any $f_1,f_2\in L^1(E_0)$, we have 
$$
\pi(f_1\convol f_2)=\int\dd\xx\dd\yy\ f_1(\yy)f_2(\xx-\yy)  e^{\frac{i\hbar}{2}\omega(\yy,\xx)} \pi(\xx).
$$
The change of variables $\xx\mapsto \xx+\yy$ and Equation \eqref{Eq:pipi} lead
then to $\pi(f_1\convol f_2)=\pi(f_1)\pi(f_2)$.
\end{proof}

According to Equation \eqref{Eq:pipi}, the following equalities hold
\begin{equation}\label{eq-pr1}
\pi(f)\pi(\yy)=\pi\left(e^{\frac{i\hbar}{2}\omega(\fois,\yy)}f(\fois-\yy)\right),\qquad \pi(\yy)\pi(f)=\pi\left(e^{-\frac{i\hbar}{2}\omega(\fois,\yy)}f(\fois-\yy)\right).
\end{equation}
Following \cite{Folland:1989}, we obtain
\begin{lemma}
\label{lem-pi-faithful}
The representation $\pi$ is faithful on $L^1(E_0)$.  
\end{lemma}
\begin{proof}
Let $f\in L^1(E_0)$ such that $\pi(f)=0$. We have to prove that $f=0$ almost everywhere.
Denoting by $\langle-,-\rangle_\ehH$ the inner product on $\ehH$ and using Equation \eqref{eq-pr1},
we have, for all $u,v\in\ehH$, $\yy\in E_0$,
\begin{align*}
0&=\langle \pi(\yy)\pi(f)\pi(-\yy)u,v\rangle_\ehH,\\
&=\langle\pi(e^{-\frac{i\hbar}{2}\omega(\fois,\yy)}e^{\frac{i\hbar}{2}\omega(\fois-\yy,-\yy)}f)u,v\rangle_\ehH,\\
&=\int \dd\xx\ e^{-i\hbar\omega(\xx,\yy)} \langle f(\xx) \pi(\xx)u,v\rangle_\ehH.
\end{align*}
Thus by the Fourier inversion theorem, $\langle f(\xx) \pi(\xx)u,v\rangle_\ehH=0$ almost everywhere.
Since this is true for any $u,v\in\ehH$, we conclude that $f=0$ almost everywhere.
\end{proof}

\subsection{Proof of the theorem in the even case}\label{subsec-even}

In this section, we prove Theorem \ref{thm:StonevN} under the assumption that 
the odd dimension $\signp+\signq$ of the Heisenberg supergroup $\rH_{2m|\signp,\signq}$
is even. The strategy of our proof is adapted from
the one in the non-graded case given in \cite{Folland:1989}. 
Namely, we start by choosing 
special states $\varphi_\kP$, $\psi_\kP$, so that 
their Wigner function $V(\varphi_\kP,\psi_\kP)$ defines a projector 
$\kP=\pi\big(V(\varphi_\kP,\psi_\kP)\big)$ via the integrated representation. 
Contrary to the non-graded case, the projector $\kP$ may not be orthogonal.
Therefore, the image $\ehH_R=\Imag(\kP)$ may not inherit a Hilbert superspace structure from $\ehH$.
We construct one  in \eqref{InnerHR}, from the integrated representation of another Wigner function.
Finally, we define an operator $\Phi:\ehH_S\hat{\otimes}\ehH_R\to\ehH$ in Equation \eqref{IsoH-HSHR} 
and prove in Lemmas \ref{lem-defphi}-\ref{PhiSuperU} that $\Phi$ is a superunitary intertwiner between 
the representation maps $U\otimes \gone$ and $\pi$.
\medskip

In the whole section, the inner products are denoted by $\langle-,-\rangle$ on $\ehH_S$ and $\langle-,-\rangle_\ehH$ on $\ehH$.
Recall that $\ehH_S^\infty=\caS(\gR^{m|r})\otimes Hol(\gC^{0|s})$.
Let $\varphi_\kP,\psi_\kP\in\caS(\gR^{m|r})\subseteq\ehH_S^\infty$ 
be two functions such that
\begin{equation}\label{phiPpsiP}
\gM\varphi_\kP=\varphi_\kP,\quad \gM\psi_\kP=\psi_\kP,\quad |\varphi_\kP|=r\hspace{-8pt}\mod 2,\quad |\psi_\kP|=0,\qquad \kappa_\hbar\langle\varphi_\kP,\psi_\kP\rangle=1,
\end{equation}
where $\gM$ is the operator defined in \eqref{Def:M} and $\kappa_\hbar\in\gC$ is defined
in Lemma \ref{lem:integrals}. For instance, we can choose any non-zero 
function $\psi_\kP\in\caS(\gR^{m|r})$ of degree 0 and take $\varphi_\kP=bJ\psi_\kP$ 
with $J$ a fundamental symmetry of $\ehH_S$ and $b\in\gC$ defined by
$b:=(\kappa_\hbar\langle J\psi_\kP,\psi_\kP\rangle)^{-1}$.

\begin{lemma}\label{lem:kP}
The operator $\kP:=\pi\big(V(\varphi_\kP,\psi_\kP)\big) \in \caB(\ehH)$ 
is a bounded projector satisfying 
$$
\forall \yy\in E_0,\quad\kP\pi(\yy)\kP=\kappa_\hbar V\big(\varphi_\kP,\psi_\kP\big)(\yy)\kP.
$$ 
Moreover, the superadjoint operator of $\kP$ is given by
$\kP^\dag=\pi\big(V(\psi_\kP,\varphi_\kP)\big)$.
\end{lemma}
\begin{proof}
By Lemma \ref{lem-gaussian} the Wigner function $V(\varphi_\kP,\psi_\kP)$ 
is in $L^1(\modE_0)$,
and by Lemma \ref{lem-extendpi} its integrated representation $\kP=\pi\big(V(\varphi_\kP,\psi_\kP)\big)$ 
is a bounded operator.

By Equation \eqref{eq-pr1}, we have 
$$
\kP \pi(\yy)\kP=\kP\pi\left(e^{-\frac{i\hbar}{2}\omega(\fois,\yy)}V\big(\varphi_\kP,\psi_\kP\big)(\fois-\yy)\right)
=\pi\big(V(\varphi_\kP,\psi_\kP)\big)\;\pi\big(V(\varphi_\kP,U(-\yy)\psi_\kP)\big).
$$
Using successively Lemma \ref{lem-extendpi} and Lemma \ref{lem-vphi}, we get 
$$
\kP \pi(\yy)\kP
=\kappa_\hbar V\big(\varphi_\kP,\psi_\kP\big)(-\yy)\, \pi\big( V(\varphi_\kP,\gP\psi_\kP)\big),
$$
and, by Lemma \ref{lem-minus} and Equation \eqref{phiPpsiP}, we obtain that
$\kP\pi(\yy)\kP=\kappa_\hbar V\big(\varphi_\kP,\psi_\kP\big)(\yy)\kP$. Taking $\yy=0$,
this implies that $\kP^2=\kP$ is a projector.

For any $\varphi,\psi\in\caS(\gR^{m|r})$, we have 
$\overline{V\big(\varphi,\psi\big)(-\xx)}=\overline{\langle \varphi,U(-\xx)\psi\rangle}
=(-1)^{|\varphi||\psi|}V\big(\psi,\varphi\big)(\xx)$. As $\pi(\xx)^\dag=\pi(-\xx)$, we deduce that
\begin{equation}\label{pif:dag}
\pi\big(V(\varphi,\psi)\big)^\dag=(-1)^{|\varphi||\psi|}\pi\big(V(\psi,\varphi)\big).
\end{equation}
Taking $\varphi=\varphi_\kP$ and $\psi=\psi_\kP$, the expression of $\kP^\dag$ follows.
\end{proof}

Since $\kP$ is an operator of degree $0$, the space $\ehH_R:=\Imag(\kP)$
inherits of a $\gZ_2$-grading from $\ehH$. Thanks to Lemma \ref{lem-extendpi} 
and Lemma \ref{lem-gaussian}, the operator $\kV=\kappa_\hbar\,\pi\big(V(\varphi_\kP,\varphi_\kP)\big)$
is bounded on $\ehH$. It allows us to introduce the following inner product on $\ehH_R$, 
\begin{equation}\label{InnerHR}
\forall v,v'\in\ehH_R,\quad \langle v,v'\rangle_R:= (-1)^{r\sigma} \langle \kV  v,v'\rangle_{\ehH},
\end{equation}
where $\sigma$ is the parity of $\ehH$.
Note that,
extending  $\langle -,-\rangle_{\ehH}$ by $\superA$-bilinearity as in Equation \eqref{eq-abilin}, 
the inner product reads also as 
$\langle v,v'\rangle_R=\kappa_\hbar\int\dd\xx\ \langle\varphi_\kP,U(\xx)\varphi_\kP\rangle\langle v,\pi(\xx)v'\rangle_\ehH$.

\begin{lemma}\label{HR}
The superspace $\ehH_R$ endowed with the inner product $\langle \cdot,\cdot\rangle_R$
is a Hilbert superspace.
\end{lemma}

\begin{proof}
We split the proof into two cases, according to the parity of $\ehH_S$, i.e.\ the parity of $r$.

If $r$ is even, we can choose an element $\psi_\kP=\varphi_\kP$ of degree 0, 
invariant under the operator $\gM$ and such that $\kappa_\hbar\langle\varphi_\kP,\varphi_\kP\rangle=1$. 
Then, we have $\langle v,v'\rangle_R=\langle \kP v, v'\rangle_\ehH=\langle v,v'\rangle_\ehH$
and  $\kP^\dag=\kP$. Hence, $\ehH_R$ is a Hilbert sub-superspace of $\ehH$ by Corollary \ref{cor-subspace}.

If $r$ is odd, the function $\varphi_\kP$ is of degree 1 and the function $\psi_\kP$ is of degree 0. 
As a consequence, we have $\langle\varphi_\kP,\varphi_\kP\rangle=\langle\psi_\kP,\psi_\kP\rangle=0$.
Then, both operators $\kV$ and
$\kV':=-\frac{1}{\kappa_\hbar}\, \pi\big(V(\psi_\kP,\psi_\kP)\big)$ are odd.
Using Lemma \ref{lem-extendpi} and Lemma \ref{lem-vphi}, we obtain that
\begin{equation*}
\kP\kP^\dag=\kP^\dag\kP=0,\quad \kV\kV'=\kP^\dag,\quad \kV'\kV=\kP, \quad \kV^2=(\kV')^2=0, 
\end{equation*}
and Equation \eqref{pif:dag} yields to 
\begin{equation*}
\kV^\dag=\kV,\quad (\kV')^\dag=-\kV'.
\end{equation*}
In particular, we have
\begin{equation*}
\kV(\Imag(\kP))\subseteq\Imag(\kP^\dag),\qquad \kV(\Imag(\kP^\dag))=\algzero,\qquad \kV'(\Imag(\kP))=\algzero,\qquad \kV'(\Imag(\kP^\dag))\subseteq\Imag(\kP).
\end{equation*}
As $\kP\kP^\dag=\kP^\dag\kP=0$, we deduce that $\kP+\kP^\dag$ is an orthogonal 
projector. By Corollary \ref{cor-subspace}, this means that 
$\ehH_P:=\Imag(\kP)\oplus\Imag(\kP^\dag)$ is a Hilbert sub-superspace of $\ehH$. 
Note that, as $\kP^\dag\kP=0$, $\ehH_R=\Imag\kP$ is an isotropic subspace of $\ehH_P$
for the inner product of $\ehH$.

We define a degree zero operator $\Psi$ as follows
\begin{equation*}
\Psi:\ehH_R\to \text{Span}(\varphi_\kP,\psi_\kP)\otimes\ehH_P,\qquad v\mapsto \frac{1}{\sqrt 2}\big(\psi_\kP\otimes v+(-1)^{\sigma}\,\overline{\kappa_\hbar}\ \varphi_\kP\otimes \kV v\big),
\end{equation*}
where $\sigma\in\gZ_2$ is the parity of $\ehH$. 
The operator $\Psi$ is injective and satisfies
\begin{equation*}
\forall v,v'\in\ehH_R,\quad \langle\Psi(v),\Psi(v')\rangle_\ehH=\langle v,v'\rangle_R.
\end{equation*}
Therefore, to prove that $(\ehH_R,\langle -,-\rangle_R)$ 
is a Hilbert superspace, it suffices to prove that 
$\Psi(\ehH_R)$ is a Hilbert sub-superspace of $\text{Span}(\varphi_\kP,\psi_\kP)\otimes\ehH_P$.

Consider $w\in\Psi(\ehH_R)\cap\Psi(\ehH_R)^\perp$. 
As $w\in\Psi(\ehH_R)$, there exists $v\in\ehH_R$ such that $w=\psi(v)$.
The relation $w\in\Psi(\ehH_R)^\perp$ yields then
$\langle \kV  v,v'\rangle_{\ehH}=0$ for all $v'\in\ehH_R=\Imag\kP$.
By the equality $\kV^\dag\kP^\dag=\kV\kP^\dag=0$, we have that 
$\langle \kV  v,v''\rangle_{\ehH}=0$ for all $v''\in\Imag\kP^\dag$.
Hence, $\kV v$ is orthogonal to all elements in the Hilbert superspace $\ehH_P=\Imag(\kP)\oplus\Imag(\kP^\dag)$, that is $\kV v=0$.
As a consequence, we get that $v=\kP v = \kV'\kV v = 0$ and then
$w=\psi(v)=0$. This proves that $\Psi(\ehH_R)\cap\Psi(\ehH_R)^\perp
=\algzero$.

For any $u_0,u_1\in\ehH_P$ we set
\begin{equation*}
v=\frac{1}{\sqrt2}\left(\kP u_0+(-1)^{\sigma}\kV' u_1\right),\qquad w_0=u_0-\frac{1}{\sqrt 2}v,\qquad w_1=\kP u_1 -(-1)^{\sigma}\kV w_0,
\end{equation*}
so that $v\in\ehH_R$ and $w_0,w_1\in\ehH_P$. Direct computations show that
\begin{align*}
& w=\psi_\kP\otimes w_0+\overline{\kappa_\hbar}\ \varphi_\kP\otimes w_1\in\Psi(\ehH_R)^\perp, \\
& \psi_\kP\otimes u_0+\overline{\kappa_\hbar}\ \varphi_\kP\otimes u_1= 
w+\Psi(v).
\end{align*}
Together with $\Psi(\ehH_R)\cap\Psi(\ehH_R)^\perp=\algzero$, this shows that
$\text{Span}(\varphi_\kP,\psi_\kP)\otimes\ehH_P=\Psi(\ehH_R)^\perp \oplus \Psi(\ehH_R)$. 
By Theorem \ref{thm-subspace}, 
$\Psi(\ehH_R)$ is then a Hilbert sub-superspace and the conclusion follows.
\end{proof}

We introduce a homogeneous linear map of degree 0 defined by
\begin{equation}
\begin{aligned}\label{IsoH-HSHR}
\Phi:& \ehH_S\hat{\otimes}\ehH_R\to\ehH, \\ 
& \varphi\otimes v\mapsto (-1)^{|\varphi|}\int\dd \xx\ \langle U(\xx)\varphi_\kP,\varphi\rangle\pi(\xx)v.
\end{aligned}
\end{equation}
In the remaining of the section, we prove that this operator is a superunitary intertwiner between the representations $\pi$ and $U\otimes\gone$. 

\begin{lemma}
\label{lem-defphi}
The operator $\Phi$ is defined and bounded on the whole Hilbert superspace $\ehH_S\hat{\otimes}\ehH_R$.
\end{lemma}
\begin{proof}
Let $\varphi\in\ehH_S^\infty$ and $v\in \ehH_R$.
By Lemma \ref{lem-gaussian}, the function 
$f_\varphi:\xx\mapsto\langle U(\xx)\varphi_\kP,\varphi\rangle
=V(\varphi_\kP,\varphi)(-\xx)$ is in $L^1(\modE_0)$.
The operator $\pi(f_\varphi)$ is then bounded by Lemma \ref{lem-extendpi}. Hence, $\Phi(\varphi\otimes v)=\pi(f_\varphi)(v)$
pertains to $\ehH$ and $\Phi$ is a well-defined operator on the
algebraic tensor product of $\ehH_S^\infty$ and $\ehH_R$.

Since $(\ehH,\pio,\pi)$ is a strong SUR, there exists a fundamental symmetry $J$ such that $(\ehH,\pio)$ is unitary w.r.t.\ the scalar product $(-,-)_J$.
For all $\varphi,\varphi'\in\ehH_S^\infty$ and $v,v'\in \ehH_R$, 
the operator $\Phi$ satisfies
\begin{multline*}
\left(\Phi(\varphi\otimes v),\Phi(\varphi'\otimes v')\right)_J\\
=(-1)^{|\varphi|+|\varphi'|+r(\sigma+|v|)+\sigma|\varphi|+|\varphi'||v|}\int\dd \xx\dd \yy\ \overline{\langle U(\xx)\varphi_\kP,\varphi\rangle}\langle U(\yy)\varphi_\kP,\varphi'\rangle \langle \pi(\xx)v,J\pi(\yy)v'\rangle.
\end{multline*}
The operators $U(\xx)$ and $\pi(\xx)$ can be decomposed along odd variables 
(see Proposition \ref{Prop:pibounded}): $U(\xx)=U_0(x^\ev)\sum_{\gamma}(x^{od})^\gamma U_\gamma$ and $\pi(\yy)
=\pio(y^\ev)\sum_{\gamma}(y^{od})^{\gamma'} \pi_{\g'}$. Setting $x:=\gB\xx$ and $y:=\gB\yy$
and integrating over the odd variables $x^\od$ and $y^\od$, we get
\begin{multline*}
\left(\Phi(\varphi\otimes v),\Phi(\varphi'\otimes v')\right)_J=\\
=\sum_{\a,\g,\a',\g'} \eps_{\a,\g,\a',\g'} \int_{\gR^{4m}}\dd x\dd y\ \overline{\left( U_0(x) (U_\gamma\, \varphi_\kP)_\alpha,\varphi_{\bar{\a}}\right)}\left(U_0(y) (U_{\g'}\,\varphi_\kP)_{\a'},\varphi'_{\bar{\a'}}\right) \big\langle \pio(x)\pi_{\bar{\gamma}}v,J\pio(y)\pi_{\bar{\g'}}v'\big\rangle,
\end{multline*}
where $\eps_{\a,\g,\a',\g'}=\pm 1$. Using successively the change of variables $y\mapsto y+x$ 
and the resolution of the identity in the variable $x$ (see Proposition \ref{prop-resol} in the case $r=s=0$), we obtain
\begin{multline*}
\left(\Phi(\varphi\otimes v),\Phi(\varphi'\otimes v')\right)_J=\\
=\sum_{\a,\g,\a',\g'} \eps_{\a,\g,\a',\g'}\int\dd x\dd y\ \left( \varphi_{\bar{\a}},U_0(x) (U_\gamma\, \varphi_\kP)_\alpha\right)\left( U_0(x)U_0(y) (U_{\g'}\,\varphi_\kP)_{\a'},\varphi'_{\bar{\a'}}\right)  \left(\pi_{\bar{\gamma}}v,\pio(y)\pi_{\bar{\g'}}v'\right)_J\\
=\sum_{\a,\g,\a',\g'} \kappa_\hbar\, \eps_{\a,\g,\a',\g'}\int\dd y\ \left( \varphi_{\bar{\a}},\gP\varphi'_{\bar{\a'}}\right) \left( U_0(y) (U_{\g'}\,\varphi_\kP)_{\a'}, (U_\gamma\,\varphi_\kP)_\alpha\right)  \left(\pi_{\bar{\gamma}}v,\pio(y)\pi_{\bar{\g'}}v'\right)_J.
\end{multline*}
Since the operators $U_\g$ stabilize the space of smooth vectors $\ehH^\infty_S$,
Lemma \ref{lem-gaussian} ensures that 
$\int\dd y\ \left| \left( U_0(y) (U_{\g'}\,\varphi_\kP)_{\a'}, (U_\gamma\,\varphi_\kP)_\alpha\right)\right|<\infty$.
Moreover, the operators $\pio(y)$ being $(-,-)_J$-unitary, we have 
$\left|\left(\pi_{\bar{\gamma}}v,\pio(y)\pi_{\bar{\g'}}v'\right)_J\right|\leq \norm \pi_{\bar{\gamma}}v\norm_J\norm \pi_{\bar{\g'}}v'\norm_J$. The operators $\pi_\g$ being bounded, we finally obtain that
$$
|(\Phi(\varphi\otimes v),\Phi(\varphi'\otimes v'))_J|\leq 
C\norm \varphi\norm_S\norm \varphi'\norm_S \norm v\norm_J\norm v'\norm_J,
$$ 
where $C$ is a constant and $\norm -\norm_S$ is a norm on $\ehH_S$, defined by a choice of fundamental symmetry. 
This means that $\Phi$ is continuous and 
can be extended to the Hilbert tensor product $\ehH_S\hat{\otimes}\ehH_R$.
\end{proof}

\begin{lemma}\label{lem:Phi}
The operator $\Phi$ intertwines the representations $U\otimes\gone$ and $\pi$.
\end{lemma}

\begin{proof}
Let $g=(\xx,t)\in \rH_{2m|\signp,\signq}$. For $\varphi\in\ehH_S^\infty$, we compute
\begin{multline*}
\Phi((U(g)\otimes\gone)(\varphi\otimes v))=(-1)^{|\varphi|}\int\dd \yy\ e^{i\hbar t}\langle U(\yy)\varphi_\kP,U(\xx)\varphi\rangle\pi(\yy)v\\
= (-1)^{|\varphi|}\int\dd \yy\ e^{i\hbar t}e^{\frac{i\hbar}{2}\omega(\xx,\yy)}\langle U(\yy-\xx)\varphi_\kP,\varphi\rangle\pi(\yy)v=\pi(g)\Phi(\varphi\otimes v),
\end{multline*}
by performing the change of variables $\yy\mapsto \yy+\xx$. Hence, $\Phi$ intertwines
$U\otimes\gone$ and $\pi$.
\end{proof}

By Proposition \ref{prop-superhilbert-tensprod}, $\ehH_S\hat\otimes\ehH_R$ is a 
Hilbert superspace whose inner product is given by
\begin{equation}\label{InnerHSHR}
\forall \varphi,\varphi'\in\ehH_S, \; v,v'\in\ehH_R,\quad\langle \varphi\otimes v,\varphi'\otimes v'\rangle=(-1)^{r(\sigma+r)+|v||\varphi'|}\langle\varphi,\varphi'\rangle\langle v,v'\rangle_R,
\end{equation}
where $r$ is the parity of $\ehH_S$ and $\sigma$ the parity of $\ehH$.

\begin{lemma}\label{PhiSuperU}
The operator $\Phi:\ehH_S\hat{\otimes}\ehH_R\to\ehH$ is superunitary.
\end{lemma}

\begin{proof}
The proof is in two steps: \textbf{1.} $\Phi$ is an isometry, \textbf{2.} $\Phi$ is surjective.

\textbf{1.} Consider $\varphi,\varphi'\in\ehH_S$ and $v,v'\in\ehH_R$ four homogeneous elements.
The definition of $\Phi$ leads to
\begin{equation*}
\langle\Phi(\varphi\otimes v),\Phi(\varphi'\otimes v')\rangle= (-1)^{(|\varphi|+|\varphi'|)(\sigma+1)+|v||\varphi'|} \int\dd \xx\dd \xx'\ \overline{\langle U(\xx)\varphi_\kP,\varphi\rangle} \langle U(\xx')\varphi_\kP,\varphi'\rangle \langle \pi(\xx)v,\pi(\xx')v'\rangle  .
\end{equation*}
The change of variables $\xx'=\xx+\yy$ and the resolution of the identity in the variable $\xx$ yield successively
\begin{multline*}
\langle\Phi(\varphi\otimes v),\Phi(\varphi'\otimes v')\rangle\\
= (-1)^{(|\varphi|+|\varphi'|)(\sigma+1)+|v||\varphi'|+r|\varphi|} \int\dd \xx\dd \yy\ \langle \varphi,U(\xx)\varphi_\kP\rangle \langle U(\xx)U(\yy)\varphi_\kP,\varphi'\rangle \langle v,\pi(\yy)v'\rangle_\ehH ,\\
= (-1)^{(|\varphi|+|\varphi'|)(\sigma+1)+|v||\varphi'|+r|\varphi|+|\varphi||\varphi'|} \kappa_\hbar\int\dd \yy\ \langle U(\yy)\varphi_\kP,\varphi_\kP\rangle\langle \varphi,\gP\varphi'\rangle\langle v,\pi(\yy)v'\rangle_\ehH.
\end{multline*}
By Lemma \ref{lem-minus} and the equality $\gM\varphi_\kP=\varphi_\kP$, we have
\begin{equation*}
\langle U(\yy)\varphi_\kP,\varphi_\kP\rangle=V(\varphi_\kP,\varphi_\kP)(-\yy)=(-1)^rV(\varphi_\kP,\varphi_\kP)(\yy)=(-1)^r\langle \varphi_\kP,U(\yy)\varphi_\kP\rangle.
\end{equation*}
Considering Equation \eqref{InnerHSHR} and the defintion of $\langle-,-\rangle_R$, see below Equation \eqref{InnerHR},  we get
\begin{multline*}
\langle\Phi(\varphi\otimes v),\Phi(\varphi'\otimes v')\rangle= 
(-1)^{(|\varphi|+|\varphi'|)(\sigma+1)+|v||\varphi'|+r|\varphi|+|\varphi||\varphi'|}(-1)^{r} \langle \varphi,\gP\varphi'\rangle\langle v,v'\rangle_R,\\
= (-1)^{(|\varphi|+|\varphi'|+r)\sigma+|\varphi|(1+r+|\varphi'|)} \langle \varphi\otimes v,\varphi'\otimes v'\rangle.
\end{multline*}
If $|\varphi|+|\varphi'|=r$, then we obtain that
$\langle\Phi(\varphi\otimes v),\Phi(\varphi'\otimes v')\rangle=\langle \varphi\otimes v,\varphi'\otimes v'\rangle$. 
Otherwise, we have $|\varphi|+|\varphi'|=r+1$ and both inner products are zero. Hence, $\Phi$ is an isometry. 

\textbf{2.} By Proposition \ref{prop:ImIso}, the continuous isometry $\Phi$ satisfies
$\Imag\Phi\oplus(\Imag\Phi)^\perp=\ehH$ and $\Imag\Phi$, $(\Imag\Phi)^\perp$ are Hilbert superspaces. 
The source space of $\Phi$ is stable under the representation $U\otimes\gone$. Hence, by Lemma \ref{lem:Phi},
the image $\Imag\Phi$ is stable under the representation $\pi$. As $\pi$ is a superunitary representation,
$(\Imag\Phi)^\perp$ is also stable under $\pi$. In particular, if $f\in L^1(\modE_0)$, the operator $\pi(f)$ restricts
to $(\Imag\Phi)^\perp$. Taking $f=V(\varphi_\kP,\psi_\kP)$, we conclude that
 $\kP\left((\Imag\Phi)^\perp\right)\subseteq(\Imag\Phi)^\perp$. 
Using successively Equations \eqref{IsoH-HSHR}, \eqref{eq-wigner}, Lemma \ref{lem-minus} and Equation \eqref{phiPpsiP}, we get
\begin{align}\nonumber
\Phi(\psi_\kP\otimes v)&=\int\dd\xx \ \langle U(\xx) \varphi_\kP,\psi_\kP\rangle\pi(\xx) v 
=\int\dd\xx\ V(\varphi_\kP,\psi_\kP)(-\xx) \pi(\xx) v\\ \label{Eq:Phi-HR}
&=\pi\big(V(\varphi_\kP,\psi_\kP)\big) v = v,
\end{align}
for all $v\in\ehH_R=\Imag\kP$.
Thus, $\Imag\kP\subseteq\Imag\Phi$ and $\kP(\Imag\Phi)^\perp\subseteq \Imag\Phi\cap(\Imag\Phi)^\perp=\algzero$. 
That is $\kP=\pi(V(\varphi_\kP,\psi_\kP))=0$ on $(\Imag\Phi)^\perp$. 
Hence, either $(\Imag\Phi)^\perp=\algzero$ or, by Lemma \ref{lem-pi-faithful}, 
$V(\varphi_\kP,\psi_\kP)=0$ almost everywhere. By  Proposition \ref{prop-resol}, we obtain that 
\begin{equation}
\int\dd \xx\ V\big(\varphi_\kP,\psi_\kP\big)(-\xx)\,V\big(\varphi_\kP,\psi_\kP\big)(\xx)
=\kappa_\hbar\langle\varphi_\kP,\psi_\kP\rangle\langle\varphi_\kP,\gP\psi_\kP\rangle=\frac{1}{\kappa_\hbar}\neq 0.
\end{equation}
Therefore, $V(\varphi_\kP,\psi_\kP)$ does not vanish almost everywhere and we have $(\Imag\Phi)^\perp=\algzero$. 
This means that $\Phi$ is surjective.
\end{proof}

Lemmas \ref{HR}-\ref{PhiSuperU} show that $\Phi$
is a superunitary operator between the Hilbert superspaces
$\ehH_S\otimes\ehH_R$ and $\ehH$, which intertwines the
unitary representations $U\otimes\gone$ and $\pi$. 
This proves the Theorem \ref{thm:StonevN} under the assumption
$\signp+\signq$ even. 
\subsection{Proof of the theorem in the odd case}
\label{subsec-odd}

In this section, we prove Theorem \ref{thm:StonevN} under the assumption that 
the odd dimension $\signp+\signq$ of the Heisenberg supergroup $\rH_{2m|\signp,\signq}$ is odd. All the way through,  $\hbar$ denotes a non zero real number.
We need three preliminary Lemmas.
The first one relies on the proof of Theorem \ref{thm:StonevN} for $\signp+\signq$ even.

\begin{lemma}\label{lem:commutation}
Let $\signp+\signq$ be even and $(\ehH,\pio,\pi)$ be a strong SUR of $\rH_{2m|\signp,\signq}$
such that central elements act by $\pi(0,t)=e^{i\hbar t}\gone$, for all $t\in\gR$.
Consider the Hilbert superspace $\ehH_R$ (see Lemma \ref{HR} and paragraph above)  
and the superunitary intertwinner $\Phi:\ehH_S\hat{\otimes}\ehH_R\to\ehH$ (see \eqref{IsoH-HSHR}).
If $A\in\caB(\ehH)$ satisfies $[\pi(\xx,t),A]=0$  for all $(\xx,t)\in\rH_{2m|\signp,\signq}$,
then there exists $A_R\in\caB(\ehH_R)$ such that
$\Phi^{-1}\circ A\circ \Phi=\gone\otimes A_R$.
\end{lemma}
\begin{proof}
Since $\kP=\pi(V(\varphi_\kP,\psi_\kP))$ and $V(\varphi_\kP,\psi_\kP)$ is of degree $0$,
the hypothesis implies that $[A,\kP]=0$.
In particular, if $w\in\ehH_R$ then $Aw\in \ehH_R$, and the operator $A_R:=A_{|\ehH_R}$ is in $\caB(\ehH_R)$. 
By Equation \eqref{Eq:Phi-HR}, we have 
$$
\Phi^{-1}(Aw)=\psi_\kP\otimes Aw=\psi_\kP\otimes A_R w.
$$
As $\Phi^{-1}$ is an intertwining operator, the following equalities hold
 for all $(\xx,t)\in\rH_{2m|\signp,\signq}$ and all $w\in\ehH_R$,
$$
\Phi^{-1}(\pi(\xx,t)w)=U(\xx,t)\psi_\kP\otimes w \quad\text{and}\quad
\Phi^{-1}(A\pi(\xx,t)w)=U(\xx,t)\psi_\kP\otimes A_R w.
$$
Since $\Phi$ is an intertwining operator, Equation \eqref{Eq:Phi-HR} and the previous equality 
yield
\begin{align*}
\Phi^{-1}\circ A\circ \Phi\left(U(\xx,t)\psi_\kP\otimes w\right)&=\Phi^{-1}(A\pi(\xx,t)  w),\\
&=\left(\gone\otimes A_R\right)\left(U(\xx,t)\psi_\kP\otimes w\right).
\end{align*}
By Theorem \ref{thm-irreduct}, the Schr\"odinger representation is irreducible, i.e., 
$\mathrm{span}\{U(\xx,t)\psi_\kP \ | \ (\xx,t)\in\rH_{2m|\signp-1,\signq}\}$
is dense in $\ehH_S$. Therefore, $A$ being bounded, we deduce that
$\Phi^{-1}\circ A\circ \Phi=\gone\otimes A_R$.
\end{proof}

\begin{lemma}\label{lem:1to2}
Let $\gP$ be the parity operator on a Hilbert superspace $\ehH_R$
of parity $\sigma\in\{0,1\}$ and $(\ehH_R,\pio^R,\pi^R)$ be a strong SUR of $\rH_{0|1,0}$ such that central elements act by $\pi(0,t)=e^{i\hbar t}\gone$, for all $t\in\gR$. We have two cases:
\begin{itemize}
\item if $\sigma=0$, then
$(\ehH_R,\pio^R,\pi^R)$ extends into a strong SUR of $\rH_{0|2,0}$;
\item  if $\sigma=1$, then $(\ehH_R,\pio^R,\pi^R)$ extends into a strong SUR of $\rH_{0|1,1}$.
\end{itemize}
\end{lemma}
\begin{proof}
By Theorem \ref{thm:Equivalence}, $(\ehH_R,\pio^R,\pi^R)$ corresponds to a strong SUR
$(\ehH_R,\pio^R,\piG^R)$ of the super Harish-Chandra pair $(\gR,(\kh_{0|1,0})_\gR)$
associated to $\rH_{0|1,0}$.
Denote by $(e_1,e_2,Z)$ a basis of $(\kh_{0|2,0})_\gR$ and by $(e_1,e_{-1},Z)$
 a basis of $(\kh_{0|1,1})_\gR$, so that $(e_1,Z)$ is a basis of $(\kh_{0|1,0})_\gR$
and the non-vanishing commutators are $[e_1,e_1]=Z$, $[e_2,e_2]=Z$ and
$[e_{-1},e_{-1}]=-Z$.
Then, extend $\piG^R$ by $\gR$-linearity to $(\kh_{0|2,0})_\gR$ by setting
$\piG^R(e_2)=[\frac{i}{2}\gP,\piG^R(e_1)]$ and 
extend it to $(\kh_{0|1,1})_\gR$ by setting
$\piG^R(e_{-1})=[\frac12 \gP,\piG^R(e_1)]$.

As $\piG^R(e_1)$ is an odd operator, it admits a block decomposition 
along the decomposition $\ehH_R=(\ehH_R)_0\oplus(\ehH_R)_1$, namely
$
\piG^R(e_1)=\left(\begin{smallmatrix}
0 & A \\
B & 0
\end{smallmatrix}\right)
$. The block decomposition of the parity operator is 
$\gP=\left(\begin{smallmatrix}
1 & 0 \\
0 & -1
\end{smallmatrix}\right)$, so that 
$$
\piG^R(e_2)=i\begin{pmatrix}
0 & A \\
-B & 0
\end{pmatrix}
\quad \text{and}\quad
\piG^R(e_{-1})=\begin{pmatrix}
0 & A \\
-B & 0
\end{pmatrix} \ .
$$
A direct consequence is 
$$
[\piG^R(e_1),\piG^R(e_2)]=0=[\piG^R(e_1),\piG^R(e_{-1})]\ .
$$
Since $[\piG^R(e_1),\piG^R(e_1)]=\piG^R(Z)=i\hbar\gone$, we have 
$ 2\left(\begin{smallmatrix}
AB & 0 \\
0 & BA
\end{smallmatrix}\right)=i\hbar\gone$ and then
$$
[\piG^R(e_2),\piG^R(e_2)]=i\hbar\gone 
\quad\text{and}\quad 
[\piG^R(e_{-1}),\piG^R(e_{-1})]=-i\hbar\gone \ .
$$
This means that the extensions of $\piG^R$ to $(\kh_{0|2,0})_\gR$ and
to $(\kh_{0|1,1})_\gR$ are Lie superalgebra morphisms. Moreover,
using the equalities $\piG^R(e_1)^\dag=-\piG^R(e_1)$ and $\gP^\dag=(-1)^\sigma\gP$,
we get 
$$
\piG^R(e_2)^\dag=[\piG^R(e_1)^\dag,(\frac{i}{2}\gP)^\dag]=-(-1)^\sigma \piG^R(e_2)
\quad\text{and}\quad
\piG^R(e_{-1})^\dag=[\piG^R(e_{1})^\dag,(\frac12\gP)^\dag]=(-1)^\sigma \piG^R(e_2).
$$
Hence, the extensions of $\piG^R$ yield strong SUR's of the super Harish-Chandra pairs 
$(\gR,(\kh_{0|2,0})_\gR)$ if $\sigma=0$ and
$(\gR,(\kh_{0|1,1})_\gR)$ if $\sigma=1$.
Using again the equivalence stated in Theorem \ref{thm:Equivalence}, 
these become strong SUR's of $\rH_{0|2,0}$ and $\rH_{0|1,1}$ respectively.
\end{proof}

\begin{lemma}\label{lem:oddtoeven}
Let $\signp+\signq$ be odd, $\ehH$ be a Hilbert superspace of parity $\sigma\in\{0,1\}$ 
and $(\ehH,\pio,\pi)$ be a strong SUR of $\rH_{2m|\signp,\signq}$
such that central elements act by $\pi(0,t)=e^{i\hbar t}\gone$, for all $t\in\gR$. We have two cases:
\begin{itemize}
\item if $\sigma=\signq\mod 2$, then
$(\ehH,\pio,\pi)$ extends into a strong SUR of $\rH_{2m|\signp+1,\signq}$
\item  if $\sigma=\signq+1\mod 2$, then $(\ehH,\pio,\pi)$ extends into a strong SUR 
of $\rH_{2m|\signp,\signq+1}$. 
\end{itemize}
\end{lemma}
\begin{proof}
For this proof,  we assume that $\signp>0$ for simplicity of presentation. 
The remaining case ($\signp=0$ and $\signq>0$) is similar and left to the reader.

According to Proposition \ref{Prop:pibounded}, the restrictions of $(\ehH,\pio,\pi)$ to the subgroup 
$\rH_{2m|\signp-1,\signq}\leq\rH_{2m|\signp,\signq}$ and to its centralizer $\rH_{0|1,0}\leq\rH_{2m|\signp,\signq}$ are strong SUR's.
The central elements also act by $\pi(0,t)=e^{i\hbar t}\gone$ for both sub-representations. According to the previous section, there exists a Hilbert superspace $\ehH_R$ (see Lemma \ref{HR} and paragraph above) 
and a superunitary operator $\Phi:\ehH_S\hat{\otimes}\ehH_R\to \ehH$ (see \eqref{IsoH-HSHR}) which
intertwines the representation $(\ehH,\pio,\pi)$, restricted to $\rH_{2m|\signp-1,\signq}$, 
and the tensor product of the Schr\"odinger representation $(\ehH_S,U_0,U)$ of $\rH_{2m|\signp-1,\signq}$
with the trivial representation on $\ehH_R$.

Since $[\pi(\xx,t),\pi(\yy,t')]=0$ for all $(\xx,t)\in\rH_{2m|\signp-1,\signq}$ and $(\yy,t')\in\rH_{0|1,0}$,
Lemma \ref{lem:commutation} applies to the operators $\pi(\yy,t')$. 
As a result, there exists a strong SUR $(\ehH_R,\pio^R,\pi^R)$ of $\rH_{0|1,0}$ such that $\Phi^{-1}\circ\pi(\yy,t')\circ\Phi=\gone\otimes\pi^R(\yy,t')$ for all $(\yy,t')\in\rH_{0|1,0}$. Of course, we have 
$\pi(0,t)=e^{i\hbar t}\gone$ for all $t\in\gR$.

According to Lemma \ref{lem:1to2}, the strong SUR $(\ehH_R,\pio^R,\pi^R)$ of $\rH_{0|1,0}$ extends as a strong SUR
of $\rH_{0|2,0}$ (if $\sigma(\ehH_R)=0$) or $\rH_{0|1,1}$ (if $\sigma(\ehH_R)=1$).
Then, the map 
$$
\widehat{\pi}:(\xx,t)\times (\yy,t')\mapsto \Phi \circ\left(U(\xx,t)\otimes\gone+\gone\otimes \pi^R(\yy,t')\right)\circ\Phi^{-1},
$$
defines a strong SUR $(\ehH,\pio,\widehat{\pi})$ on $\ehH$ of the direct product
$\rH_{2m|\signp-1,\signq}\times\rH_{0|2,0}$ (if $\sigma(\ehH_R)=0$) 
or $\rH_{2m|\signp-1,\signq}\times\rH_{0|1,1}$ (if $\sigma(\ehH_R)=1$).
Since the center of both factors acts in the same way, $(\ehH,\pio,\widehat{\pi})$ descends as a 
strong SUR of $\rH_{2m|\signp+1,\signq}$ (if $\sigma(\ehH_R)=0$) or
$\rH_{2m|\signp,\signq+1}$ (if $\sigma(\ehH_R)=1$). 
By definition of $\Phi^{-1}$ and $\pi^R$,
the restriction of that strong SUR to $\rH_{2m|\signp,\signq}$ is $(\ehH,\pio,\pi)$.
As the parity of $\ehH_S$ is given by $\signp-1\mod 2=\signq\mod 2$,
the represented supergroup is $\rH_{2m|\signp+1,\signq}$ if $\sigma(\ehH)=\signq\mod 2$
and $\rH_{2m|\signp,\signq+1}$ if $\sigma(\ehH)=\signq+1\mod 2$.
\end{proof}

Consider  a strong SUR $(\ehH,\pio,\pi)$ of $\rH_{2m|\signp,\signq}$,
with $\signp+\signq$ odd and $\ehH$ of parity $\sigma\in\{0,1\}$, 
such that central elements act by $\pi(0,t)=e^{i\hbar t}\gone$ for all $t\in\gR$.
By Lemma \ref{lem:oddtoeven}, $(\ehH,\pio,\pi)$
admits an extension as a strong SUR of $\rH_{2m|\signp+1,\signq}$ if $\sigma=\signq\mod 2$
 or of $\rH_{2m|\signp,\signq+1}$ if $\sigma=\signq+1\mod 2$.
Denote by $(\ehH_S,U_0,U)$ the Schr\"odinger representation with parameter $\hbar\in\gR^\times$ of the extended supergroup
$\rH_{2m|\signp+1,\signq}$ or $\rH_{2m|\signp,\signq+1}$.
As the odd dimension of both supergroups is even, Theorem \ref{thm:StonevN}  applies to the extension of $(\ehH,\pio,\pi)$.
Hence, there exists a Hilbert superspace $\ehH_R$ and a superunitary intertwiner $\Phi:\ehH_S\otimes\ehH_R\to\ehH$ between the strong SUR's
$(\ehH,\pio,\pi)$ and $(\ehH_S\otimes\ehH_R,U_0\otimes\gone,U\otimes\gone)$  of the extended supergroup
($\rH_{2m|\signp+1,\signq}$ or $\rH_{2m|\signp,\signq+1}$).
The operator $\Phi$ interwines automatically the restrictions of both strong SUR's to $\rH_{2m|\signp,\signq}$.
According to Definition \ref{Def:SchrodingerRep}, 
the restriction of $(\ehH_S,U_0,U)$ to $\rH_{2m|\signp,\signq}$ is the Schr\"odinger representation of $\rH_{2m|\signp,\signq}$
with parameter $\hbar$ and parity $\sigma$.
This concludes the proof of Theorem \ref{thm:StonevN}.

\section{Applications}

We provide now several applications of the Stone-von Neumann Theorem \ref{thm:StonevN}.

\subsection{Strong superunitary dual}
\label{subsec-dual}


Recall that two strong SUR's of the Heisenberg supergroup $\rH_{2m|\signp,\signq}$ 
are equivalent if and only if there exists a superunitary intertwiner between them.

\begin{definition}
The {\defin strong superunitary dual} of the Heisenberg supergroup $\rH_{2m|\signp,\signq}$ 
is the set of equivalence classes of graded-irreducible strong SUR's of $\rH_{2m|\signp,\signq}$.
\end{definition}

In the following, we will prove that every graded-irreducible strong SUR (strong SUIR) 
of $\rH_{2m|\signp,\signq}$  is built from the following ones:
\begin{itemize}
\item the trivial representations $(\gC_j,\gone,\gone)$, where $j\in\{0,1,2,3\}$
and $\gC_j$ is the Hilbert superspace of dimension $1$ such that a normalized vector
$e_j\in\gC_j$ satisfies $\langle e_j,e_j\rangle=i^j$;
\item the representations $(\gC_j,\pio^{a,b},\pi^{a,b})$, where $j\in\{0,1,2,3\}$
and $a,b\in\gR^m$ are such that
$$
\pi_0^{a,b}(\gB q,\gB p,t)=e^{i(a\gB q+b\gB p)} \quad\text{and}\quad
\pi^{a,b}(q,p,\zeta,\bzeta,t)=e^{i(aq^{ev}+bp^{ev})},
$$
for all $(q,p,\zeta,\bzeta,t)\in\rH_{2m|\signp,\signq}$ (see notations of \eqref{Def:M});
\item the Schr\"odinger representations with parameter $\hbar\in\gR^\times$, 
and parity $\sigma\in\{0,1\}$ if $\signp+\signq$ is odd, 
denoted by $(\ehH_S^{(\sigma)},U_0^\hbar,U^\hbar)$.
\end{itemize}

\begin{lemma}
\label{lem-centerdual}
Let $(\ehH,\pio,\pi)$ be a strong SUIR of $\rH_{2m|\signp,\signq}$. 
There exists $\hbar\in\gR$  such that, for all $t\in\gR$, the elements $(0,t)$ in the center of 
$\rH_{2m|\signp,\signq}$ act by $\pi(0,t)=e^{i\hbar t}\gone$.
\end{lemma}
\begin{proof}
By definition of a strong SUR, the representation $(\ehH,\pio)$ is unitarizable.
Namely, there exists a fundamental symmetry $J$ turning $(\ehH,(-,-)_J)$ into a Hilbert
space and such that the operators $\pio(x,t)$ are $J$-unitary (i.e. unitary w.r.t. $(-,-)_J$)
for all $(x,t)$ in the body of $\rH_{2m|\signp,\signq}$. 
The operators $i\piG(X)$, with $X\in(\gB\kh_{2m|\signp,\signq})_0$, are then $J$-selfadjoint operators. 
In particular, denoting by $Z$ the infinitesimal generator of the center, with $e^Z=(0,1)$, 
the operator $i\pi_*(Z)$ is a $J$-selfadjoint operator of degree $0$. 

Assume that the spectrum of $i\pi_*(Z)$ is not a singleton. Then, its spectral measure 
splits into two non-trivial parts, which  produce two $J$-orthogonal projectors $P_1$, $P_2$ of degree $0$.
Since $Z$ is in the center, the operators $P_1$ and $P_2$ commute with 
all the representing operators $\pi(\xx,t)$, with $(\xx,t)\in\rH_{2m|\signp,\signq}$. 
Hence, $P_1\ehH$ and $P_2\ehH$ are proper closed $(\pio,\pi)$-invariant graded 
subspaces of $\ehH$. It is a contradiction with the graded-irreducibility of $\pi$. 
Therefore, the spectrum of $i\pi_*(Z)$ is a singleton, 
that we denote by $\hbar\in\gR$. This means that $\pi(0,t)=e^{i\hbar t}\gone$
for all $t\in\gR$.
\end{proof}

\begin{Theorem}
\label{prop-suirdual}
Up to equivalence, the strong SUIR's of $\rH_{2m|\signp,\signq}$ are given by 
\begin{enumerate}[label=(\roman*)]
\item the strong SUIR's $(\gC_j,\pio^{a,b},\pi^{a,b})$ with $j\in\{0,1,2,3\}$
and $a,b\in\gR^m$;
\item the strong SUIR's $(\ehH_S^{(\sigma)}\otimes\gC_j,U_0^\hbar\otimes\gone,U^\hbar\otimes\gone)$, 
with $j\in\{0,1,2,3\}$ and $\hbar\in\gR^\times$, as well as $\sigma\in\{0,1\}$ if $\signp+\signq$ is odd.
\end{enumerate}
Moreover, the above strong SUIR's are all inequivalent if  $\signp+\signq$ is even.
If $\signp+\signq$ is odd, the only isomorphic ones are those of type $(ii)$ 
with same parameters $\sigma\in\{0,1\}$ and $\hbar\in\gR^\times$,
and whose parameters $j,j'\in\{0,1,2,3\}$ satisfy
$i^{ j'+j-1}=(-1)^{\signq} \mathrm{sign}(\hbar)$.
\end{Theorem}
\begin{proof}
\textbf{I.} We start by proving that any strong SUIR of 
the Heisenberg supergroup $\rH_{2m|\signp,\signq}$ is of the type $(i)$ or $(ii)$.
Consider a strong SUIR $(\ehH,\pio,\pi)$ of $\rH_{2m|\signp,\signq}$.
By Lemma \ref{lem-centerdual}, there exists $\hbar\in\gR$ such that, for all $t\in\gR$, 
the elements $(0,t)$ in the center of $\rH_{2m|\signp,\signq}$ act by $\pi(0,t)=e^{i\hbar t}\gone$.

Assume that $\hbar\neq 0$. Then, the Stone-von Neumann Theorem \ref{thm:StonevN} 
applies. As a result, $(\ehH,\pio,\pi)$ is equivalent to the tensor product
$(\ehH_S^{(\sigma)}\otimes\ehH_R,U_0^\hbar\otimes\gone,U^\hbar\otimes\gone)$,
with $\sigma\in\{0,1\}$ if $\signp+\signq$ is odd. Since $(\ehH,\pio,\pi)$ is graded-irreducible, 
$(\ehH_R,\gone,\gone)$ is also graded-irreducible. This implies that $\ehH_R$ is one-dimensional,
equal to $\gC_j$ for some $j\in\{0,1,2,3\}$. 

Assume that $\hbar=0$, i.e., the center acts trivially. Then, 
$(\ehH,\pio,\pi)$ descends as a strong SUIR of the Abelian supergroup $\gR^{2m|\signp+\signq}$. 
The infinitesimal representation induces a representation $\pi_*$ of the 
universal enveloping algebra 
$\kU(\gR^{2m|\signp+\signq})=\kU(\gR^{2m|0})\otimes\kU(\gR^{0|\signp+\signq})$, 
where $\kU(\gR^{0|\signp+\signq})\simeq\bigwedge \gR^{\signp+\signq}$ is a Grassmann algebra. 
Consider an element $X\in\kU(\gR^{0|\signp+\signq})$ of maximal (tensor) degree~$j$ such that
\begin{equation*}
\pi_*(X)\neq 0,\qquad \pi_*(\bigwedge{}^{j+1}\,\gR^{\signp+\signq})=\algzero.
\end{equation*}
Then, $\pi_*(X)\ehH$ is a closed $(\pio,\pi)$-invariant graded subspace of $\ehH$.
Since $(\ehH,\pio,\pi)$ is graded-irreducible and $\piG(X)\neq 0$, this implies that $\pi_*(X)\ehH=\ehH$.
Hence, $\pi_*(X^2)=\pi_*(X)^2\neq 0$. As $X^2$ is of degree $2j$, this is possible only if $j=0$. 
In other words, $\kU(\gR^{0|\signp+\signq})$ acts trivially on $\ehH$ and $(\ehH,\pio)$
is then an irreducible representation of $\gR^{2m}$. 
As $(\ehH,\pio)$ is unitarizable, the space $\ehH$ is of dimension $1$, equal to $\gC_j$ for some $j\in\{0,1,2,3\}$, 
and $\pio=\pio^{a,b}$ for some $a,b\in\gR^m$ (see e.g \cite{Folland:1989}).
\\

\textbf{II.} Now, we prove that any two strong SUIR's of type $(i)$ or $(ii)$
are inequivalent, except in the stated case.

If two strong SUIR's are equivalent, the action of the central elements is the same.
Hence, strong SUIR's of type $(i)$ and $(ii)$ are inequivalent and two strong SUIR's
of type $(ii)$ are equivalent only if they have same parameter $\hbar$.
Moreover, equivalent strong SUIR's have isometric Hilbert superspaces.
Hence, two strong SUIR's of type $(i)$ are isometric  only if they have same parameter $j$
and, if $\signp+\signq$ is odd, two strong SUIR's of type $(ii)$ are isometric 
only if they have same parameter $\sigma$.
Clearly, two equivalent strong SUIR's restrict as equivalent unitarizable representations
of the body group. Hence, $(\gC_j,\pio^{a,b})$ and $(\gC_{j},\pio^{a',b'})$ are equivalent
only if $a=a'$ and $b=b'$ (see e.g \cite{Folland:1989}).

Consider two equivalent strong SUIR's of type $(ii)$, 
$(\ehH_S^{(\sigma)}\otimes\gC_j,U_0^\hbar\otimes\gone,U^\hbar\otimes\gone)$ and
$(\ehH_S^{(\sigma)}\otimes\gC_{j'},U_0^{\hbar}\otimes\gone,U^{\hbar}\otimes\gone)$,
intertwined by 
$$
\Psi:\ehH_S^{(\sigma)}\otimes\gC_j\to\ehH_S^{(\sigma)}\otimes\gC_{j'}.
$$
It only remains to determine which are the possible values of $j,j'\in\{0,1,2,3\}=\gZ_4$.
We define the maps $\Phi_{j,j'}:\ehH_S\otimes\gC_{j}\to\ehH_S\otimes\gC_{j'}$, 
given by $\Phi_{j,j'}(\varphi\otimes e_{j}):=\varphi\otimes e_{j'}$. These maps are invertible,
and satisfy
\begin{equation*}
(\Phi_{j,j'})^\dag\Phi_{j,j'}=(-1)^{(j+\sigma)(j'-j)} i^{j'-j}\gone.
\end{equation*}
There exists an operator $T$ such that $\Psi=\Phi_{j,j'}\circ (T\otimes\gone)$, 
with $T\in\caB(\ehH_S^{(\sigma)})$. It satisfies $|T|=j'-j\mod 2$ and $TU^\hbar(g)=U^\hbar(g)T$ 
for any $g\in\rH_{2m|\signp,\signq}$. 

If $\signp+\signq$ is even, then $T$ is a homogeneous intertwiner of the Schr\"odinger representation.
The latter is irreducible by Theorem \ref{thm-irreduct}, so that Schur's Lemma (Proposition \ref{Schur})
applies. Therefore, we have $T=\lambda\gone$ for some $\lambda\in\gC$. 
The superunitarity of $\Psi=\lambda\Phi_{j,j'}$ implies that $j=j'$. Hence, all listed SUIR's are inequivalent
if $\signp+\signq$ is even.

If $\signp+\signq$ is odd, then $T$ is an intertwiner of degree $j'-j$ of the Schr\"odinger representation,
which is graded-irreducible by Corollary \ref{cor-irreduct}.
By Proposition \ref{Schur}, either $T=\lambda\gone$ if $j'-j$ is even 
or $T=\lambda A$ for some fixed odd operator $A$ if $j'-j$ is odd. 
In the first case, the superunitarity of $\Psi=\lambda\Phi_{j,j'}$ implies that $j=j'$.
In the remaining case, the operator $A$ can be defined from the extension of 
the Schr\"odinger representation to the supergroups $\rH_{2m|\signp+1,\signq}$ or $\rH_{2m|\signp,\signq+1}$
(see Lemma \ref{lem:oddtoeven}), as  the representing operator of the extra dimension of those supergroups.
It satisfies
$$
A^2=(-1)^{\sigma+\signq}\frac{i\hbar}{2}\gone,\quad A^\dag=-A \quad \text{and}
\quad AU^\hbar(g)=U^\hbar(g)A,
$$
for all $g\in\rH_{2m|\signp,\signq}$. 
Then, the operator $\lambda\Phi_{j,j'}\circ (A\otimes\gone)$ is a superunitary 
intertwiner if and only if the following operator is the identity:
\begin{align*}
(\lambda\Phi_{j,j'}\circ A)^\dag(\lambda\Phi_{j,j'}\circ A)&=
-|\lambda|^2 A\circ\Phi_{j,j'}^\dag\circ\Phi_{j,j'}\circ A\\
&=(-1)^{\sigma+j+1} i^{j'-j}|\lambda|^2 A\circ A\\
&=(-1)^{\signq+j+1} i^{j'-j+1} |\lambda|^2 \frac{\hbar}{2}\gone.
\end{align*}
This means that $|\lambda|=\sqrt{\frac{2}{|\hbar|}}$ and $i^{ j'+j-1}=(-1)^{\signq} \mathrm{sign}(\hbar)$.
Hence, the only non-trivial superunitary intertwiner occurs for $\signp+\signq$ odd 
and relates representations with same parameters $\hbar$ and
$\sigma$, whose parameters $j,j'$ satisfy $i^{ j'+j-1}=(-1)^{\signq} \mathrm{sign}(\hbar)$.
\end{proof}

\subsection{Fourier transformation}

In this section, we construct the group Fourier transformation and exhibit the Plancherel measure 
of the Heisenberg supergroup $G=\rH_{2m|\signp,\signq}$, with $\signp+\signq$ even. 
This extends the previous work \cite{Alldridge:2013} where $\signq$ is assumed to be $0$.
 Other works on the Fourier transformation over supermanifolds can be found in 
\cite{Rempel:1983,Brackx:2005,DeBie:2008}. 

The Schr\"odinger representation of $G$ with parameter $\hbar$ is denoted here
by $(\ehH^\hbar_S,U_0^\hbar,U^\hbar)$. 

\begin{definition}
The {\defin group Fourier transformation} of a function $f\in L^1(G)$ is the following operator family,
indexed by $\hbar\in\gR^\times$,
\begin{equation*}
\hat f(\hbar):=\int_G\dd g\ f(g)U^{-\hbar}(g)\in\caB(\ehH_S^{-\hbar}).
\end{equation*}
\end{definition}
Due to the unimodularity of $G$, the group Fourier transformation obviously satisfies 
$\hat{\overline{f}}(\hbar)=(\hat f(\hbar))^\dag$ for all $\hbar\in\gR^\times$. 
Following \cite{Bieliavsky:2010su}, we define the supertrace of rank one operators by
\begin{equation*}
\Str(|\psi\rangle\langle\varphi|):=\langle\varphi,\gP\psi\rangle.
\end{equation*}
The resolution of the identity (Proposition \ref{prop-resol}) permits to extend 
the supertrace to other operators.

\begin{definition}
Let $T\in\caB(\ehH_\hbar)$ with trace class, and $\varphi,\psi\in\caS(\gR^{m|r})$ 
with $\langle\varphi,\psi\rangle\neq 0$. We define the {\defin supertrace} of $T$ as
\begin{equation*}
\Str(T):=\frac{1}{\kappa_\hbar\langle\varphi,\psi\rangle}\int\dd \xx\ \langle U^\hbar(\xx)\varphi,TU^\hbar(\xx)\psi\rangle.
\end{equation*}
\end{definition}

The above definition does not depend on the choice of $\varphi,\psi$.
The group Fourier transformation admits an  inversion formula, as given below.

\begin{theorem}[Inversion formula]
\label{thm-inv}
Let $\mu$ be the measure on $\gR$
given by $\dd\mu(\hbar):=\frac{|\hbar|}{2\pi\kappa_{-\hbar}}\dd \hbar$.
For any $f\in L^1(G)$ and $g\in G$, we have
\begin{equation*}
f(g)=\int_\gR \dd\mu(\hbar)\ \Str\big(\hat f(\hbar)U^{-\hbar}(g^{-1})\big).
\end{equation*}
The above measure $\mu$ is called the Plancherel measure of $G$.
\end{theorem}
\begin{proof}
We need a Lemma.

\begin{lemma}
\label{lem-supertrace}
For any $\hbar\in\gR^\times$ and $g=(q,p,\zeta,\bzeta,t)\in G$, we have
\begin{equation*}
\Str(U^\hbar(g))=\kappa_\hbar e^{i\hbar t}\delta(p)\delta(q)\delta(\zeta,\bzeta),
\end{equation*}
with the following Dirac distributions $\delta(q)=\delta(\gB q)q_{m+1}\cdots q_{m+r}$, $\delta(p)=\delta(\gB p)p_{m+1}\cdots p_{m+r}$, and $\delta(\zeta,\bzeta)=\frac{1}{(2i)^s}\bzeta_1\zeta_1\cdots\bzeta_s\zeta_s$.
\end{lemma}
\begin{proof}
By definition, we have
\begin{multline*}
\Str(U^\hbar(g_1))=(2i)^{2s}\int\dd q\dd p\dd\zeta\dd\bzeta\dd q_0\dd\zeta_0\dd\bzeta_0\ \overline{\varphi(q_0-q)}\psi(q_0-q-q_1,\zeta_0-\zeta-\zeta_1)\\
 e^{i\hbar (t_1+(\frac12q_1-q_0)p_1+\frac12(\frac12\zeta_1-\zeta_0)\bzeta_1+q_1p+\frac12(\zeta+\zeta_1-\zeta_0)\bzeta +\frac12(\zeta_0-\zeta)\bzeta_0)}.
\end{multline*}
Using the identities \eqref{eq-dirac0} on the variable $p$ and \eqref{eq-dirac1} on the variable $\bzeta_0$, we obtain
\begin{equation*}
\Str(U^\hbar (g_1))=(2i)^{s}\kappa_\hbar\int\dd q\dd\zeta\dd\bzeta\dd q_0\ \delta(q_1)\overline{\varphi(q)}\psi(q,-\zeta_1)  e^{i\hbar (t_1-q_0p_1+\frac12\zeta_1\bzeta+\frac12(\frac12\zeta_1-\zeta)\bzeta_1)}.
\end{equation*}
The result follows then from the use of Equations \eqref{eq-dirac0}-\eqref{eq-dirac1}.
\end{proof}

By definition of the group Fourier transformation, the right hand side
 of the inversion formula reads as
\begin{equation*}
RHS=\int\dd\mu(\hbar)\ \Str\Big(\int\dd g'\ f(g')U^{-\hbar}(g')U^{-\hbar}(g^{-1})\Big)=\int\dd\mu(\hbar)\ \Str\Big(\int\dd g'\ f(g'g)U^{-\hbar}(g')\Big).
\end{equation*}
If we choose a function $\varphi\in\ehH_{-\hbar}$ of degree $r$ in the definition of the supertrace, 
we can commute the supertrace with the integration over $g'$ and with the function $f$. 
Then, we use Lemma \ref{lem-supertrace} to get
\begin{multline*}
RHS=\int\dd\mu(\hbar)\dd g'\ f(g'g)\kappa_{-\hbar}e^{-i\hbar t'}\delta(p')\delta(q')\delta(\zeta',\bzeta')\\
=\int\dd\mu(\hbar)\dd t'\ f((0,t').g)\kappa_{-\hbar}e^{-i\hbar t'}=f(g),
\end{multline*}
due to the definition of the Plancherel measure $\mu$.
\end{proof}

\begin{corollary}
We have the Parseval-Plancherel formula for the superhermitian inner product on $G$: 
\begin{equation*}
\int\dd g\ \overline{f_1(g)}f_2(g)=\int\dd\mu(\hbar)\ \Str\big(\hat f_1(\hbar)^\dag\hat f_2(\hbar)\big),
\end{equation*}
for any $f_1,f_2\in L^2(G)$.
\end{corollary}
\begin{proof}
Indeed, we have
\begin{equation*}
\hat f_1(\hbar)^\dag\hat f_2(\hbar)=\int\dd g\dd g'\ \overline{f(g)}f(gg') U^{-\hbar}(g').
\end{equation*}
Using Lemma \ref{lem-supertrace}, we get
\begin{equation*}
\Str(\hat f_1(\hbar)^\dag\hat f_2(\hbar))=\int\dd g\dd t'\ \overline{f(g)} f(g.(0,t'))\kappa_{-\hbar} e^{-i\hbar t'}.
\end{equation*}
The result follows from integration of the $\hbar$-variable with respect to the Plancherel measure.
\end{proof}

The regular representation $(L^2(G),L_0,L)$ of the Heisenberg supergroup $G$ is defined by
\begin{align*}
\forall g_0\in \gB G=\rH_{2m},\  g'\in G,\ f\in L^2(G),\quad  L_0(g_0)\left(f(g')\right)&:=f((g_0)^{-1}g'),\\ 
\forall g, g'\in G,\ f\in (L^2(G))^\infty,\quad L(g)\left(f(g')\right)&:=f((g)^{-1}g),\\ 
\end{align*}
with $(L^2(G))^\infty$ the space of smooth functions on $G$, whose all derivatives are square integrable.

\begin{corollary}
 The regular representation $(L^2(G),L_0,L)$
can be decomposed 
via the group Fourier transformation: 
\begin{align*}
\forall\hbar\in\gR^\times,\ g\in \gB G,\ f\in L^2(G),\ \quad\widehat{L_0(g)f}\, (\hbar)&=U^{-\hbar}_0(g)\circ \hat f(\hbar),\\
\forall\hbar\in\gR^\times,\ g\in G,\ f\in (L^2(G))^\infty,\ \quad\widehat{L(g)f}\, (\hbar)&=U^{-\hbar}(g)\circ \hat f(\hbar).
\end{align*}
\end{corollary}
\begin{proof}
It is a direct consequence of the inversion formula, with a change of variable.
\end{proof}

\subsection{Metaplectic representation}\label{Sec:meta}

Let  $(\modE,\om)$ be a symplectic superspace of dimension
$2m|\signp+\signq$. 
We use the Schr\"odinger representation of the Heisenberg supergroup 
$\rH(\modE_0,\omega)$ and previous works \cite{Nishiyama:1990}
to build the metaplectic representation as a SUR of the  metaplectic supergroup. 
For simplicity, we assume that the signature is of the form $(\signp,\signq)=(r+2s,r)$. 
The other cases admit a similar treatment and are left to the reader.


We start with algebraic preliminaries, following \cite{Nishiyama:1990}.
The  orthosymplectic $\gR$-Lie superalgebra is defined as
\begin{equation*}
\mathfrak{spo}(\modE,\omega)_\gR:=\{A\in\caL(\gB\modE)\;|\; \forall X,Y\in\gB\modE,\ \omega(AX,Y)+(-1)^{|A||X|}\omega(X,AY)=0\}. 
\end{equation*}
Denote by $\kg_\gR:=(\kh(\modE,\om))_\gR$ the Heisenberg $\gR$-Lie superalgebra. 
The canonical Lie supebracket on the quotient algebra $\mathfrak U(\kg_\gR)/\langle Z-\gone\rangle$
induces a $\gR$-Lie superalgebra structure on the space of quadratic elements,
\begin{equation*}
\caL_2:=\mathrm{span}\{XY+(-1)^{|X||Y|}YX \; |\; X,Y\in\kg_\gR\}\subseteq \mathfrak U(\kg_\gR)/\langle Z-\gone\rangle,
\end{equation*}
and a linear action of $\caL_2$ on the space of degree $1$ elements,
$\gB \modE\subseteq \mathfrak U(\kg_\gR)/\langle Z-\gone\rangle$.

\begin{proposition}[\cite{Nishiyama:1990}, Theorem 3.5]
\label{prop-nishi}
The linear action of $\caL_2$ on $\gB \modE$ defines an isomorphism of 
$\gR$-Lie superalgebras between $\caL_2$ and $\mathfrak{spo}(\modE,\omega)_\gR$.
\end{proposition}

In the sequel, we do not distinguish between $\caL_2$ and $\mathfrak{spo}(\modE,\omega)_\gR$.
Consider the basis $(c_i,e_\alpha, Z)$ of $\kg_\gR$, with $(c_i)_{1\leq i\leq 2m}$ a basis of $\gR^{2m}$ 
with degree 0 and $(e_\alpha)_{1\leq \alpha\leq 2(r+s)}$ a basis of $\gR^{2(r+s)}$ with degree 1 (see Equation \eqref{eq-baseheis}).
It induces the following basis of $\caL_2\simeq\mathfrak{spo}(\modE,\omega)_\gR$:
\begin{equation*}
m_{ij}:=c_ic_j+c_jc_i,\qquad s_{\alpha\beta}:=e_\alpha e_\beta-e_\beta e_\alpha,\qquad t_{i\alpha}=c_ie_\alpha + e_\alpha c_i,
\end{equation*}
where $1\leq i,j\leq m$ and $1\leq \alpha,\beta\leq 2(r+s)$.
The Schr\"odinger representation $(\ehH_S,U_0,U_*)$ with parameter $\hbar$ induces 
a superalgebra morphism, $U_*:\mathfrak{U}(\kg_\gR)/\langle Z-i\hbar\gone\rangle\to\caL(\ehH_S^\infty)$.
We define the map $\mu_*:\mathfrak{spo}(\modE,\omega)_\gR\to\caL(\ehH_S^\infty)$  by setting
$\mu_*:=\frac{1}{i\hbar} (U_*)|_{\caL_2}$. In the above basis we have
\begin{align}
&\mu_*(m_{ij}):=\frac{1}{i\hbar}\big(U_*(c_i)U_*(c_j)+U_*(c_j)U_*(c_i)\big)= -\frac{2i}{\hbar}U_*(c_i)U_*(c_j)-\omega_{ij},\nonumber\\
&\mu_*(s_{\alpha\beta}):=\frac{1}{i\hbar}\big(U_*(e_\alpha)U_*(e_\beta)-U_*(e_\beta)U_*(e_\alpha)\big)= -\frac{2i}{\hbar}U_*(e_\alpha)U_*(e_\beta)-\omega_{\alpha\beta},\nonumber\\
&\mu_*(t_{i\alpha}):=\frac{1}{i\hbar}\big(U_*(c_i)U_*(e_\alpha)+U_*(e_\alpha)U_*(c_i)\big)= -\frac{2i}{\hbar}U_*(c_i)U_*(e_\alpha)=-\frac{2i}{\hbar}U_*(e_\alpha)U_*(c_i).\label{eq-meta}
\end{align}

\begin{proposition}[\cite{Nishiyama:1990}, Proposition 4.6]
\label{prop-nishibis}
The map $\mu_*:\mathfrak{spo}(\modE,\omega)_\gR\to\caL(\ehH_S^\infty)$ 
is a $\gR$-Lie superalgebra morphism satisfying $\mu_*(X)^\dag=-\mu_*(X)$
for all $X\in\mathfrak{spo}(\modE,\omega)_\gR$.
\end{proposition}

Let $\mathrm{GL}(\modE_0)$ be the supergroup of invertible $\superA$-linear transformations of $\modE_0$.
The orthosymplectic supergroup $\SpO(\modE_0,\omega)$ is defined as 
the sub-supergroup of $\mathrm{GL}(\modE_0)$ preserving the symplectic form $\omega$,
 $$
\SpO(\modE_0,\omega):=\{M\in\mathrm{GL}(\modE_0)\; |\;  \forall X,Y\in\modE_0,\ \omega(MX,MY)=\omega(X,Y)\}.
$$
As soon as both $\signp$ and $\signq$ are non-zero, $\SpO(\modE_0,\omega)$ is not connected
and we denote by $\SpO^0(\modE_0,\omega)$  its connected component at the identity.
The super Harish-Chandra pair of $\SpO^0(\modE_0,\omega)$ is 
$$(\Sp(2m,\gR)\times \SO^0(\signp,\signq),\mathfrak{spo}(2m|\signp,\signq)_\gR),$$
where the adjoint action $\Ad$ is given by
$$
\forall M\in\Sp(2m,\gR)\times \SO^0(\signp,\signq),\
A\in\mathfrak{spo}(2m|\signp,\signq)_\gR,\quad\Ad_M A=MAM^{-1},$$ 
the product being taken in the algebra $\caL(\gB \modE)$.
We denote by $\Mp(2m,\gR)$ the metaplectic group, $\Spin^0(\signp,\signq)$ the connected component at the identity of the spin group, and $$\rho:\Mp(2m,\gR)\times \Spin^0(\signp,\signq)\to \Sp(2m,\gR)\times \SO^0(\signp,\signq)$$
the canonical $\gZ_2\times\gZ_2$-covering.

\begin{definition}
The {\defin metaplectic supergroup} $\Mp(\modE_0,\omega)$ is a $\gZ_2\times\gZ_2$-covering 
of the orthosymplectic supergroup $\SpO^0(\modE_0,\omega)$. Its associated 
super Harish-Chandra pair is $(\Mp(2m,\gR)\times \Spin^0(\signp,\signq),\mathfrak{spo}(2m|\signp,\signq)_\gR)$,
where the adjoint action $\widetilde{\Ad}$ is given by 
$$\widetilde{\Ad}_M A=\Ad_{\rho(M)}A,$$
for all $M\in\Mp(2m,\gR)\times \Spin(\signp,\signq)$
and $A\in\mathfrak{spo}(2m|\signp,\signq)_\gR$.
\end{definition}

If $\signp=\signq=0$, then $\mathfrak{spo}(\modE,\omega)_\gR\simeq\mathfrak{sp}(2m,\gR)$
and the representation $\mu_*$ is known to integrate as a unitary representation of $\Mp(2m,\gR)$
on $L^2(\gR^m)$, called the metaplectic representation \cite{Folland:1989}.
If $m=0$, then $\mathfrak{spo}(\modE,\omega)_\gR\simeq\mathfrak{o}(\signp,\signq)$.
According to Section \ref{subsec-spinor}, the representation $U_*$ defines an algebra isomorphism
$$\End(L^2(\gR^{0|r})\otimes Hol(\gC^{0|s}))\simeq\mathbb{C}\ell(2n).$$ Since the Lie algebra
$\mathfrak{o}(\signp,\signq)$ exponentiates in the Clifford algebra $\mathbb{C}\ell(2n)$ 
as the spin group $\Spin(\signp,\signq)$, 
the representation $\mu_*$ integrates as a SUR of $\Spin(\signp,\signq)$
on $L^2(\gR^{0|r})\otimes Hol(\gC^{0|s})$, which is the spinor representation.
Let $\mu_0: \Mp(2m,\gR)\times \Spin^0(\signp,\signq)\to \caU(\ehH_S)$ be the tensor product 
of the metaplectic representation of $\Mp(2m,\gR)$ with the spin representation 
of $\Spin^0(\signp,\signq)$. 

\begin{theorem}
The triple $(\ehH_S,\mu_0,\mu_*)$ is a SUR of the metaplectic supergroup $\Mp(\modE_0,\omega)$ 
on the Hilbert superspace $\ehH_S$. It is called the {\defin metaplectic representation}.
\end{theorem}
\begin{proof}
Let us check the axioms of Definition \ref{def-reprhc}. By definition, 
$\mu_0$ is a SUR of the Lie group $\Mp(2m,\gR)\times \Spin^0(\signp,\signq)$ and we have 
$\mu_*=\dd\mu_0$ on the even part of $\mathfrak{spo}(\modE,\omega)_\gR$.
By Proposition~\ref{prop-nishibis}, the map $\mu_*:\mathfrak{spo}(\modE,\omega)_\gR\to\caL(\ehH_S^\infty)$  is a $\gR$-Lie superalgebra morphism
satisfying $\mu_*(X)^\dag=-\mu_*(X)$
for all $X\in\mathfrak{spo}(\modE,\omega)_\gR$.
Since $\Mp(2m)\times \Spin^0(\signp,\signq)$ is connected, the condition \eqref{Eq:HC-rep}
 is automatically satisfied.

To conclude the proof, it suffices to show that
the space $\ehH_\mu^\infty$ of smooth vectors of the representation $\mu_0$
is equal to $\ehH_S^\infty$.
On one side, by definition of $\mu_0$, we have $\ehH_S^\infty\subseteq\ehH_\mu^\infty$.
On the other side, the restriction of $d\mu_0$ to $\mathfrak{sp}(2m,\gR)$ acts by operators 
of type $\partial_{x_i}\partial_{x_j}\otimes\gone$, $x_i\partial_{x_j}\otimes\gone$, and $x_ix_j\otimes\gone$ 
on $L^2(\gR^m)\otimes\left(L^2(\gR^{0|r})\otimes Hol(\gC^{0|s})\right)$.
Hence, $\ehH^\infty_\mu$ is included in all the domains of powers of the operator $(1+x^2-\partial_x^2)\otimes \gone$,
that is $\ehH^\infty_\mu\subseteq \caS(\gR^{m})\otimes\left(L^2(\gR^{0|r})\otimes Hol(\gC^{0|s})\right)$
(see \cite[page 141]{Reed:1980}).
Since $\ehH^\infty_S=\caS(\gR^{m|r+s})$ (see Proposition \ref{prop-smoothschwartz}),
the result follows.
\end{proof}

The $\gZ_2\times\gZ_2$-covering introduced above extend to the supergroups,
$\rho:\Mp(\modE_0,\omega)\to \SpO^0(\modE_0,\omega)$.
If $\xx\in\modE_0$ and $M\in\Mp(\modE_0,\omega)$, we denote by $\rho(M)\xx$ the result of
the natural action $\SpO^0(\modE_0,\omega)\times\modE_0\to\modE_0$ on the pair $(\rho(M),\xx)$.

\begin{proposition}
The metaplectic and the Schr\"odinger representations satisfy: 
\begin{equation*}
\forall M\in\Mp(\modE_0,\omega),g\in\rH(\modE_0,\om),\quad U(\rho(M)g)=\mu(M)U(g)\mu(M)^{-1},
\end{equation*}
where, using the decomposition $g=(\xx,t)\in\modE_0\times\gR^{1|0}$, we have $\rho(M)g=(\rho(M)\xx,t)$. 
\end{proposition}
\begin{proof}
Since the body of the metaplectic supergroup is connected, it is enough to check the relation at the infinitesimal level on $\ehH^\infty_\mu=\ehH_S^\infty$, that is 
\begin{equation*}
\forall A\in\mathfrak{spo}(\modE,\omega)_\gR,\ \forall X\in\gB\modE,\quad U_*(AX)=[\mu_*(A),U_*(X)],
\end{equation*}
where $AX\in\gB\modE$ is the result of the standard action of $\mathfrak{spo}(\modE,\omega)_\gR$ on $\gB\modE$.
Identifying $X$ and $A$ as elements of $\mathfrak{U}(\kg_\gR)/\langle Z-i\hbar\gone\rangle$
of order $1$ and $2$, we obtain $AX=\frac{1}{i\hbar}[A,X]$ where the commutator is taken in
the algebra $\mathfrak{U}(\kg_\gR)/\langle Z-i\hbar\gone\rangle$. Since $U_*$ is an algebra morphism,
we get that
$$
U_*(AX)=\frac{1}{i\hbar} U_*([A,X])= [\frac{1}{i\hbar}U_*(A),U_*(X)]. 
$$
The result follows from the definition of $\mu_*$.
\end{proof}

%

\appendix \section{Graded $\BCH$ series}

%
%

Let $\kg$ be a finite dimensional Lie superalgebra over the algebra $\superA$
and $\caN_\superA$ the ideal of nilpotent elements in $\superA$.
For any $X\in\caN_\superA\otimes\gB \kg$, the exponential of $X$ is given by the sum
\begin{equation*}
\exp(X)=\sum_{k=0}^\infty\frac{X^k}{k!} .
\end{equation*}
Since $X$ has nilpotent coefficients, the sum is finite and $\exp(X)$
is a formal polynomial over $\superA$.
We set $\kg^{(1)}:=\superA_1\otimes(\kg_\gR)_1$
and $\kg^{(0)}:=(\caN_\superA\cap\superA_0)\otimes(\kg_\gR)_0$.
These are subspaces of $(\caN_\superA\otimes\gB \kg)\cap \kg_0$.
Moreover, $\kg^{(0)}\oplus\kg^{(1)}$ is stable under the bracket.

A Lie polynomial is a polynomial where multiplication of variables
is given by Lie bracket.

\begin{proposition}\label{Prop:BCH}
There exists Lie polynomials $\BCH_0:\kg^{(1)}\times\kg^{(1)}\to\kg^{(0)}$
and $\BCH_1:\kg^{(1)}\times\kg^{(1)}\to\kg^{(1)}$ such that
\begin{equation}\label{BCH}
\forall X,Y\in\kg^{(1)},\quad \exp(X)\exp(Y)=
\exp(\BCH_0(X,Y))\exp(\BCH_1(X,Y)).
\end{equation}
\end{proposition}

\begin{proof}
Clearly, there exists unique polynomials $\BCH_0$ of even degree 
and $\BCH_1$ of odd degree such that Equation \eqref{BCH}
is fulfilled. It remains to prove that they are Lie polynomials.
Taking $X,Y,Z\in\kg^{(1)}$, we compute
\begin{multline*}
\left(\exp(X)\exp(Y)\right)\exp(Z)=\exp(\BCH_0(X,Y))\exp(\BCH_1(X,Y))\exp(Z)\\
=\exp(\BCH_0(X,Y))\exp(\BCH_0(\BCH_1(X,Y),Z)\exp(\BCH_1(\BCH_1(X,Y),Z)))
\end{multline*}
and 
\begin{align*}
\exp(X)\left(\exp(Y)\exp(Z)\right)=&\exp(X)\exp(\BCH_0(Y,Z))\exp(\BCH_1(Y,Z))\\
=&\exp(\Ad_{\exp(X)}\BCH_0(Y,Z))\exp(X)\exp(\BCH_1(Y,Z))\\
=&\exp(\Ad_{\exp(X)}\BCH_0(Y,Z))\exp(\BCH_0(X,\BCH_0(Y,Z))\\
&\times\exp(\BCH_1(X,\BCH_1(Y,Z))).
\end{align*}
Products of polynomials of even degree is again of even degree, thus the 
equality of the two above terms yields
\begin{equation*}
\BCH_1(\BCH_1(X,Y),Z)=\BCH_1(X,\BCH_1(Y,Z)).
\end{equation*}
By decomposing the polynomial $\BCH_1$ in terms of its degrees,
we get $\BCH_1(X,Y)=\sum_k F_{2k+1}(X,Y)$, and then
\begin{equation*}
\sum_iF_{2i+1}\Big(\sum_jF_{2j+1}(X,Y),Z\Big)=
\sum_iF_{2i+1}\Big(X,\sum_jF_{2j+1}(Y,Z)\Big).
\end{equation*}
Since $F_1(X,Y)=X+Y$ and $F_{2i+1}(\l X,\m Y)=0$ for all $i\in\gN$ and $\l,\m\in\gR$,
the proof in \cite{Eichler:1968} that the standard $\BCH$ series is a Lie series applies.
As a result, each $F_{2i+1}$ is a Lie polynomial and $\BCH_1$ is also one.
Using that the standard $\BCH$ series is a Lie polynomial on $\kg^{(1)}$, 
the equality $\exp(\BCH(X,Y))=\exp(\BCH_0(X,Y))\exp(\BCH_1(X,Y))$ 
implies that $\BCH_0$ is also a Lie polynomial.
\end{proof}
\begin{example}
The first terms of $\BCH_1$ are 
$$\BCH_1(tX,tY)=t(X+Y)+\frac{t^3}{6}\left(2[X,[X,Y]]+[[X,Y],Y]\right)+O(t^5),$$
with $t\in\gR$.
\end{example}

\begin{remark}
By Ado theorem, any Lie supergroup is locally a matrix supergroups.
For matrix supergroups, the exponential map defined from the flow
of left invariant vector field coincides with the exponential series computed
in the algebra of supermatrices (see \cite[Paragraph VI.3.10]{Tuynman:2005}).  
Hence, for any Lie supergroup with Lie superalgebra $\kg$, 
the exponential map satisfies
\begin{equation*}
\quad e^Xe^Y=e^{\BCH_0(X,Y)}e^{\BCH_1(X,Y))},
\end{equation*}
for all  $X,Y\in\kg^{(1)}$ in a suitable neighborhood of zero.
\end{remark}

\begin{remark}
Let $G$ be a Lie supergroup. According to Proposition \ref{prop-diffeom}, 
it is diffeomorphic to $G^\wod\times\kg^{(1)}$ via $(g_0,X)\mapsto g_0\, e^X$. 
By pull-back to $G^\wod\times\kg^{(1)}$, the group law of $G$ reads as 
\begin{equation}\label{GroupLaw:BCH}
(g_0,X)(g_0',X')=(g_0g_0'e^{\BCH_0(\Ad_{(g_0')^{-1}}X,X')},\BCH_1(\Ad_{(g_0')^{-1}}X,X')),
\end{equation}
for all $g_0,g_0'\in G^\wod$ and $X,X'\in\kg^{(1)}$.
\end{remark}

\begin{remark}\label{Rmk:SupergrouptoSHC}
Let $(G_0,\kg_\gR)$ be a super Harish-Chandra pair and
$G^\wod$ the Lie supergroup obtained from $G_0$ by 
$\superA_0$-scalar extension. Then, the supermanifold
$G^\wod\times\kg^{(1)}$ endowed with the product
\eqref{GroupLaw:BCH} is a Lie supergroup. This provides 
a construction of the Lie supergroup associated to a super
Harish-Chandra pair.
\end{remark}

%
%
%
%
%

\bibliographystyle{utcaps}
\bibliography{biblio-these,biblio-perso,biblio-recents}

\providecommand{\href}[2]{#2}\begingroup\raggedright\begin{thebibliography}{10}

\bibitem{Neveu:1971}
A.~Neveu and J.~H. Schwarz, ``{Factorizable dual model of pions},'' {\em Nucl.
  Phys. B.} {\bf 31} (1971)  86--112.

\bibitem{Wess:1974}
J.~Wess and B.~Zumino, ``{Supergauge transformations in four dimensions},''
  {\em Nucl. Phys.} {\bf B70} (1974)  39--50.

\bibitem{Salam:1974}
A.~Salam and J.~Strathdee, ``{Unitary representations of super-gauge
  symmetries},'' {\em Nucl. Phys. B} {\bf 80} (1974)  499--505.

\bibitem{Scheunert:1977}
M.~Scheunert, W.~Nahm, and V.~Ritteberg, ``{Graded lie algebras: Generalization
  of hermitian representations},'' {\em J. Math. Phys.} {\bf 18} (1977)
  146--154.

\bibitem{Nishiyama:1990}
K.~Nishiyama, ``{Oscillator Representations for Orthosymplectic Algebras},''
  {\em J. Algebra} {\bf 129} (1990)  231--262.

\bibitem{Sternberg:1978}
S.~Sternberg and J.~A. Wolf, ``{Hermitian Lie algebras and metaplectic
  representations I},'' {\em Trans. Amer. Math. Soc.} {\bf 238} (1978)  1--43.

\bibitem{Deligne:1999su}
P.~Deligne and J.~W. Morgan, {\em {Notes on supersymmetry (after Joseph
  Bernstein)}}.
\newblock Quantum Fields and Strings: A Course for Mathematicians, pp. 41-98,
  1999.

\bibitem{Rudolph:2000}
O.~Rudolph, ``{Super Hilbert Spaces},'' {\em Commun. Math. Phys.} {\bf 214}
  (2000)  449--467.

\bibitem{ElGradechi:1996}
A.~El~Gradechi and L.~M. Nieto, ``{Supercoherent States, Super K\"ahler
  Geometry and Geometric Quantization},'' {\em Commun. Math. Phys.} {\bf 175}
  (1996)  521--564.

\bibitem{Carmeli:2006}
C.~Carmeli, G.~Cassinelli, A.~Toigo, and V.~S. Varadarajan, ``{Unitary
  Representations of Super Lie Groups and Applications to the Classification
  and Multiplet Structure of Super Particles},'' {\em Commun. Math. Phys.} {\bf
  263} (2006)  217--258.

\bibitem{Salmasian:2010}
H.~Salmasian, ``{￼￼Unitary Representations of Nilpotent Super Lie
  Groups},'' {\em Commun. Math. Phys.} {\bf 297} (2010)  189--227.

\bibitem{Furutsu:1988}
H.~Furutsu and T.~Hirai, ``{Representations of Lie superalgebras I Extensions
  of representations of the even part},'' {\em J. Math. Kyoto Univ.} {\bf 28}
  (1988)  695--749.

\bibitem{Neeb:2011}
K.-H. Neeb and H.~Salmasian, ``{Lie Supergroups, Unitary Representations, and
  Invariant Cones},'' {\em in S. Ferrara et al. (eds) Sypersymmetry in
  Mathematics and Physics, Springer, Lecture Notes in Mathematics} {\bf 2027}
  (2011)  195--239.

\bibitem{Alldridge:2013}
A.~Alldridge, J.~Hilgert, and M.~Laubinger, ``{Harmonic Analysis on
  Heisenberg-Clifford Lie Supergroups},'' {\em J. London Math. Soc.} {\bf 87}
  (2013)  561--585.

\bibitem{Bognar:1974}
J.~Bognar, {\em {Indefinite inner product spaces}}.
\newblock Springer-Verlag, 1974.

\bibitem{Bieliavsky:2010su}
P.~Bieliavsky, A.~de~Goursac, and G.~Tuynman, ``{Deformation quantization for
  Heisenberg supergroup},''
  \href{http://dx.doi.org/10.1016/j.jfa.2012.05.002}{{\em J. Funct. Anal.} {\bf
  263} (2012)  549--603}, \href{http://arxiv.org/abs/1011.2370}{{\tt
  arXiv:1011.2370 [math.QA]}}.

\bibitem{deGoursac:2014kv}
A.~de~Goursac, ``{Fr\'echet Quantum Supergroups},''
  \href{http://dx.doi.org/10.2140/pjm.2015.273.169}{{\em Pacif. J. Math.} {\bf
  273} (2015)  169--195}, \href{http://arxiv.org/abs/1105.2420}{{\tt
  arXiv:1105.2420 [math.QA]}}.

\bibitem{deGoursac:2008bd}
A.~de~Goursac, T.~Masson, and J.-C. Wallet, ``{Noncommutative
  $\varepsilon$-graded connections},'' {\em J. Noncommut. Geom.} {\bf 6} (2012)
   343--387,
\href{http://arxiv.org/abs/0811.3567}{{\tt arXiv:0811.3567 [math-ph]}}.

\bibitem{Naimark:1968b}
M.~A. Naimark and R.~S. Ismagilov, ``{Representations of Groups and Algebras in
  Spaces with Indefinite Metric},'' {\em Mathematical Analysis} (1968)
  73--105.

\bibitem{Tuynman:2005}
G.~M. Tuynman, {\em {Supermanifolds and Supergroups}}.
\newblock Kluwer Academic Publishers, 2005.

\bibitem{Tuynman:2016}
G.~M. Tuynman, ``{Integrating Infinitesimal (Super) Actions},'' {\em J. Lie
  Theory} {\bf 26} (2016)  297--358.

\bibitem{DeWitt:1984}
B.~DeWitt, {\em {Supermanifolds}}.
\newblock Cambridge UP, 1984.

\bibitem{Berezin:1976}
F.~Berezin and D.~Leites, ``{Supermanifolds},'' {\em Soviet Maths Doklady} {\bf
  16} (1976)  1218--1222.

\bibitem{Rogers:2007}
A.~Rogers, {\em {Supermanifolds, Theory and Applications}}.
\newblock World Scientific Publishing, 2007.

\bibitem{Kostant:1977}
B.~Kostant, {\em {Graded manifolds, graded Lie theory and prequantization}}.
\newblock Lecture Notes in Mathematics 570 (Springer), 1977.

\bibitem{Koszul:1982}
J.-L. Koszul, ``{Graded manifolds and graded Lie algebras},'' {\em Proceedings
  of the international meeting on geometry and physics (Florence, 1982)} (1982)
   71–84.

\bibitem{Warner:1972}
G.~Warner, {\em {Harmonic Analysis on Semi-Simple Lie Groups I}}.
\newblock Springer-Verlag, 1972.

\bibitem{Reed:1980}
M.~Reed and B.~Simon, {\em {Methods of modern mathematical Physics, Vol 1:
  Functional Analysis}}.
\newblock Academic Press Inc., 1980.

\bibitem{Alekseevsky:1997}
D.~V. Alekseevsky and V.~Cortes, ``{Classification of $N$-(Super)-Extended
  Poincar\'e Algebras and Bilinear Invariants of the Spinor Representation of
  $Spin(p,q)$},'' {\em Commun. Math. Phys.} {\bf 183} (1997)  477--510.

\bibitem{Folland:1989}
G.~B. Folland, {\em {Harmonic analysis in phase space}}.
\newblock Princeton University Press, 1989.

\bibitem{Groechenig:2001}
K.~Groechenig, {\em Foundations of Time-Frequency Analysis}.
\newblock Birkhauser Boston Inc., Boston, MA, 2001.

\bibitem{Rempel:1983}
S.~Rempel and T.~Schmitt, ``{Pseudodifferential operators and the index theorem
  on supermanifolds},'' {\em Math. Nachr.} {\bf 111} (1983)  153--175.

\bibitem{Brackx:2005}
F.~Brackx, N.~De~Schepper, and F.~Sommen, ``{The Clifford-Fourier transform},''
  {\em J. Fourier Anal. Appl.} {\bf 11} (2005)  669–681.

\bibitem{DeBie:2008}
H.~De~Bie, ``{Fourier transform and related integral transforms in
  superspace},'' {\em J. Math. Anal. Appl.} {\bf 345} (2008)  147--164.

\bibitem{Eichler:1968}
M.~Eichler, ``{A new proof of the Baker-Campbell-Hausdorff formula},'' {\em J.
  Math. Soc. Japan} {\bf 20} (1968)  23--25.

\end{thebibliography}\endgroup

\end{document}